\documentclass[a4paper,12pt,times,preprint]{elsarticle}
\usepackage{amsmath,amssymb,amsthm,subfig,bm,units}
\usepackage[mathscr]{euscript}
\usepackage[referable]{threeparttablex}
\usepackage{fullpage}
\usepackage{grffile}
\usepackage{tikz}
\usetikzlibrary{arrows}
\usepackage{mdframed,booktabs,comment,rotating,multirow}
\usepackage[title,toc,titletoc,page]{appendix}
\usepackage[suffix=]{epstopdf}
\DeclareGraphicsExtensions{.png,.pdf}
\graphicspath{{pic/}}

\newcommand{\dd}{\mathrm d}
\newcommand\pdd[1]{\dfrac{\partial}{\partial {#1}}}
\newcommand\pd[2]{\dfrac{\partial {#1}}{\partial {#2}}}
\newcommand\od[2]{\dfrac{\mathrm{d} {#1}}{\mathrm{d} {#2}}}

\DeclareMathOperator{\tr}{tr}

\DeclareMathOperator{\sgn}{sgn}

\theoremstyle{definition}
\newtheorem{lemma}{Lemma}
\newtheorem{theorem}{Theorem}

\newtheorem{remark}{Remark}
\newcommand{\an}{\text{~~and~~}}

\title{A Robust Riemann Solver for Multiple Hydro-Elastoplastic Solid
  Mediums}

\author[pkuextra]{Ruo Li}
\ead{rli@math.pku.edu.cn}
\author[pkucoe]{Yanli Wang}
\ead{wangyanliwyl@gmail.com}
\author[pkumath,nint]{Chengbao Yao\corref{cor}}
\ead{yaocheng@pku.edu.cn}
\cortext[cor]{Corresponding author}
\address[pkuextra]{HEDPS \& CAPT, LMAM \& 
	School of Mathematical Sciences,
	Peking University, Beijing, China}
\address[pkucoe]{College of Engineering,
	Peking University, Beijing, China}
\address[pkumath]{School of Mathematical Sciences,
	Peking University, Beijing, China}
\address[nint]{Northwest Institute of Nuclear Technology, 
        Xi'an, China}

\begin{document}
\begin{abstract}
We propose a robust approximate solver for the hydro-elastoplastic solid
material, a general constitutive law extensively applied in explosion and 
high speed impact dynamics, and provide a natural transformation between the 
fluid and solid in the case of phase transitions. The hydrostatic components of 
the solid is described by a family of general Mie-Gr{\"u}neisen equation of 
state (EOS), while the deviatoric component includes the elastic phase, 
linearly hardened plastic phase and fluid phase. The approximate solver 
provides the interface stress and normal velocity by an iterative method. The 
well-posedness and convergence of our solver are proved with mild assumptions 
on the equations of state. The proposed solver is applied in computing the 
numerical flux at the phase interface for our compressible multi-medium flow 
simulation on Eulerian girds. Several numerical examples, including Riemann 
problems, shock-bubble interactions, implosions and high speed impact 
applications, are presented to validate the approximate solver.
 
\end{abstract}

\maketitle

\section{Introduction}
Significant interest has arisen in the modeling and simulation of dynamic 
events that involve high-load conditions and large deformations, such as 
shock-driven motions, high-speed impacts, implosions, and so on. The numerical 
analysis of these problems demands the implementation of very specific 
capabilities that enable the simulation of multiple mediums and their 
interactions through accurate descriptions of boundary conditions and 
high-resolution shock and wave capturing. 

There are two typical frameworks to describe the motion of multi-medium flows
\cite{Benson1992}, that is, the Lagrangian framework and the Eulerian 
framework. In the Lagrangian framework, the equations for mass, momentum and 
energy conservations are solved using a computational mesh that conforms to the 
material boundaries and moves with particles \cite{Camacho1997adaptive, 
Bessette2003modeling}, which benefits from its simplicity and natural 
description of deformation, but suffers from mesh distortion when dealing with 
large deformation problems. In Eulerian framework the mesh is fixed in space, 
which makes these methods very suitable for flows with large deformations, such 
as Udaykumar \textit{et al.} \cite{Udaykumar1997, Tran2004, Udaykumar2003, 
Sambasivan2011, Sambasivan2013, Kapahi2013}, Liu \textit{et al.} \cite{Wang2009, 
Wang2007, Wang2010, Chen2012, Chen2010, Hua2011}, Mehmandoust \textit{et al.} 
\cite{Mehmandoust2009}, Sijoy \textit{et al.} \cite{Sijoy2015}, and so on. A 
typical procedure of multi-medium interaction in Eulerian grids mainly consists 
of two steps. The first step is the interface capture, including the diffuse 
interface method (DIM) \cite{Abgrall1996, Abgrall2001, Saurel1999, Saurel2009, 
Petitpas2009, Ansari2013}, and the sharp interface method (SIM), such as the 
volume of fluid (VOF) method \cite{Scardovelli1999, Noh1976}, level set method 
\cite{Sethian2001, Sussman1994}, moment of fluid (MOF) method \cite{Ahn2007, 
Dyadechko2008, Anbarlooei2009} and front-tracking method \cite{Glimm1998, 
Tryggvason2001}. The second step is the accurate prediction of the interface 
states, which can be used to stabilize the numerical diffusion in diffuse 
interface methods, and to compute the numerical flux and interface motion in 
sharp interface methods. One common approach is to solve a multi-medium Riemann 
problem which contains the fundamentally physical and mathematical properties of 
the governing equations and plays a key role in designing the numerical flux.

The solution of a multi-medium Riemann problem depends not only on the initial
states at each side of the interface, but also on the forms of constitutive 
relations. There exist some difficulties in the cases of real materials due to 
the high nonlinearity of the equation of state and non-conservation of the 
deviatoric evolution. A variety of methods to solve the corresponding Riemann 
problems have then been proposed. For example, Yadav \cite{Yadav1982converging} 
analyzed spherical shocks in metals by employing a hydrostatic Mie-Gr{\"u}neisen 
equation of state that does not consider the effects of shear deformation. Shyue 
\cite{Shyue2001} developed a Roe's approximate Riemann solver for the 
Mie-Gr{\"u}neisen EOS with variable Gr{\"u}neisen coefficient. Arienti 
\textit{et al.} \cite{Arienti2004} applied a Roe-Glaster solver to compute the 
equations combining the Euler equations involving chemical reaction with the 
Mie-Gr{\"u}neisen EOS. Lee \textit{et al.} \cite{Lee2013} developed an exact 
Riemann solver for the Mie-Gr{\"u}neisen EOS with constant Gr{\"u}neisen 
coefficient, where the integral terms are evaluated using an iterative Romberg 
algorithm. Banks \cite{Banks2010} and Kamm \cite{Kamm2015} developed a Riemann 
solver for the convex Mie-Gr{\"u}neisen EOS by solving a nonlinear equation for 
the density increment involved in the numerical integration of rarefaction 
curves. Unlike the fluid, there may exist more than one nonlinear wave in a 
solid when it undergoes an elastoplastic deformation, which will increase the 
difficulty to obtain the exact solution of the Riemann problem. Kaboudian 
\textit{et al.} \cite{Kaboudian2014} analyzed the elastic Riemann problem in the 
Lagrangian framework, and established the corresponding Riemann solver according 
to the characteristic theory. Xiao \textit{et al.} \cite{Xiao1996} raised an 
iterative procedure to solve the Riemann problem approximately by linearizing 
the Riemann invariants. Tang \textit{et al.} \cite{Tang1999} put forward a 
nearly exact Riemann solver for the perfectly elastoplastic solid based on the 
physical observation, where the Murnagham EOS and perfectly plastic model were 
chosen for the hydrostatic pressure and deviatoric stress respectively. 
Abouziarov \textit{et al.} \cite{Abouziarov2000} and Bazhenov \textit{et al.} 
\cite{Bazhenov2002} analyzed the structures of shock waves and rarefaction waves 
in an elastoplastic material on the assumption of barotropy, without taking into 
account the internal energy equation. Cheng \textit{et al.} \cite{Chen2012, 
Chen2010} analyzed the wave structures of one-dimensional elastoplastic flows 
and developed a two-rarefaction approximate Riemann solver. Menshov \textit{et 
al.} \cite{Menshov2014} provided an analysis of the Riemann problem in a 
complete statement for the perfect plasticity on the assumption of 
one-dimensional motion and uniaxial strain. Liu \textit{et al.} \cite{Liu2008, 
Liu2011}, Feng \textit{et al.} \cite{Feng2017} and Gao \textit{et al.} 
\cite{Gao2017, Gao2018} analyzed the exact solution of the elastic-perfectly 
plastic solid with the Murnagham EOS and stiffened gas EOS, and combined it with 
the modified ghost fluid method to solve multi-medium problems. Gavirilyuk 
\textit{et al.} \cite{Gavrilyuk2008} constructed a Riemann solver for the 
linearly elastic system of the hyperbolic non-conservative models with 
transverse waves. In addition, the elastic energy was included in the total 
energy, and an extra evolution equation, on the basis of Despres \textit{et al.} 
\cite{Despres2007}, was added in order to make the elastic transformation 
reversible in the absence of shock wave.

In this paper, we propose an approximate multi-medium Riemann solver
with a family of general Mie-Gr{\"u}neisen EOS and hydro-elastoplastic
deviatoric deformation, which can provide a smooth transformation
between the fluid and solid in the case of phase transitions. The
Riemann problem together with its approximate solver in such case,
which has not been well studied in the literature yet, can be applied
in the numerical scheme developed in \cite{Guo2016} conveniently. The
study we carried out here is a further exploration of our previous
work in \cite{Lichen2018}, which is restricted on the fluid-fluid
Riemann solver with Mie-Gr{\"u}neisen EOS. Similar to the solver in
\cite{Lichen2018}, some mild conditions on the coefficients of
Mie-Gr{\"u}neisen EOS are assumed to ensure the convexity of the
equation of state, which guarantees the existence and uniqueness of
the algebraic equation derived from the Riemann problem. The algebraic
equation is derived by a detailed analysis on the structure of the
Riemann fan. Then we solve the algebraic equation by an inexact Newton
method \cite{Dembo1982}, where the function and its derivatives are
evaluated approximately since the analytical expressions are not
available. The approximate evaluations of the function and its
derivatives are quite involved since they depend on the wave structure
and the error estimate in the run time. In spite of its complexity, we
find that the convergence of the inexact Newton iteration can be
achieved, which is significant to the success of large-scale
simulations in engineering applications. To validate the proposed
approximate Riemann solver, we employ it in the computation of
multi-medium compressible flows with Mie-Gr{\"u}neisen EOS and
elastoplastic deformation. The approximate solver developed here
enhances the capacity of the numerical scheme for our multi-medium
compressible fluid flows \cite{Guo2016, Lichen2018}, and allows us to
simulate the problems with highly nonlinear fluids and elastoplastic
solids.

The rest of this paper is arranged as follows. In Section \ref{sec:rp}, a
solution strategy for the multi-medium Riemann problem with Mie-Gr{\"u}neisen
EOS and hydro-elastoplastic deviatoric deformation is presented. In Section 
\ref{sec:aps}, the procedures of our approximate Riemann solver are outlined, 
and the well-posedness and convergence are analyzed. In Section \ref{sec:sch}, 
the application of our Riemann solver in multi-medium compressible flow 
calculations is briefly introduced. In Section \ref{sec:num}, several classical 
Riemann problems and applications for shock-bubble interaction, implosion and 
high speed impact problems are carried out to validate the accuracy and 
robustness of our schemes. Finally, a short conclusion is drawn in Section 
\ref{sec:conclusion}.


\section{Multi-medium Riemann Problem}\label{sec:rp}
The one-dimensional compressible multi-medium Riemann problem, in the absence  
of heat conduction and radiation, can be written as
\begin{equation}
  \dfrac{\partial \bm U}{\partial\tau}
  +\dfrac{\partial \bm F(\bm U)}{\partial\xi}
  =\bm 0, \quad 
 \bm U (\xi,\tau=0)=
   \begin{cases}
    \bm U_l, & \xi<0, \\
    \bm U_r, & \xi>0.
  \end{cases}
\label{system:oneriemann}
\end{equation}
Here $\tau$ is time, $\xi$ is spatial coordinate. $\bm U=[\rho, \rho u, E]^\top$ 
is the vector of conservative variables, and $\bm F(\bm U)=[\rho u, \rho 
u^2-\sigma, (E-\sigma)u]^\top$ is the corresponding flux. $\rho$, $u$ and $E$
are the density, velocity and total energy respectively, and $\sigma$ is the 
normal Cauchy stress of the hydro-elastoplastic solid.

To close the governing equations \eqref{system:oneriemann}, we need an equation 
of state or constitutive law to relate the thermodynamic variables. The 
hydro-elastoplastic model is a general form of nonlinear fluid, elasticity, 
perfect elastoplasticity and linearly hardened elastoplasticity, as a mix and 
match combination of isotropic models.

In the hydro-elastoplastic model, the deformation is decomposed into the 
volumetric deformation and shear deformation, and the Cauchy stress tensor 
$\bm{\sigma}$ is also divided into the hydrostatic pressure and deviatoric 
stress tensor respectively,
\[
\bm{\sigma} = -p{\bf I} + {\bf S},
\]
where $p$ is the hydrostatic pressure, $\bf S$ is the deviatoric stress tensor,  
and $\bf I$ is the unit tensor.

The hydrostatic pressure $p$ is expressed by the Mie-Gr\"uneisen EOS, which may  
be varied independently of the deviatoric response and has the following general 
form 
\begin{equation}
p(\rho, e)=\varGamma(\rho)\rho e + h(\rho),
\label{eq:particulareos}
\end{equation}
where $e$ is the specific internal energy, $\varGamma(\rho)$ is the  
Gr{\"u}neisen coefficient, and $h(\rho)$ is a reference state associated with 
the cold contribution resulting from the interactions of atoms at rest 
\cite{Heuze2012}. For the ease of our analysis, we impose on $\varGamma(\rho)$ 
and $h(\rho)$ the following assumptions

\bigskip
\textbf{(C1)} $\varGamma'(\rho) \le 0,~(\rho\varGamma(\rho))' \ge 
0,~(\rho\varGamma(\rho))''\ge 0$;
\medskip

\textbf{(C2)} $\lim\limits_{\rho\rightarrow+\infty}\varGamma(\rho)=
  \varGamma_\infty>0,~\varGamma(\rho)\le \varGamma_{\infty}+2$;
\medskip

\textbf{(C3)} $h'(\rho)\ge 0,~h''(\rho)\ge 0$,
\bigskip
\\
similar to our previous work in \cite{Lichen2018}. A lot of equations of state  
of our interests fulfill these assumptions. Particularly, we collect some 
equations of state in Appendix A which are used in our numerical tests as 
examples.

The deviatoric stress $\bf S$ has a piecewisely complex constitutive relations, 
which is governed by means of Hooke's law in the elastic region, the linearly 
hardened plastic flow rule during the plastic region, and constant states when the 
plastic limit is violated. The {\it von} Mises criterion is adopted to determine 
whether the material is under elastic region, plastic region or fluid region, 
which can be written in terms of the deviatoric stress
\[
{\mathscr H}(\mathbf S, Y)
=\mathbf S:\mathbf S -\dfrac{2}{3}Y^2, 
\]
where $Y$ is the yield stress limit of the solid material. $Y=Y^{_\mathscr E}$
corresponds to the elastic yield stress, and $Y=Y^{_\mathscr P}$ stands for the
plastic yield stress, respectively. 

The evolution of the deviatoric stress tensor can be written in the following
piecewise expressions 

\[
\pd{\mathbf S}{t}+\bm u\cdot\nabla\mathbf S=
\left\{
\begin{array}{ll}
2\mu^{_{\mathscr E}}\left(\mathbf D-\dfrac{1}{3}\tr(\mathbf D)\mathbf I\right),
& \left|S_{\rm eff}\right| \le Y^{_{\mathscr E}}, \\
2\mu^{_{\mathscr P}}\left(\mathbf D-\dfrac{1}{3}\tr(\mathbf D)\mathbf I\right),
& Y^{_{\mathscr E}} < \left|S_{\rm eff}\right| < Y^{_{\mathscr P}}, \\
\bm 0 , & \left|S_{\rm eff}\right| = Y^{_{\mathscr P}}, \\
\end{array}
\right.
\]
where 
\[
\mathbf D=\dfrac{1}{2}\left(\pd{\bm u}{\bm x}+
\left(\pd{\bm u}{\bm x}\right)^\top\right)
\]
is the rate of the deformation tensor, 
$S_{\rm eff}=\sqrt{\frac{3}{2}\mathbf S:\mathbf S}$ is the effective stress,
and $\mu^{_{\mathscr E}}$ and $\mu^{_{\mathscr P}}$ are the elastic and plastic
shear modulus, respectively. 

Utilizing the continuity equation, we can obtain the following balance law

\[
\pd{\rho S_{ij}}{t}+\pd{\rho u_k S_{ij}}{x_k}=
\left\{
\begin{array}{ll}
\beta^{_{\mathscr E}}\left(\pd{u_i}{x_j}+\pd{u_j}{x_i}\right)
-\dfrac{2}{3}\beta^{_{\mathscr E}}\delta_{ij}\pd{u_k}{x_k}, &
\left|S_{\rm eff}\right| \le Y^{_{\mathscr E}}, \\
\beta^{_{\mathscr P}}\left(\pd{u_i}{x_j}+\pd{u_j}{x_i}\right)
-\dfrac{2}{3}\beta^{_{\mathscr P}}\delta_{ij}\pd{u_k}{x_k}, &
 Y^{_{\mathscr E}}<\left|S_{\rm eff}\right| \le Y^{_{\mathscr P}}, \\
0, & \left|S_{\rm eff}\right| = Y^{_{\mathscr P}},
\end{array}
\right.
\]
where $\beta^{_{\mathscr E}}=\rho\mu^{_{\mathscr E}}$, 
$\beta^{_{\mathscr P}}=\rho\mu^{_{\mathscr P}}$, $\delta_{ij}$ is the Dirac 
function.

\begin{remark}
The hydro-elastoplastic model can degenerate to the elastic model, perfectly 
elastoplastic model, linearly hardened elastoplastic model and fluid model 
naturally. Fox example, it will degenerate to the elastic model when
$\mu^{_{\mathscr E}}=\mu^{_{\mathscr P}}$ and  
$Y^{_{\mathscr E}}=Y^{_{\mathscr P}}=\infty$, to the perfectly elastoplastic
model when $\mu^{_{\mathscr P}}=0$ and $Y^{_{\mathscr P}}=\infty$, to the 
linearly hardened elastoplasticity when 
$\mu^{_{\mathscr P}}<\mu^{_{\mathscr E}}$ and 
$Y^{_{\mathscr P}}=\infty$, and to the fluid model when 
$\mu^{_{\mathscr P}}=\mu^{_{\mathscr E}}=0$ and 
$Y^{_{\mathscr E}}=Y^{_{\mathscr P}}=0$, respectively.
\end{remark}

The model presented above is a conventional Eulerian non-conservative model for 
the elastoplastic behavior, which couples the nonlinear Euler equations of 
compressible fluids with the augmented elastoplastic deformation. In high-rate 
and large deformation region, the volumeric deformation is dominant and the 
deviatoric deformation can be neglected. When the load is removed, the 
elastoplastic effect should be taken into account again. For elastic deviatoric 
response the shear moduli may be taken to be functions of temperature and 
pressure. Plasticity is based on an additive decomposition of the rate of 
deformation tensor into elastic and plastic parts \cite{Trangenstein1991}.

Here we discuss the multi-medium Riemann problem between hydro-elastoplastic 
models, which can be treated in a similar way as the single-medium Riemann 
problem as long as the materials remain immiscible. The Riemann solution
consists of several constant regions separated by the phase interface and
genuinely nonlinear waves. The key of the Riemann problem is to compute the
states in the region adjacent to the phase interface (the so called star
region). To understand the influence of material deformation on the interface
states, the solution structure in each medium should be analyzed with
consideration of the elastoplastic deformation. Without loss of generality, we
take the medium at the right side of the interface as the example, and the left
side can be analyzed in a similar manner.

\subsection{Solution in the elastic phase}
The Riemann solution in the elastic phase consists of two constant states
separated by an elastic acoustic wave, whose speed is given by 
\[
\lambda = u  +\sqrt{c^2+\dfrac{4\mu^{_\mathscr E}}{3\rho}}.
\]
\begin{figure}[htb]
 \begin{minipage}[b]{.49\textwidth}
 \begin{tikzpicture}[scale=.65]
  \tikzset{font={\fontsize{9pt}{12}\selectfont}}
  \fill [lightgray,opacity=.5] (0,0)--(4.4,0) arc(0:110:4.4);
  \draw[very thick,-triangle 45](-5.5,0)--(5.5,0)
   node[right]{\normalsize $\xi$};
  \draw[very thick,-triangle 45] (0,0)--(0,5.5)
node[above] {\normalsize $\tau$};
  \draw (2.7,0.3) node{right initial state $\bm U_r$};
  \draw (0.4,1.8) node{$\bm U_r^*$};
  \draw (3.5,1.2) node{\bf elastic phase};
  \draw (0,0)--(45:5) node[above]{\it right-facing elastic wave};
  \draw[thick] (0,0)--(110:4.8) node[above]{\it interface};
  \end{tikzpicture}
    \caption{Wave structure of the elastic solid phase in the $\xi-\tau$
space.}
  \label{fig:Riemann-elastic}
 \end{minipage}
  \begin{minipage}[b]{.49\textwidth}
  \begin{tikzpicture}[scale=.65]
  \tikzset{font={\fontsize{9pt}{12}\selectfont}}
  \fill [lightgray,opacity=.5] (0,0)--(4.4,0) arc(0:110:4.4);
  \draw[very thick,-triangle 45](-5.5,0)--(5.5,0)node[right]
  {\normalsize $\xi$};
  \draw[very thick,-triangle 45] (0,0)--(0,5.5) node[above]
  {\normalsize $\tau$};
  \draw (2.7,0.3) node{right initial state $\bm U_r$};
  \draw (-.3,1.8) node{$\bm U_r^*$};
  \draw (1.5,1.5) node{$\bm U_r^{\Delta}$};  
  \draw (4.2,1.2) node{\bf elastoplastic phase};
  \draw (0,0)--(30:5) node[above]{\it right-facing elastic wave};
  \draw (0,0)--(60:5) node[above]{\it right-facing plastic wave};
  \draw[thick] (0,0)--(110:4.8) node[above]{\it interface};
 \end{tikzpicture}
   \caption{Wave structure of the elastoplastic solid phase in the $\xi-\tau$ 
space.}
  \label{fig:Riemann-elastoplastic}
 \end{minipage}
\end{figure}

A typical wave structure for the elastic Riemann problem is shown in Fig. 
\ref{fig:Riemann-elastic}. The acoustic wave is genuinely nonlinear, while
the contact wave is linearly degenerate \cite{Gavrilyuk2008}. The jump of the
normal deviatoric stress $S$ across the acoustic wave satisfies
\begin{equation}
S_k^* -S_k=\dfrac{4\beta^{_\mathscr E}_k}{3}
\left(\dfrac{1}{\rho_k^*}-\dfrac{1}{\rho_k}\right),
\label{eq:devjump}
\end{equation}
where the superscript ``*'' stands for the star region state.

\begin{itemize}
\item[-] {\bf Rarefaction wave}

Denote by $q=p-S$ the negative normal component of Cauchy stress tensor on
the interface. If $q_k^* \le q_k$, the acoustic wave is a rarefaction wave. It
can be found that 
\[
u-\int \dfrac{1}{\rho}\sqrt{c^2+\dfrac{4\beta^{_\mathscr E}}{3\rho^2}}
\dd\rho,\quad p-\int c^2\dd\rho,
\]
are Riemann invariants, which yield the relation
\begin{align*}
u^*_k-u_k&=\displaystyle\int_{q_k}^{q^*_k}
  \left(\rho^2c^2+\dfrac{4}{3}\beta^{_\mathscr E}_k\right)^{1/2}\dd q, \\
\rho_k^*-\rho_k&=\displaystyle\int_{p_k}^{p_k^*} \dfrac{\dd p}{c^2}.
\end{align*}

\item[-] {\bf Shock wave}

If $q_k^*>q_k$, then the acoustic wave is a shock wave. Applying the analysis of
non-conservative product \cite{Gavrilyuk2008} we have
\[
\begin{array}{c}
   u^*_k-u_k=\left(\dfrac1{\rho^*_k}-\dfrac1{\rho_k}\right)
   \left(-\dfrac{q^*_k-q_k}
   {1/\rho^*_k-1/\rho_k}\right)^{1/2},\\
   e_k(p^*_k,\rho^*_k)-e_k(p_k,\rho_k)
  +\dfrac{1}{2}(p^*_k+p_k)\left(\dfrac{1}{\rho^*_k}
  -\dfrac{1}{\rho_k}\right)=0.
\end{array}
\]

We define $\varphi^{_\mathscr E}_k(p,\rho)$ to relate $p$ and $\rho$ on the
elastic Hugoniot locus, 
\[
\varphi^{_\mathscr E}_k(p,\rho):
=\varGamma_k(\rho_k)\rho_k (p-h_k(\rho)) -
\varGamma_k(\rho) \rho (p_k-h_k(\rho_k)) 
-\dfrac{1}{2}\varGamma_k(\rho_k)(p+p_k)\varGamma_k(\rho)
(\rho-\rho_k),
\]
and $\varphi^{_\mathscr{SE}}_k(S,\rho)$ to relate $S$ and $\rho$ on
the elastic Hugoniot locus, according to \eqref{eq:devjump}, 
\[
\varphi^{_\mathscr{SE}}_k(S,\rho):=
  \left(\varGamma_k(\rho_k)\rho_k
 -\dfrac{1}{2}\varGamma_k(\rho_k)\varGamma_k(\rho)(\rho-\rho_k)\right)
  \left(S-S_k-\dfrac{4\beta^{_\mathscr{E}}_k}{3}
  \left(\dfrac{1}{\rho}-\dfrac{1}{\rho_k}\right)\right)=0.
\]
Define
\begin{equation}
\begin{aligned}
\varPhi^{_\mathscr{E}}_k(q,\rho):&=
\varphi^{_\mathscr{E}}_k(p,\rho)-\varphi^{_\mathscr{SE}}_k(S,\rho)\\
&=\varGamma_k(\rho_k)\rho_k\left(q+S_k^{_\mathscr E}-h_k(\rho)\right)
-\varGamma_k(\rho) \rho (p_k-h_k(\rho_k)) \\
&-\dfrac{1}{2}\varGamma_k(\rho_k)(q+S_k^{_\mathscr E}+p_k)
\varGamma_k(\rho)(\rho-\rho_k),
\label{eq:phi_se}
\end{aligned}
\end{equation}
where $S_k^{_\mathscr E}=
S_k+\dfrac{4}{3}\left(\dfrac{1}{\rho}-\dfrac{1}{\rho_k}\right)$.

We have the following results on the function
$\varPhi^{_\mathscr{E}}_k(q,\rho)$.
\begin{lemma}\label{thm:phie}
The Hugoniot function $\varPhi^{_\mathscr{E}}_k(q,\rho)$ defined in
\eqref{eq:phi_se} satisfies the following properties: 
1). $\varPhi^{_\mathscr{E}}_k(q,\rho_k)>0$;~
2). $\varPhi^{_\mathscr{E}}_k(q,\rho_{\max})<0$;~
3). ${\partial\varPhi^{_\mathscr{E}}_k}(q,\rho)/{\partial \rho}<0$;~
4). ${\partial^2\varPhi^{_\mathscr{E}}_k}(q,\rho)/{\partial \rho^2}<0$ if 
$h''_k(\rho)\ge (8+2\varGamma_k(\rho))\beta^{_\mathscr E}_k/3\rho^3$.
\end{lemma}

\begin{proof}
(1). 1), 2) are obvious results from our previous work in \cite{Lichen2018}.
	
(2). The first derivative of $\varPhi^{_\mathscr{E}}_k(q,\rho)$ in the elastic
region with respect to the density is
\begin{align*}
\pd{\varPhi_k^{_\mathscr E}}{\rho}(q,\rho) &= 
\pd{\varphi_k^{_\mathscr E}}{\rho}(p,\rho)-
\pd{\varphi^{_\mathcal S}_k}{\rho}(S,\rho) \\
&= \pd{\varphi_k^{_\mathscr E}}{\rho}(p,\rho)-
\varGamma_k(\rho_k)\left(2\rho_k-
\varGamma_k(\rho)(\rho-\rho_k)\right)
\dfrac{2\beta^{_\mathscr E}_k}{3\rho^2}.
\end{align*}
Since $\pd{\varphi_k^{_\mathscr E}}{\rho}(p,\rho)<0$, and
$2\rho_k-\varGamma_k(\rho)(\rho-\rho_k)>0$, we can conclude that
\[
\pd{\varPhi_k^{_\mathscr E}}{\rho}(q,\rho) < 0.
\]

(3). The second derivative of $\varPhi^{_\mathscr{E}}_k(q,\rho)$ with respect to
the density is
\begin{align*}
\pd{^2\varPhi_k^{_\mathscr E}}{\rho^2}(q,\rho) &= 
\pd{^2\varphi_k^{_\mathscr E}}{\rho^2}(p,\rho)-
\pd{^2\varphi^{_\mathcal S}_k}{\rho^2}(S,\rho) \\
&= \pd{^2\varphi_k^{_\mathscr E}}{\rho^2}(p,\rho)
+\dfrac{2\beta^{_\mathscr E}_k}{3\rho^3}\varGamma_k(\rho_k)
\left(4\rho_k-\rho\varGamma_k(\rho)+2\rho_k\varGamma_k(\rho)\right)\\
&< \pd{^2\varphi_k^{_\mathscr E}}{\rho^2}(p,\rho)
+\dfrac{(8+2\varGamma_k(\rho))\varGamma_k(\rho_k)
\beta^{_\mathscr E}_k\rho_k}{3\rho^3}.
\end{align*}
It is an obvious result that 
$\pd{^2\varPhi_k^{_\mathscr E}}{\rho^2}(q,\rho)<0$ when
$\varGamma_k''(\rho)=0$ and 
$h''_k(\rho)\ge (8+2\varGamma_k(\rho))\beta^{_\mathscr E}_k/3\rho^3$.

This completes the whole proof. 
\end{proof}

The slope of the Hugoniot locus in the elastic solid phase can be found by the
method of implicit differentiation, namely,
\[
\chi_k^{_\mathscr E}(q,\rho) := \left.\pd{q}{\rho}
\right|_{\varPhi_k^{_\mathscr E}}=
-\dfrac{2\partial\varPhi^{_\mathscr E}_k(q,\rho)/\partial\rho}{
  \varGamma_k(\rho_k)(2\rho_k-{\varGamma_k(\rho)}
  (\rho-\rho_k))}>0.
\]
\end{itemize}

\begin{remark}
The wave structure in the linearly hardened region can be treated as a similar
case as the elastic region. And the wave structure in the fluid region will
can be viewed as $\beta_k^{_\mathscr P}=\mu_k^{_\mathscr P}=0$.
\end{remark}

\subsection{Solution in the elastic-plastic phase}
When the solid undergoes an elastoplastic phase transition, the constitutive
model is distinguished by the elastic limit. Due to the discrepancy of the
elastic and plastic wave, there exists a jump in the slope of the rarefaction
curve or Hugoniot locus, which leads to the occurrence of \emph{split wave}.
Since the elastic wave propagates faster than the plastic wave, the acoustic
wave structure, shown in Fig. \ref{fig:Riemann-elastoplastic}, will include a 
\emph{leading elastic wave} which connects the initial state $\bm U_k$ to the 
elastic limit state $\bm U_k^\Delta$, and a \emph{trailing plastic wave} which 
connects the elastic limit state to the star region state $\bm U^*_k$, where the 
superscript ``$\Delta$'' denotes the state at the elastic limit.

\begin{itemize}
\item[-] {\bf Elastic limit state}

Before we discuss the elastoplastic flow, let us introduce the solid
densities at the elastic limit of compression $\rho_{_\mathscr C}$ and tension
$\rho_{_\mathscr T}$ respectively, such that the effective stress 
$\sqrt{\mathbf S:\mathbf S}$ reaches the elastic yield stress limit
$\sqrt{2/3}Y^{_\mathscr E}$.

According to the jump conditions of the deviatoric stress across the acoustic
wave \eqref{eq:devjump}, we can get the corresponding effective stress after the
elastic acoustic wave,
\[
\mathbf S_k^{\Delta}:\mathbf S_k^{\Delta}
=\dfrac{8}{3}\left(\beta^{_\mathscr E}_k\right)^2
\left(\dfrac{1}{\rho_k^\Delta}-\dfrac{1}{\rho_k}\right)^2
+ 4\beta^{_\mathscr E}_k\left(\dfrac{1}{\rho_k^\Delta}
-\dfrac{1}{\rho_k}\right)S_k
+ \mathbf S_k:\mathbf S_k.
\]
where $\mathbf S_k$ is the deviatoric stress tensor in the normal direction of 
the phase interface. Setting $\mathbf S_k^{\Delta}:\mathbf 
S_k^{\Delta}=2\left(Y^{_\mathscr
E}\right)^2/3$ yields the definition of $\rho_{_\mathscr C}$ and
$\rho_{_\mathscr T}$

\begin{equation}
\begin{aligned}
\rho_{_\mathscr C}
&=\left(\dfrac{1}{\rho_k}-\dfrac{3}{4\beta^{_\mathscr E}_k}S_k
 - \dfrac{3}{4\beta^{_\mathscr E}_k}\sqrt{S_k^2+
 \dfrac 49\left(Y^{_\mathscr E}\right)^2
-\dfrac 23\mathbf S_k:\mathbf S_k}\right)^{-1}, \\
\rho_{_\mathscr T}
&=\left(\dfrac{1}{\rho_k}-\dfrac{3}{4\beta^{_\mathscr E}_k}S_k
 + \dfrac{3}{4\beta^{_\mathscr E}_k}
 \sqrt{S_k^2+\dfrac 49\left(Y^{_\mathscr E}\right)^2
 -\dfrac 23\mathbf S_k:\mathbf S_k}\right)^{-1}.
\end{aligned}
\label{eq:elastic_rho}
\end{equation}

Note that for most applications the elastic yield stress limit $Y^{_\mathscr 
E}$ is much smaller than the elastic shear modulus $\mu^{_\mathscr E}$ (about 
$2\sim3$ orders of magnitude smaller). Therefore, both $\rho_{_\mathscr C}$ and
$\rho_{_\mathscr T}$ must be positive. The other relevant quantities can also 
be calculated
\[
p_{_\mathscr C}=\dfrac{2\varGamma_k(\rho_k)\rho_kh_k(\rho_{_\mathscr C})+
2\varGamma_k(\rho_{_\mathscr C})\rho_{_\mathscr C}(p_k-h_k(\rho_k))
+\varGamma_k(\rho_k)\varGamma_k(\rho_{_\mathscr C})
p_k(\rho_{_\mathscr C}-\rho_k)}	
{\varGamma_k(\rho_k)\left(
(2+\varGamma_k(\rho_{_\mathscr C}))\rho_k
-\varGamma_k(\rho_{_\mathscr C})\rho_{_\mathscr C}\right)},
\]
\[
p_{_\mathscr T}=\int_{\rho_k}^{\rho_{_\mathscr T}} c^2 \dd\rho,
\]
\[
S_{_\mathscr C}=S_k+\dfrac{4\beta_k^{_\mathscr E}}{3}
\left(\dfrac{1}{\rho_{_\mathscr C}}-\dfrac{1}{\rho_k}\right),
\]
\[
S_{_\mathscr T}=S_k+\dfrac{4\beta_k^{_\mathscr E}}{3}
\left(\dfrac{1}{\rho_{_\mathscr T}}-\dfrac{1}{\rho_k}\right), 
\]
\[
q_{_\mathscr T}=p_{_\mathscr T}-S_{_\mathscr T}, \quad
q_{_\mathscr C}=p_{_\mathscr C}-S_{_\mathscr C},
\]
where $p_{_\mathscr C}, S_{_\mathscr C}, q_{_\mathscr C}$ and $p_{_\mathscr T},
S_{_\mathscr T}, q_{_\mathscr T}$ are the hydrostatic pressure, the normal
component of deviatoric stress tensor and negative Cauchy stress tensor at the
elastic limit of compression and tension, respectively.

\item[-] {\bf Elastoplastic rarefaction wave}

If $q_k^* \le q_{_\mathscr T} \le q_k$, the acoustic elastic wave and plastic
wave are both rarefaction waves,
\[
\begin{cases}
u^*_k-u_k=\displaystyle\int_{q_k}^{q_{_\mathscr T}}
\left(\rho^2c^2+\dfrac{4\beta^{_\mathscr E}_k}{3}\right)^{-1/2}\dd q +
\displaystyle\int_{q_{_\mathscr T}}^{q_k^*}
\left(\rho^2c^2+\dfrac{4\beta^{_\mathscr P}_k}{3}\right)^{-1/2}\dd q, \\
\rho_k^*-\rho_k=\displaystyle\int_{q_k}^{q_{_\mathscr T}}
\left(c^2+\dfrac{4\beta^{_\mathscr E}_k}{3\rho^2}\right)^{-1}\dd q
+\displaystyle\int_{q_{_\mathscr T}}^{q_k^*}
\left(c^2+\dfrac{4\beta^{_\mathscr P}_k}{3\rho^2}\right)^{-1}\dd q.
\end{cases}
\]

\item[-] {\bf Elastoplastic shock wave}

If $q_k^*>q_{_\mathscr C}>q_k$, then the acoustic elastic wave and plastic wave
are both shock waves,
\[
\begin{cases}
  u^*_k-u_k=\left(-(q_{_\mathscr C}-q_k)\left(\dfrac{1}{\rho_{_\mathscr C}}
          -\dfrac{1}{\rho_k}\right)\right)^{1/2} +
   \left(-(q_k^*-q_{_\mathscr C})\left(\dfrac{1}{\rho}
          -\dfrac{1}{\rho_{_\mathscr C}}\right)\right)^{1/2},\\
   e_k(p^*_k,\rho^*_k)-e_k(p_k,\rho_k)
  +\dfrac{1}{2}(p_k^*+p_{_\mathscr C})\left(\dfrac{1}{\rho^*_k}
  -\dfrac{1}{\rho_{_\mathscr C}}\right)
   +\dfrac{1}{2}(p_{_\mathscr C}+p_k)\left(\dfrac{1}{\rho_{_\mathscr C}}
  -\dfrac{1}{\rho_k}\right)=0.
\end{cases}
\]

Similar to the elastic solid phase, we define $\varPhi^{_\mathscr 
{EP}}_k(q,\rho)$ to relate $q$ and $\rho$ on the elastoplastic Hugoniot locus, 
\begin{equation}
\varPhi^{_\mathscr {EP}}_k(q,\rho):=
\varPhi^{_\mathscr {E}}_k(q_{_\mathscr C},\rho_{_\mathscr C})+
\varPhi^{_\mathscr {P}}_k(q,\rho), \quad q>q_{_\mathscr C}>q_k,
\label{eq:phi_sp}
\end{equation}
where 
\begin{align*}
\varPhi^{_\mathscr E}_k(q_{_\mathscr C},\rho_{_\mathscr C}):&=
\varGamma_k(\rho_k)\rho_k\left(q_{_\mathscr C}+S_{_\mathscr{C}}-
h_k(\rho_{_\mathscr C})\right)
-\varGamma_k(\rho_{_\mathscr C}) \rho_{_\mathscr C}(p_k-h_k(\rho_k)) \\
&-\dfrac{1}{2}\varGamma_k(\rho_k)
(q_{_\mathscr C}+S_{_\mathscr{C}}+p_k)
\varGamma_k(\rho_{_\mathscr C})(\rho_{_\mathscr C}-\rho_k),  \\
\varPhi^{_{\mathscr {P}}}_k(q,\rho) :&=
\varGamma_k(\rho_{_\mathscr C})\rho_{_\mathscr C}
\left(q+S_k^{_\mathscr P}-h_k(\rho)\right)
-\varGamma_k(\rho) \rho (p_{_\mathscr C}-h_k(\rho_{_\mathscr C})) \\
&-\dfrac{1}{2}\varGamma_k(\rho_{_\mathscr C})
(q+S_k^{_\mathscr P}+p_{_\mathscr C})
\varGamma_k(\rho)(\rho-\rho_{_\mathscr C}), \\
S_k^{_\mathscr {P}} &= S_{_\mathscr C} +\dfrac{4}{3}\beta_k^{_\mathscr P}
 \left(\dfrac{1}{\rho}-\dfrac{1}{\rho_{_\mathscr C}}\right).
\end{align*}

We have the following results on the function 
$\varPhi^{_\mathscr {EP}}_k(q,\rho)$ by a similar calculus to the elastic  
case in Lemma \ref{thm:phie}.
\begin{lemma}\label{thm:phiep}
The Hugoniot function $\varPhi^{_\mathscr {EP}}_k(q,\rho)$ defined in
\eqref{eq:phi_sp} satisfies the following properties: 
1). $\varPhi^{_\mathscr {EP}}_k(q,\rho_k)>0$;~
2). $\varPhi^{_\mathscr {EP}}_k(q,\rho_{\max})<0$;~
3). ${\partial\varPhi^{_\mathscr {EP}}_k}(q,\rho)/{\partial \rho}<0$;~
4). ${\partial^2\varPhi^{_\mathscr {EP}}_k}(q,\rho)/{\partial \rho^2}<0$ if 
$h''_k(\rho)\ge (8+2\varGamma_k(\rho))\beta^{_\mathscr P}_k/3\rho^3$.
\end{lemma}

Similar to the elastic solid phase, the slope of the Hugoniot locus in the
elastoplastic solid phase can be found by the method of implicit
differentiation, namely,
\[
\chi_k^{_\mathscr {EP}}(q,\rho) := \left.\pd{q}{\rho}
\right|_{\varPhi_k^{_\mathscr {EP}}}=
-\dfrac{2\partial\varPhi^{_\mathscr {EP}}_k(q,\rho)/\partial\rho}{
  \varGamma_k(\rho_k)(2\rho_k-{\varGamma_k(\rho)}
  (\rho-\rho_k))}>0.
\]
\end{itemize}

\subsection{Solution in the plastic-fluid phase}
Similar to the elastoplastic phase, the constitutive model is distinguished 
by the plastic limit when the solid undergoes a plastic-fluid phase transition. 
The discrepancy of the plastic and fluid wave leads to the split of plastic and 
fluid wave, shown in Fig. \ref{fig:Riemann-plastofluid}, which includes a 
\emph{leading plastic wave} which connects the initial state $\bm U_k$ to the 
plastic limit state $\bm U_k^\nabla$, and a \emph{trailing fluid wave} which 
connects the plastic limit state to the star region state $\bm U^*_k$, where the 
superscript ``$\nabla$'' denotes the state at the plastic limit.

\begin{itemize}
\item[-] {\bf Plastic limit state}

The solid densities at the plastic limit of compression $\rho_{_\mathscr {PC}}$ 
and tension $\rho_{_\mathscr {PT}}$ can be calculated when the effective stress 
$\sqrt{\mathbf S:\mathbf S}$ reaches the plastic  stress limit 
$\sqrt{2/3}Y^{_\mathscr P}$, similar to Eq. \eqref{eq:elastic_rho}
\[
\begin{array}{c}
\rho_{_\mathscr{PC}}
 =\left(\dfrac{1}{\rho_k}-\dfrac{3}
 {4\beta_k^{_\mathscr P}}S_k
 - \dfrac{3}{4\beta_k^{_\mathscr P}}\sqrt{(S_k)^2+
 \dfrac 49 (Y^{^\mathscr P})^2
 -\dfrac 23\mathbf S_k:\mathbf S_k}\right)^{-1}, \\
\rho_{_\mathscr{PT}}
=\left(\dfrac{1}{\rho_k}-\dfrac{3}
 {4\beta_k^{_\mathscr P}}S_k 
 + \dfrac{3}{4\beta_k^{_\mathscr P}}\sqrt{(S_k)^2+
   \dfrac 49(Y^{^\mathscr P})^2
   -\dfrac 23\mathbf S_k:\mathbf S_k}\right)^{-1}. 
\end{array}
\]

Then the other relevant quantities can also be calculated
\[
\begin{array}{c}
p_{_\mathscr{PC}}=\dfrac{2\varGamma_k(\rho_k)
\rho_k h_k(\rho_{_\mathscr {PC}})+
2\varGamma_k(\rho_{_\mathscr {PC}})\rho_{_\mathscr {PC}}
(p_k-h_k(\rho_k))
+\varGamma_k(\rho_k)\varGamma_k(\rho_{_\mathscr{PC}})
  p_k(\rho_{_\mathscr{PC}}-\rho_k)}	
{\varGamma_k(\rho_{_\mathscr {C}})\left(
(2+\varGamma_k(\rho_{_\mathscr {PC}}))\rho_k
-\varGamma_k(\rho_{_\mathscr {PC}})\rho_{_\mathscr {PC}}\right)},  
\end{array}
\]

\[
\begin{array}{c}
p_{_\mathscr {PT}}=\int_{\rho_k}^{\rho_{_\mathscr {PT}}} 
c^2 \dd\rho, 
\end{array}
\]

\[
\begin{array}{c}
S_{_\mathscr {PC}}=S_k+\dfrac{4\beta_k^{_\mathscr{P}}}{3}
\left(\dfrac{1}{\rho_{_\mathscr {PC}}}-\dfrac{1}{\rho_k}\right),\\
\vspace{3mm}
S_{_\mathscr {PT}}=S_k+\dfrac{4\beta_k^{_\mathscr {P}}}{3}
\left(\dfrac{1}{\rho_{_\mathscr {PT}}}-\dfrac{1}{\rho_k}\right),\\
\vspace{2mm}
q_{_\mathscr {PC}}=p_{_\mathscr {PC}}-S_{_\mathscr {PC}}, \quad 
q_{_\mathscr {PT}}=p_{_\mathscr {PT}}-S_{_\mathscr {PT}}.
\end{array}
\]

\item[-] {\bf Plastic-fluid rarefaction wave}

If $q_k^* \le q_{_\mathscr {PT}} \le q_k \le q_{_\mathscr {T}} $, the acoustic 
plastic wave and fluid wave are both rarefaction waves,
\[
\begin{cases}
u^*_k-u_k=\displaystyle\int_{q_k}^{q_{_\mathscr {PT}}}
\left(\rho^2c^2+\dfrac{4\beta^{_\mathscr P}_k}{3}\right)^{-1/2}\dd q +
\displaystyle\int_{q_{_\mathscr {PT}}}^{q_k^*}\dfrac{1}{\rho c}\dd q, \\
\rho_k^*-\rho_k=\displaystyle\int_{q_k}^{q_{_\mathscr {PT}}}
\left(c^2+\dfrac{4\beta^{_\mathscr P}_k}{3\rho^2}\right)^{-1}\dd q
+\displaystyle\int_{q_{_\mathscr {PT}}}^{q_k^*} \dfrac{1}{c^2}\dd q.
\end{cases}
\]

\item[-] {\bf Plastic-fluid shock wave}

If $q_k^*>q_{_\mathscr {PC}}>q_k>q_{_\mathscr {C}} $, then the acoustic plastic 
wave and fluid wave are both shock waves,
\[
\begin{cases}
  u^*_k-u_k=\left(-(q_{_\mathscr {PC}}-q_k)
  \left(\dfrac{1}{\rho_{_\mathscr {PC}}}-\dfrac{1}{\rho_k}\right)\right)^{1/2}+
  \left(-(q_k^*-q_{_\mathscr {PC}})\left(\dfrac{1}{\rho}
  -\dfrac{1}{\rho_{_\mathscr {PC}}}\right)\right)^{1/2},\\
   e_k(p^*_k,\rho^*_k)-e_k(p_k,\rho_k)
  +\dfrac{1}{2}(p_k^*+p_{_\mathscr {PC}})\left(\dfrac{1}{\rho^*_k}
  -\dfrac{1}{\rho_{_\mathscr {PC}}}\right)
  +\dfrac{1}{2}(p_{_\mathscr {PC}}+p_k)
  \left(\dfrac{1}{\rho_{_\mathscr {PC}}} 
  -\dfrac{1}{\rho_k}\right)=0.
\end{cases}
\]

Similar to the elastoplastic solid phase, we define 
$\varPhi^{_\mathscr {PF}}_k(q,\rho)$ to relate $q$ and $\rho$ on the 
plastic-fluid Hugoniot locus, 
\begin{equation}
\varPhi^{_\mathscr {PF}}_k(q,\rho):=
\varPhi^{_\mathscr {P}}_k(q_{_\mathscr {PC}},\rho_{_\mathscr {PC}})+
\varPhi^{_\mathscr {F}}_k(q,\rho), \quad 
q>q_{_\mathscr {PC}}>q_k>q_{_\mathscr {C}},
\label{eq:phi_pf}
\end{equation}
where
\begin{align*}
\varPhi^{_\mathscr{P}}_k(q_{_\mathscr{PC}},\rho_{_\mathscr{PC}}):&=
\varGamma_k(\rho_k)\rho_k \left(q_{_\mathscr{PC}}
+S_{_\mathscr{PC}}-h_k(\rho_{_\mathscr{PC}})\right)
-\varGamma_k(\rho_{_\mathscr{PC}}) \rho_{_\mathscr{PC}}(p_k-h_k(\rho_k)) \\
&-\dfrac{1}{2}\varGamma_k(\rho_k)(q_{_\mathscr{PC}}+S_{_\mathscr{PC}}+p_k)
\varGamma_k(\rho_{_\mathscr{PC}})(\rho_{_\mathscr{PC}}-\rho_k),  \\ 
\varPhi^{_\mathscr{F}}_k(q,\rho) &=
\varGamma_k(\rho_{_\mathscr{PC}})\rho_{_\mathscr{PC}}
\left(q+S_{_\mathscr{F}}-h_k(\rho)\right)
-\varGamma_k(\rho) \rho 
(p_{_\mathscr{PC}}-h_k(\rho_{_\mathscr{PC}}))  \\
&-\dfrac{1}{2}\varGamma_k(\rho_{_\mathscr{PC}})(q
+S_{_\mathscr{F}}+p_{_\mathscr{PC}})
\varGamma_k(\rho)(\rho-\rho_{_\mathscr{PC}}), \\
S_{_\mathscr{F}} &= S_{_\mathscr{PC}}.
\end{align*}

Similarly, we have the following results on the function 
$\varPhi^{_\mathscr {PF}}_k(q,\rho)$ by a simple calculus
\begin{lemma}\label{thm:phipf}
The Hugoniot function $\varPhi^{_\mathscr {PF}}_k(q,\rho)$ defined in
\eqref{eq:phi_pf} satisfies the following properties: 
1). $\varPhi^{_\mathscr {PF}}_k(q,\rho_k)>0$;~
2). $\varPhi^{_\mathscr {PF}}_k(q,\rho_{\max})<0$;~
3). ${\partial\varPhi^{_\mathscr {PF}}_k}(q,\rho)/{\partial \rho}<0$;~
4). ${\partial^2\varPhi^{_\mathscr {PF}}_k}(q,\rho)/{\partial \rho^2}<0$ if 
$\varGamma''_k(\rho)=0$.
\end{lemma}
\end{itemize}

\begin{figure}[htb]
 \begin{minipage}[b]{.49\textwidth}
 \begin{tikzpicture}[scale=.65]
  \tikzset{font={\fontsize{9pt}{12}\selectfont}}
  \fill [lightgray,opacity=.5] (0,0)--(4.4,0) arc(0:110:4.4);
  \draw[very thick,-triangle 45](-5.5,0)--(5.5,0)node[right]
  {\normalsize $\xi$};
  \draw[very thick,-triangle 45] (0,0)--(0,5.5) node[above]
  {\normalsize $\tau$};
  \draw (2.7,0.3) node{right initial state $\bm U_r$};
  \draw (-.3,1.8) node{$\bm U_r^*$};
  \draw (1.5,1.5) node{$\bm U_r^{\nabla}$};  
  \draw (4.2,1.2) node{\bf plastic-fluid phase};
  \draw (0,0)--(30:5) node[above]{\it right-facing plastic wave};
  \draw (0,0)--(60:5) node[above]{\it right-facing fluid wave};
  \draw[thick] (0,0)--(110:4.8) node[above]{\it interface};
 \end{tikzpicture}
   \caption{Wave structure of the plastic-fluid phase in the $\xi-\tau$ 
space.}
  \label{fig:Riemann-plastofluid}
 \end{minipage}
  \begin{minipage}[b]{.49\textwidth}
  \begin{tikzpicture}[scale=.65]
  \tikzset{font={\fontsize{9pt}{12}\selectfont}}
  \fill [lightgray,opacity=.5] (0,0)--(4.4,0) arc(0:110:4.4);
  \draw[very thick,-triangle 45](-5.5,0)--(5.5,0)node[right]
  {\normalsize $\xi$};
  \draw[very thick,-triangle 45] (0,0)--(0,5.5) node[above]
  {\normalsize $\tau$};
  \draw (2.7,0.3) node{right initial state $\bm U_r$};
  \draw (-.3,1.8) node{$\bm U_r^*$};
  \draw (1.0,1.8) node{$\bm U_r^{\nabla}$}; 
  \draw (1.8,1.5) node{$\bm U_r^{\Delta}$};  
  \draw (4.2,1.2) node{\bf elas-plas-fluid phase};
  \draw (0,0)--(30:5) node[above]{\it right-facing elastic wave};
  \draw (0,0)--(50:5) node[above]{\it right-facing plastic wave};
  \draw (0,0)--(70:5) node[above]{\it right-facing fluid wave};
  \draw[thick] (0,0)--(110:4.8) node[above]{\it interface};
 \end{tikzpicture}
   \caption{Wave structure of the elastic-plastic-fluid solid phase in the 
$\xi-\tau$ 
space.}
  \label{fig:Riemann-elas-plas-fluid}
 \end{minipage}
\end{figure}

\subsection{Solution in the elastic-plastic-fluid phase}
If the solid undergoes an elastic-plastic-fluid phase transition, shown in Fig. 
\ref{fig:Riemann-elas-plas-fluid}, there exist acoustic elastic, plastic and 
fluid rarefaction waves or shock waves on the isentropic curves or Hugoniot 
loci, which are distinguished by the elastic limit and plastic limit, 
respectively.

\begin{itemize}
\item[-] {\bf Elastic-plastic-fluid rarefaction wave}

If $q_k^* \le q_{_\mathscr {PT}} < q_{_\mathscr T} \le q_k$, the acoustic 
elastic, plastic and fluid wave are all rarefaction waves,
\begin{align*}
u_k^*-u_k&=\displaystyle
\int_{q_k}^{q_{_\mathscr{T}}}
\left(\rho^2c^2+\dfrac{4\beta_k^{_\mathscr E}}{3}\right)^{-1/2}\dd q +
\displaystyle\int_{q_{_\mathscr{T}}}^{q_{_\mathscr{PT}}}
\left(\rho^2c^2+\dfrac{4\beta_k^{_\mathscr P}}{3}\right)^{-1/2}\dd q +
\displaystyle\int_{q_{_\mathscr{PT}}}^{q_k^*}
\dfrac{1}{\rho c}\dd q, \\
\rho_k^*-\rho_k&=\displaystyle
\int_{q_k}^{q_{_\mathscr{T}}}
\left(c^2+\dfrac{4\beta_k^{_\mathscr E}}{3\rho^2}\right)^{-1}\dd q
+\displaystyle\int_{q_{_\mathscr{T}}}^{q_{_\mathscr{PT}}}
\left(c^2+\dfrac{4\beta_k^{_\mathscr P}}{3\rho^2}\right)^{-1}\dd q +
\displaystyle\int_{q{_{_\mathscr{PT}}}}^{q_k^*}
\dfrac{1}{c^2}\dd q.
\end{align*}

\item[-] {\bf Elastic-plastic-fluid shock wave}

If $q_k^*>q_{_\mathscr{PC}}>q_{_\mathscr{C}}>q_k$, then the acoustic
elastic, plastic and fluid wave are all shock waves,
\begin{equation}
\small
\begin{aligned}
\label{eq:hug_fp}
  u_k^*-u_k &= \left(-(q_{_\mathscr{C}}-q_k)
    \left(\dfrac{1}{\rho_{_\mathscr{C}}}
          -\dfrac{1}{\rho_k}\right)\right)^{1/2} \\ 
       &+\left(-(q_{_\mathscr{PC}}-q_{_\mathscr{C}})
       \left(\dfrac{1}{\rho_{_\mathscr{PC}}}
          -\dfrac{1}{\rho_{_\mathscr{C}}}\right)\right)^{1/2} 
      +\left(-(q_k^*-q_{_\mathscr{PC}})
      \left(\dfrac{1}{\rho_k^*}
          -\dfrac{1}{\rho_{_\mathscr{PC}}}\right)\right)^{1/2},\\
   e_k(p_k^*,\rho_k^*) &- e_k(p_k,\rho_k)
  +\dfrac{1}{2}(p_k^*+p_{_\mathscr{PC}})
  \left(\dfrac{1}{\rho_k^*}
  -\dfrac{1}{\rho_{_\mathscr{PC}}}\right) \\
&+\dfrac{1}{2}(p_{_\mathscr{PC}}+p_{_\mathscr{C}})
  \left(\dfrac{1}{\rho_{_\mathscr{PC}}}
  -\dfrac{1}{\rho_{_\mathscr{C}}}\right) 
  +\dfrac{1}{2}(p_{_\mathscr{C}}+p_k)
   \left(\dfrac{1}{\rho_{_\mathscr{C}}}
  -\dfrac{1}{\rho_k}\right)=0.
\end{aligned}
\end{equation}

Similar to the elastoplastic solid phase, we define 
$\varPhi^{_\mathscr {EPF}}_k(q,\rho)$ to relate $q$ and $\rho$ on the 
elastic-plastic-fluid Hugoniot locus
\begin{equation}
\varPhi^{_\mathscr {EPF}}_k(q,\rho):=
\varPhi^{_\mathscr {E}}_k(q_{_\mathscr {C}},\rho_{_\mathscr {C}})+
\varPhi^{_\mathscr {PC}}_k(q_{_\mathscr {PC}},\rho_{_\mathscr {PC}})+
\varPhi^{_\mathscr {F}}_k(q,\rho),\quad
q>q_{_\mathscr {C}}>q_{_\mathscr {PC}}>q_k,
\label{eq:phi_epf}
\end{equation}
where 
\begin{align*}
\varPhi^{_\mathscr{PC}}_k(q_{_\mathscr{PC}},\rho_{_\mathscr{PC}}):&=
\varGamma_k(\rho_{_\mathscr{C}})\rho_{_\mathscr{C}}
\left(q_{_\mathscr{PC}}
+S_{_\mathscr{PC}}-h_k(\rho_{_\mathscr{PC}})\right)
-\varGamma_k(\rho_{_\mathscr{PC}}) \rho_{_\mathscr{PC}} 
 (p_{_\mathscr{C}}-h_k(\rho_{_\mathscr{C}})) \\
&-\dfrac{1}{2}\varGamma_k(\rho_{_\mathscr{C}})(q_{_\mathscr{PC}}
+S_{_\mathscr{PC}}+p_{_\mathscr{C}})
\varGamma_k(\rho_{_\mathscr{PC}})
(\rho_{_\mathscr{PC}}-\rho_{_\mathscr{C}}).
\end{align*}

Similarly, we have the following results on the function 
$\varPhi^{_\mathscr {EPF}}_k(q,\rho)$ by a simple calculus
\begin{lemma}\label{thm:phiepf}
The Hugoniot function $\varPhi^{_\mathscr {EPF}}_k(q,\rho)$ defined in
\eqref{eq:phi_epf} satisfies the following properties: 
1). $\varPhi^{_\mathscr {EPF}}_k(q,\rho_k)>0$;~
2). $\varPhi^{_\mathscr {EPF}}_k(q,\rho_{\max})<0$;~
3). ${\partial\varPhi^{_\mathscr {EPF}}_k}(q,\rho)/{\partial \rho}<0$;~
4). ${\partial^2\varPhi^{_\mathscr {EPF}}_k}(q,\rho)/{\partial \rho^2}<0$ if 
$\varGamma''_k(\rho)=0$.
\end{lemma}
\end{itemize}

\subsection{Solution of the Riemann problem}

For the hydro-elastoplastic solid Riemann problem, the following compatibility 
conditions are imposed across the interface
\[
 \begin{aligned}
  u_l^* &= u_r^*, \\
  p_l^*-S_l^* &= p_r^*-S_r^*.
 \end{aligned}
\]

Let $q^*=p_l^*-S_l^*=p_r^*-S_r^*$. Equating the interface normal velocity 
$u^*=u_l^*= u_r^*$ yields
\[
u^*=u_l-f_l(q^*)=u_r+f_r(q^*),
\]
where the expressions of $f_k(q),~f_k'(q),~f_k''(q)~(k=l,r)$ for each phase are 
collected in Tab. \ref{tab:coe_f}. Here ``$|$'' denotes the interface, the 
superscript ``$S$'' and ``$R$'' stand for the shock wave and rarefaction wave, 
respectively. $\varPhi^m_k(q,\rho)$ denotes the algebraic equation of the 
Hugoniot locus for the shock wave, and $\chi_k^m(q,\rho)$ denotes its slope, 
where $m=\mathscr {F}, \mathscr {E}, \mathscr {P}, \mathscr {EP}, \mathscr 
{PF}, \mathscr {EPF}$. ``$\mathscr F$'', ``$\mathscr E$'', ``$\mathscr P$'', 
``$\mathscr {EP}$'', ``$\mathscr {PF}$'' and  ``$\mathscr {EPF}$'' denote the 
types of acoustic wave in hydro-elastoplastic solid, which is elastic wave, 
plastic wave, fluid wave, elastic-plastic wave, plastic-fluid wave and 
elastic-plastic-fluid wave, respectively. 

Therefore, the interface normal stress $q^*$ is exactly the zero of the
following \emph{stress function}
\begin{equation}
f(q) := f_l(q)+f_r(q) + u_r - u_l.
\label{eq:deff}
\end{equation}
And the interface velocity $u^*$ can be determined from
\[
u^*=\dfrac{1}{2}(u_l+u_r+f_r(q^*)-f_l(q^*)).
\label{eq:usdef}
\]

The behavior of $f_k(q)$ is related to the existence and uniqueness of the
solution of the Riemann problem. We claim on $f_k(q)$ that 
\begin{lemma}\label{thm:propf}
Assume that the conditions \textbf{(C1)-(C3)} hold for $\varGamma_k(\rho)$ and 
$h_k(\rho)$, the function $f_k(q)$ is monotonically increasing and concave, i.e.
  \[
  f_k'(q)>0 \quad \text{ and } \quad f_k''(q)<0,
  \]
  if the Hugoniot function is concave with respect to the density,
  i.e. ${\partial^2 \varPhi^m_k(q,\rho)}/{\partial \rho^2}<0$.
\end{lemma}
\begin{proof}
  The first and second derivatives of $f_k(q)$ can be found in Tab.
\ref{tab:coe_f}. The result then follows by a direct observation.
\end{proof}

Here we provide a short proof of the results for the Riemann problem with 
Mie-Gr{\"u}neisen EOS and hydro-elastoplastic constitutive law in the following 
theorem.

\begin{theorem}\label{thm:unique}
The Riemann problem \eqref{eq:deff} has a unique solution (in the class of
admissible shocks, interfaces and rarefaction waves separating constant states)
if and only if the initial states satisfy the constraint
\begin{equation}
 u_r-u_l<\displaystyle\int_{q_{l,\min}}^{q_l}\dfrac{1}{\rho c}{\mathrm dq}+
\displaystyle\int_{q_{r,\min}}^{q_r}\dfrac{1}{\rho c}{\mathrm dq},
 \label{eq:vacuum}
 \end{equation}
where $q_{l,\min},~q_{r,\min}$ are the cut-off stresses in tension for each 
hydro-elastoplastic solid.
\end{theorem}
\begin{proof}
We first notice that for the left- and right-facing waves, the derivative
$f_k'(q)$ in Tab. \ref{tab:coe_f} and Tab. \ref{tab:coe_s} is always positive, 
and as a result, the stress function $f(q)$ is monotonically increasing. 

Next we study the behavior of $f(q)$ when $q$ tends to infinity. Let 
$\tilde\rho$ represent the density such that $\varPhi^m_k(\tilde q,
\tilde\rho)=0$ for a given $\tilde q$, which is the equation relates $\rho$ and
$q$ along the Hugoniot locus. When the stress $q>\tilde q$, we have
$\rho>\tilde\rho$, according to the monotonicity of the Hugoniot locus, and thus
\[
f_k^2(q)=(q-q_k)\left(\dfrac{1}{\rho_k}-\dfrac{1}{\rho}\right)
>(q-q_k)\left(\dfrac{1}{\rho_k}-\dfrac{1}{\tilde\rho}\right).
\]
As a result, $f_k(q)$ tends to positive infinity as $q\rightarrow +\infty$ and
so does $f(q)$.

Based on the behavior of the function $f(q)$, a necessary and sufficient
condition for the interface stress $q^*>q_{\min}$ such that $f(q^*)=0$ to be
uniquely defined is given by
\[
f(q_{\min})=f_l(q_{\min})+f_r(q_{\min})+u_r-u_l<0,
\]
or equivalently, the constraint given by \eqref{eq:vacuum}, where 
$q_{\min}= \max({q_{l,\min},q_{r,\min}})$. This completes the proof of the
theorem.
\end{proof}

\begin{remark}
When the initial states violate the constraint \eqref{eq:vacuum}, the 
Riemann problem has no solution in the above sense. One can yet define
a solution by introducing a \emph{vacuum}. However, we are not going
to address this issue which is beyond the scope of our current study.
\end{remark}

\setlength{\tabcolsep}{1.4pt}
\begin{sidewaystable}[p]
 \centering
 \caption{Expressions of stress functions and their derivatives for rarefaction
waves.}
 \label{tab:coe_f}
\begin{threeparttable}
{\small 
\begin{tabular}{p{4em}ccccc}
  \toprule
   & $q$  &  $f_k(q)$  &  $f_k'(q)$ & $f_k''(q)$
   & Acoustic wave type\\
 \midrule
 \multirow{19}*{Solid}
 & $q_{_\mathscr {PT}} < q_{_\mathscr T} < q\le q_k$
 & $\displaystyle\int^{q}_{q_k}
   \left(\rho^2c^2+\dfrac{4}{3}\beta_k^{\mathscr E}\right)^{-1/2}\dd
q$ 
 & $\left(\rho^2c^2+\dfrac{4}{3}\beta_k^{\mathscr E}\right)^{-1/2}$
 & $-\rho\mathscr G\left(\rho^2c^2+\dfrac{4}{3}
    \beta_k^{\mathscr E}\right)^{-3/2}$
 & $\left(\left|_k^{R}\right.\right)^{\mathscr E}$ \\ [5mm]
 & $q_{_\mathscr {PT}} < q\le q_k< q_{_\mathscr T}$
 & $\displaystyle\int^{q}_{q_k}
   \left(\rho^2c^2+\dfrac{4}{3}\beta_k^{\mathscr P}\right)^{-1/2}\dd q$
 & $\left(\rho^2c^2+\dfrac{4}{3}\beta_k^{\mathscr P}\right)^{-1/2}$ 
 & $-\rho\mathscr G\left(\rho^2c^2+\dfrac{4}{3}
    \beta_k^{\mathscr P}\right)^{-3/2}$
 & $\left(\left|_k^{R}\right.\right)^{\mathscr{P}}$ \\ [5mm]
 & $q\le q_k< q_{_\mathscr {PT}} < q_{_\mathscr {T}} $
 & $\displaystyle\int^{q}_{q_k} \dfrac{1}{\rho c}\dd q$
 & $\dfrac{1}{\rho c}$ 
 & $-\dfrac{\mathscr G}{\rho^2c^3}$
 & $\left(\left|_k^{R}\right.\right)^{\mathscr{F}}$ \\ [5mm]
  & $q_{_\mathscr {PT}} < q\le q_{_\mathscr T}\le q_k$
 & $\begin{aligned}
   &\displaystyle\int^{q_{_\mathscr T}}_{q_k}
   \left(\rho^2c^2+\dfrac{4}{3}\beta_k^{\mathscr E}\right)^{-1/2}\dd q+ \\ 
   &\displaystyle\int^{q}_{q_{_\mathscr T}}
   \left(\rho^2c^2+\dfrac{4}{3}\beta_k^{\mathscr P}\right)^{-1/2}\dd q
   \end{aligned}$
 & $\left(\rho^2c^2+\dfrac{4}{3}\beta_k^{\mathscr P}\right)^{-1/2}$ 
 & $-\rho\mathscr G\left(\rho^2c^2+\dfrac{4}{3}
    \beta_k^{\mathscr P}\right)^{-3/2}$
 & $\left(\left|_k^{R}\right.\right)^{\mathscr{EP}}$ \\ [5mm]
 & $q \le q_{_\mathscr{PT}} < q_k \le q_{_\mathscr{T}}$
 & $\begin{aligned} 
   &\displaystyle\int^{q_{_\mathscr{PT}}}_{q_k}
   \left(\rho^2c^2+\dfrac{4}{3}\beta_k^{\mathscr P}\right)^{-1/2}\dd q+\\   
   &\displaystyle\int^{q}_{q_{_\mathscr{PT}}}
   \dfrac{1}{\rho c}\dd q
   \end{aligned}$
 & $\dfrac{1}{\rho c}$ 
 & $-\dfrac{\mathscr G}{\rho^2c^3}$
 & $\left(\left|_k^{R}\right.\right)^{\mathscr{PF}}$ \\ [5mm]
 & $q\le q_{_\mathscr{PT}}< q_{_\mathscr{T}}\le q_k$
 & $\begin{aligned}
   &\displaystyle\int^{q_{_\mathscr{T}}}_{q_k}
   \left(\rho^2c^2+\dfrac{4}{3}\beta_k^{\mathscr E}\right)^{-1/2}\dd q+\\ 
   &\displaystyle\int^{q_{_\mathscr{PT}}}_{q_\mathscr{T}}
   \left(\rho^2c^2+\dfrac{4}{3}\beta_k^{\mathscr P}\right)^{-1/2}\dd q+\\   
   &\displaystyle\int^{q}_{q_{_\mathscr{PT}}}
   \dfrac{1}{\rho c}\dd q
   \end{aligned}$
 & $\dfrac{1}{\rho c}$ 
 & $-\dfrac{\mathscr G}{\rho^2c^3}$
 & $\left(\left|_k^{R}\right.\right)^{\mathscr{EPF}}$ \\ [5mm]
 \bottomrule
\end{tabular}
 }
\end{threeparttable}   
\end{sidewaystable}

\setlength{\tabcolsep}{1.4pt}
\begin{sidewaystable}[p]
 \centering
 \caption{Expressions of stress functions and their derivatives for shock
waves.}
 \label{tab:coe_s}
\begin{threeparttable}
{\small 
\begin{tabular}{p{4em}ccccc}
  \toprule
   & $q$  &  $f_k(q)$  &  $f_k'(q)$ & $f_k''(q)$
   & Acoustic wave type \\
 \midrule
 \multirow{19}*{Solid}
 & $q_k<q<q_{_\mathscr C}<q_{_\mathscr {PC}}$
 & $\left((q-q_k)\left(\dfrac{1}{\rho_k}-\dfrac{1}{\rho}\right)
    \right)^{1/2}$
 & $\dfrac{1}{2f_k(q)}\left(\dfrac{1}{\rho_k}-
   \dfrac{1}{\rho}+\dfrac{q-q_k}
   {\rho^2\chi_{_\mathscr E}}\right)$  
 & $\begin{aligned}
   &-\dfrac{1}{4f_k^3(q)}\left(\dfrac{2(q-q_k)^2}
   {\rho^2\chi_{_\mathscr E}^2}
   \left(\dfrac{1}{\rho_k}-\dfrac {1}{\rho }\right)\right.+ \\
   &\left.\left(\dfrac{2}{\rho}+\left.
   \pd{\chi_{_\mathscr E}}{q}
   \right|_{\varPhi_k^{\mathscr E}}\right)
   \left(\dfrac{1}{\rho_k}-\dfrac{1}{\rho}-\dfrac{q-q_k}
   {\rho^2\chi_{_\mathscr E}}\right)^2\right) 
   \end{aligned}$
 & $\left(\left|_k^{S}\right.\right)^{\mathscr E}$\\ [5mm]
 & $q_{_\mathscr C}\le q_k<q<q_{_\mathscr {PC}}$
 & $\left((q-q_k)\left(\dfrac{1}{\rho_k}-\dfrac{1}{\rho}\right)
   \right)^{1/2}$
 & $\dfrac{1}{2f_k(q)}\left(\dfrac{1}{\rho_k}-
   \dfrac{1}{\rho}+\dfrac{q-q_k}{\rho^2
   \chi_{_\mathscr{P}}}\right)$
 & $\begin{aligned}
   &-\dfrac{1}{4f_k^3(q)}\left(\dfrac{2(q-q_k)^2}
   {\rho^2\chi_{_\mathscr{P}}^2}  
   \left(\dfrac{1}{\rho_k}-\dfrac {1}{\rho }\right)\right.+ \\
   &\left.\left(\dfrac{2}{\rho}+\left.
   \pd{\chi_{_\mathscr{P}}}{q}
   \right|_{\varPhi_k^\mathscr{P}}\right)
   \left(\dfrac{1}{\rho_k}-\dfrac{1}{\rho}-\dfrac{q-q_k}
   {\rho^2\chi_{_\mathscr{P}}}\right)^2\right) \end{aligned}$
 & $\left(\left|_k^{S}\right.\right)^{\mathscr{P}}$\\ [5mm]
 & $q_{_\mathscr {PC}}\le q_k<q<q_{_\mathscr {PC}}$
 & $\left((q-q_k)\left(\dfrac{1}{\rho_k}
   -\dfrac{1}{\rho}\right)\right)^{1/2}$ 
 & $\dfrac{1}{2f_k(q)}\left(\dfrac{1}{\rho_k}-
   \dfrac{1}{\rho}+\dfrac{q-q_k}{\rho^2\chi_{_\mathscr F}}\right)$ 
 & $\begin{aligned}
   &-\dfrac{1}{4f_k^3(q)}\left(\dfrac{2(q-q_k)^2}
   {\rho^2\chi_{_\mathscr F}^2}  
   \left(\dfrac{1}{\rho_k}-\dfrac {1}{\rho }\right)\right.+ \\
   &\left.\left(\dfrac{2}{\rho}+\left.\pd{\chi_{_\mathscr F}}{q}
   \right|_{\varPhi_k^{\mathscr F}}\right)
   \left(\dfrac{1}{\rho_k}-\dfrac{1}{\rho}-\dfrac{q-q_k}
   {\rho^2\chi_{_\mathscr F}}\right)^2\right) \end{aligned}$
 &  $\left(\left|_k^{S}\right.\right)^{\mathscr{F}}$\\[5mm] 
& $q_k<q_{_\mathscr C}\le q<q_{_\mathscr {PC}}$ 
 & $\begin{aligned}
   &\left((q_{_\mathscr C}-q_k)\left(\dfrac{1}{\rho_k}-\dfrac{1}
   {\rho_{c}}\right)\right)^{1/2} \\
   +&\left((q-q_{_\mathscr C})
   \left(\dfrac{1}{\rho_{_\mathscr C}}-\dfrac{1}{\rho}\right)
   \right)^{1/2} \end{aligned} $ 
 & $\dfrac{1}{2f_k(q)}\left(\dfrac{1}{\rho_{_\mathscr C}}-
   \dfrac{1}{\rho}+\dfrac{q-q_{_\mathscr C}}
   {\rho^2\chi_{_\mathscr{EP}}}\right)$ 
 & $\begin{aligned}
   &-\dfrac{1}{4f_k^3(q)}\left(\dfrac{2(q-q_{_\mathscr C})^2}
   {\rho^2\chi_{_\mathscr{EP}}^2}  
   \left(\dfrac{1}{\rho_{_\mathscr C}}-\dfrac {1}{\rho }\right)\right.+ \\
   &\left.\left(\dfrac{2}{\rho}
   +\left.\pd{\chi_{_\mathscr{EP}}}{q}
   \right|_{\varPhi_k^\mathscr{EP}}\right)
   \left(\dfrac{1}{\rho_{_\mathscr C}}
   -\dfrac{1}{\rho}-\dfrac{q-q_{_\mathscr C}}
   {\rho^2\chi_{_\mathscr{EP}}}\right)^2\right) \end{aligned}$
 & $\left(\left|_k^{S}\right.\right)^{\mathscr{EP}}$ \\ [5mm]
& $q_{_\mathscr{C}} < q_k<q_{_\mathscr{PC}} < q  $ 
 & $\begin{aligned}
   &\left((q_{_\mathscr{PC}}-q_k)
    \left(\dfrac{1}{\rho_k}-\dfrac{1}
   {\rho_{_\mathscr{PC}}}\right)\right)^{1/2} \\
   +&\left((q-q_{_\mathscr{PC}})
   \left(\dfrac{1}{\rho_{_\mathscr{PC}}}-\dfrac{1}{\rho}\right)
   \right)^{1/2} \end{aligned} $ 
 & $\dfrac{1}{2f_k(q)}\left(\dfrac{1}{\rho_{_\mathscr{PC}}}-
   \dfrac{1}{\rho}+\dfrac{q-q_{_\mathscr{PC}}}
   {\rho^2\chi_{_\mathscr{PF}}}\right)$ 
 & $\begin{aligned}
   &-\dfrac{1}{4f_k^3(q)}\left(\dfrac{2(q-q_{_\mathscr {PC}})^2}
   {\rho^2\chi_{_\mathscr{PF}}^2}  
   \left(\dfrac{1}{\rho_{_\mathscr{PC}}}-\dfrac {1}{\rho }\right)\right.+ \\
   &\left.\left(\dfrac{2}{\rho}
   +\left.\pd{\chi_{_\mathscr{PF}}}{q}
   \right|_{\varPhi_k^\mathscr{PF}}\right)
   \left(\dfrac{1}{\rho_{_\mathscr{PC}}}
   -\dfrac{1}{\rho}-\dfrac{q-q_{_\mathscr{PC}}}
   {\rho^2\chi_{_\mathscr{PF}}}\right)^2\right) \end{aligned}$
 & $\left(\left|_k^{S}\right.\right)^{\mathscr{PF}}$ \\ [5mm]
 & $q_k<q_{_\mathscr{PC}} < q_{_\mathscr{C}} < q$ 
 & $\begin{aligned}
   &\left((q_{_\mathscr{C}}-q_k)\left(\dfrac{1}{\rho_k}-\dfrac{1}
   {\rho_{_\mathscr{C}}}\right)\right)^{1/2} \\
  +&\left((q_{_\mathscr{PC}}-q{_{_\mathscr{C}}})
    \left(\dfrac{1}{\rho_{_\mathscr{C}}}-\dfrac{1}
   {\rho_{_\mathscr{PC}}}\right)\right)^{1/2} \\
   +&\left((q-q_{_\mathscr{PC}})
   \left(\dfrac{1}{\rho_{_\mathscr{PC}}}-\dfrac{1}{\rho}\right)
   \right)^{1/2} \end{aligned} $ 
 & $\dfrac{1}{2f_k(q)}\left(\dfrac{1}{\rho_{_\mathscr{PC}}}-
   \dfrac{1}{\rho}+\dfrac{q-q_{_\mathscr{PC}}}
   {\rho^2\chi_{_\mathscr{EPF}}}\right)$ 
 & $\begin{aligned}
   &-\dfrac{1}{4f_k^3(q)}\left(\dfrac{2(q-q_{_\mathscr {PC}})^2}
   {\rho^2\chi_{_\mathscr{EPF}}^2}  
   \left(\dfrac{1}{\rho_{_\mathscr{PC}}}-\dfrac {1}{\rho }\right)\right.+ \\
   &\left.\left(\dfrac{2}{\rho}
   +\left.\pd{\chi_{_\mathscr{EPF}}}{q}
   \right|_{\varPhi_k^\mathscr{EPF}}\right)
   \left(\dfrac{1}{\rho_{_\mathscr{PC}}}
   -\dfrac{1}{\rho}-\dfrac{q-q_{_\mathscr{PC}}}
   {\rho^2\chi_{_\mathscr{EHF}}}\right)^2\right) \end{aligned}$
 & $\left(\left|_k^{S}\right.\right)^{\mathscr{EPF}}$ \\ [5mm]
 \bottomrule
\end{tabular}
 }
\end{threeparttable}   
\end{sidewaystable}

\section{Approximate Riemann Solver}\label{sec:aps}
Our algorithm of the Riemann solver is to find the unique zero of the stress
function $f(q)$ using the Newton-Raphson method \cite{Godunov1976}
\[
q_{n+1}=q_{n}-\dfrac{f(q_{n})}{f'(q_{n})}
=q_{n}-\dfrac{f_l(q_{n})+f_r(q_{n}) + u_r - u_l}
{f_l'(q_{n})+f_r'(q_{n})}.
\]

Unfortunately, there is generally no close-form expression for the stress
function $f(q)$ or its derivative $f'(q)$ for some complex equations of
state. Instead we perform the \emph{inexact Newton method}, which is formulated 
as
\begin{equation}
\left\{
\begin{aligned}
q_{n+1}&=q_n-\dfrac{F_n}{F_n'}=q_n-\dfrac{F_{n,l}+F_{n,r}+u_r-u_l}
{F'_{n,l}+F'_{n,r}},\\
u_{n}&=\dfrac{1}{2}(u_l+u_r+F_{n,r}-F_{n,l}),
\end{aligned}\right.
\label{eq:iterp2}
\end{equation}
where $F_{n,k}$ and $F'_{n,k}$ approximate $f_k(q_{n})$ and $f_k'(q_{n})$,
respectively.

To specify the sequences $F_{n,k}$ and $F'_{n,k}$, we compute the shock branch
using an iterative method, and the rarefaction branch through numerical
integration. It is natural to expect that the sequences $q_n$ and $u_n$ will
tend to $q^*$ and $u^*$ respectively, whenever the evaluation errors
$|F_{n,k}-f_k(q_n)|$ and $|F'_{n,k}-f'_k(q_n)|$ are going to zero, which
have been proved in our previous work \cite{Lichen2018}. The convergence is
guaranteed by \textit{a posteriori} control on the evaluation errors of
$f_k(q_n)$ and $f'_k(q_n)$, which depend on the residual of the algebraic
equation in the shock branch as well as the truncation error of the ordinary
differential equation in the rarefaction branch. Here we apply the
Newton-Raphson method to solve the Hugoniot loci, and the adaptive
Runge-Kutta-Fehlberg method \cite{Fehlberg1970} to solve the isentropic curves.

Precisely, if $q_n>q_k$, for the given $n$-th iterator $q_n$, we solve the
following algebraic equation
\begin{equation}
\varPhi^m_k(q_n,\tilde \rho_{n,k})=0,
\label{eq:algphi}
\end{equation}
to obtain $\tilde\rho_{n,k}$ to a prescribed tolerance by the Newton-Raphson
method
\[
\rho_{n,k,j+1}=\rho_{n,k,j}- \dfrac{\varPhi^m_k(q_n,\rho_{n,k,j})}
{\partial\varPhi^m_k(q_n,\rho_{n,k,j})/\partial\rho}.
\]

By Lemma \ref{thm:phie}, \ref{thm:phiep}, \ref{thm:phipf} and 
\ref{thm:phiepf}, we can naturally get the conclusion that the Newton-Raphson 
iteration for \eqref{eq:algphi} must converge for any initial guess 
$\rho>\rho_k$. Then the values of $F_{n,k}$ and $F'_{n,k}$ for the shock branch 
are thus taken as
\begin{align}
\label{eq:shockf}
  F_{n,k}&=\left((q_n-q_k)\left(\dfrac{1}{\rho_k}-\dfrac{1}
           {\tilde\rho_{n,k}}\right)\right)^{1/2},\\ 
\label{eq:shockdf}
  F'_{n,k}&=\dfrac{1}{2F_{n,k}}\left(
            \dfrac{1}{\rho_k}-\dfrac{1}{\tilde\rho_{n,k}}
            +\dfrac{q_n-q_k}{\rho_{n,k}^2~
            \chi_k^{m}(q_n,\tilde\rho_{n,k})}
            \right).
\end{align}

If, on the other hand, $q_n\le q_k$, then we solve the following system of the
initial value problem 
\begin{equation}
\label{ode:solid}
\left\{
 \begin{array}{ll}
   \dfrac{\mathrm d f_k(q)}{\mathrm d q}=
   \left(\rho^2c^2+\dfrac{4\beta^m_k}{3}\right)^{-1/2}, & 
   f_k|_{q=q_k}=0, \\
 \dfrac{\mathrm d\rho}{\mathrm d q}=
 \left(c^2+\dfrac{4\beta^m_k}{3\rho^2}\right)^{-1}, &
 \rho|_{q=q_k}=\rho_k,
 \end{array}\right.
\end{equation}
backwards until $q=q_n$ using the adaptive Runge-Kutta-Fehlberg method. 

When the initial states $\bm U_l, \bm U_r$ and the global tolerance $\epsilon_0$
are given, the whole procedure of the approximate Riemann solver for
\eqref{eq:deff} is as below.
\begin{mdframed}
\setlength{\parindent}{0pt}
\textbf{Step 1} Provide an initial estimate of the interface normal stress
\[
 q_{0} =
\dfrac{\rho_lc_lq_r+\rho_rc_rq_l+\rho_lc_l\rho_rc_r(u_l-u_r)}
{\rho_lc_l+\rho_rc_r}.
\]
	
\textbf{Step 2} Assume that the $n$-th iteration $q_{n}$ is obtained. Determine
the type of the left and right nonlinear waves.
	
(i) If $q_{n}>\max\{q_l,q_r\}$, then both nonlinear waves are shock waves.
	
(ii) If $\min\{q_l,q_r\}\le q_{n} \le \max\{q_l,q_r\}$, then one of the two
nonlinear waves is a shock wave, and the other is a rarefaction wave.
	
(iii) If $q_{n}<\min\{q_l,q_r\}$, then both nonlinear waves are rarefaction
waves. 

\textbf{Step 3} Evaluate $F_{n,k}$ and $F'_{n,k}$ according to the type of
nonlinear waves and the local evaluation error $\varepsilon_{n,k}$. 

(i) When the nonlinear wave is a rarefaction wave, estimate the local evaluation
error $\varepsilon_{n,k}$ according to the condition numbers of system
\eqref{ode:solid}, and calculate $F_{n,k}$ and $F'_{n,k}$ by using the adaptive 
Runge-Kutta-Fehlberg method. 

(ii) When the nonlinear wave is a shock wave, estimate the local residual of
the algebraic equation \eqref{eq:algphi} according to its condition number, and
get the corresponding $\tilde\rho_{n,k}$ by using the Newton-Raphson method.
Then calculate $F_{n,k}$ and $F'_{n,k}$ by \eqref{eq:shockf} and
\eqref{eq:shockdf}. 
 
\textbf{Step 4} Update the interface normal stress through
\[
q_{n+1}=q_{n}-\dfrac{F_{n,l}+F_{n,r} + u_r - u_l}{F'_{n,l}+F'_{n,r}}.
\]
	
\textbf{Step 5} Terminate whenever the relative change of the stress reaches the
prescribed tolerance $\epsilon_0$. The sufficiently accurate estimate $q_{n}$ is
then taken as the approximate interface normal stress $q^*$. Otherwise return to
Step 2.
	
\textbf{Step 6} Compute the interface velocity $u^*$ through
\[
u^*=\dfrac{1}{2}\left(u_l+u_r+F_{n,r}-F_{n,l}\right).
\]
\end{mdframed}


\section{Application on Multi-medium Interaction}\label{sec:sch} 

Now we consider the compressible multi-medium interaction problems described by 
an immiscible model in the domain $\Omega$. Two mediums are separated by a sharp 
interface $\Gamma(t)$ characterized by the zero of the level set function 
$\phi(\bm x,t)$. The region occupied by each medium can be expressed in terms of 
$\phi(\bm x,t)$
\[
\Omega^+(t):=\{\bm x\in\Omega~|~\phi(\bm x,t)> 0\}
\an
\Omega^-(t):=\{\bm x\in\Omega~|~\phi(\bm x,t)< 0\}.
\]
And the medium in each region is governed by the following governing equations
\begin{equation} 
  \dfrac{\partial \bm{U}}{\partial t} + 
  \nabla\cdot \bm F(\bm U)= 
  \bm 0,\quad \bm x\in \Omega^\pm (t),
  \label{eq:euler} 
\end{equation}
where
$\bm{U}=
[\rho, ~ \rho \bm u, ~ E]^\top$, 
$\bm{F(\bm U)}=
[\rho \bm u^\top, ~\rho \bm u\otimes \bm u-\bm\sigma,~E\bm
u^\top-\bm\sigma\cdot\bm u^\top]^\top$. 
Here $\bm u$ stands for the velocity vector, and other variables represent the 
same as that in \eqref{system:oneriemann}. The equation of state and 
constitutive law have been given in Section \ref{sec:rp}.

We extend the numerical scheme in Guo \textit{et al.} \cite{Guo2016} to the
hydro-elastoplastic problems, which is implemented on Eulerian grids. For 
completeness, we briefly sketch the main steps of the numerical scheme for the
multi-medium flow therein. The approximate Riemann solver we proposed is applied
to calculate the numerical flux at the phase interface in the overall numerical
scheme. The whole domain $\Omega$ is divided into a conforming mesh with simplex
cells, and the overall scheme is mainly divided into three steps:

\begin{itemize}
\item[(1).] {\bf Evolution of the interface}

The level set function is approximated by a continuous piecewisely linear
function, which satisfies
\begin{equation}
\dfrac{\partial \phi}{\partial t} +
\tilde u|\nabla\phi| = 0.
\label{eq:levelset:eov}
\end{equation}
Here $\tilde u$ denotes the normal velocity of the phase interface, where the
normal direction is chosen as the gradient of the level set function. 

The discretized level set function \eqref{eq:levelset:eov} is updated through
the characteristic line tracking method once the motion of the phase interface
is given. Due to the nature of the level set equation, it remains to specify the
normal velocity $\tilde u$ within a narrow band near the phase interface. This
can be achieved by firstly solving a multi-medium Riemann problem across the
phase interface and then extending the velocity field to the nearby region using
the harmonic extension technique of Di \textit{et al.} \cite{Di2007}. The
solution of the multi-medium Riemann problem has been elaborated in Section
\ref{sec:rp}.

In order to keep the property of the signed distance function, we solve the
following reinitialization equation 
\[
\begin{cases}
  \dfrac{\partial \psi}{\partial \tau} =
  \sgn(\psi_0)\cdot\left(1-\left|\nabla \psi \right|\right), \\
  \psi(\bm x,0)=\psi_0=\phi(\bm x, t),
\end{cases}
\]
until steady state using the explicitly positive coefficient scheme
\cite{Di2007}.

Once the level set function is updated until the $n$-th time level, we can
obtain the discretized phase interface $\Gamma_{h,n}$. A cell $\mathscr K_{i,n}$
is called an \emph{interface cell} if the intersection of $\mathscr K_{i,n}$ and
$\Gamma_{h,n}$, denoted as $\Gamma_{\mathscr K_{i,n}}$, is nonempty. Since the
level set function is piecewisely linear and the cell is simplex,
$\Gamma_{\mathscr K_{i,n}}$ must be a linear manifold in $\mathscr K_{i,n}$.
The interface $\Gamma_{h,n}$ further cuts the cell $\mathscr K_{i,n}$ and one
of its boundaries $\mathscr L_{ij,n}$ into two parts, which are represented as
$\mathscr K_{i,n}^\pm$ and $\mathscr L_{ij,n}^\pm$ respectively (may be an empty
set). The unit normal of $\Gamma_{\mathscr K_{i,n}}$, pointing from $\mathscr
K_{i,n}^-$ to $\mathscr K_{i,n}^+$, is denoted as $\bm n_{_{\mathscr K_{i,n}}}$.
These quantities can be readily computed from the geometries of the phase 
interface and cells. See Fig. \ref{fig:twophasemodel} for an illustration.

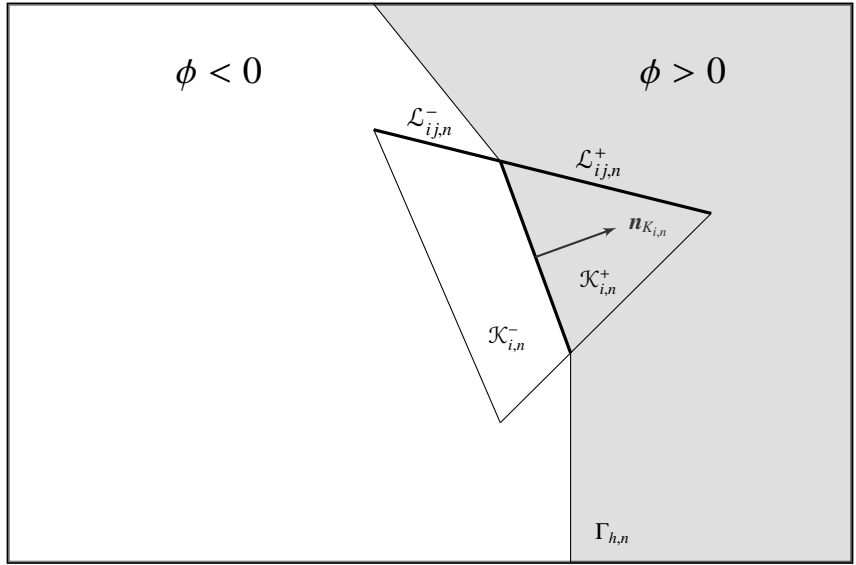
\begin{figure}[htb]
\centering
\begin{tikzpicture}[scale=1.85]
\tikzset{font={\fontsize{9pt}{12}\selectfont}}
\draw [very thick,black](0,0) rectangle (6,4);
\fill[fill=white,opacity=.5] 
(2.6,4)--(3.5,2.875)--(4,1.5)--(4,0)--(0,0)--(0,4)--cycle;
\fill[fill=lightgray,opacity=.5] 
(2.6,4)--(3.5,2.875)--(4,1.5)--(4,0)--(6,0)--(6,4)--cycle;
\draw  (2.6,4)--(3.5,2.875)--(4,1.5)--(4,0);
\draw [very thick] (3.5,2.875)--(4,1.5);
\draw [very thick] (3.5,2.875)--(2.6,3.1);
\draw [very thick] (3.5,2.875)--(5,2.5);
\draw (4,1.5)--(5,2.5);
\draw (4,1.5)--(3.5,1);
\draw (3.5,1)--(2.6,3.1);
\draw (1.5,3.5)  node[black]{\large $\phi<0$};
\draw (4.8,3.5)  node[black]{\large $\phi>0$}; 
\draw (3.55,1.6)  node[black]{$\mathscr K_{i,n}^-$};
\draw (4.2,2)  node[black]{$\mathscr K_{i,n}^+$};
\draw (3.0,3.15) node[black]{\footnotesize $\mathscr L_{ij,n}^-$};
\draw (4.2,2.85)  node[black]{\footnotesize $\mathscr L_{ij,n}^+$};
\draw [-latex',thick,black!80]
(3.75,2.1875)--++(0.4125*1.4,0.15*1.4) node[right]{$\bm n_{K_{i,n}}$};
\draw (4.3,0.2) node {$\Gamma_{h,n}$};
\end{tikzpicture}
\caption{Illustration of the fluid-solid interaction model.}
\label{fig:twophasemodel}
\end{figure}

\item[(2).] {\bf Numerical flux}

The numerical flux for the multi-medium flow is composed of two parts: the cell
edge flux and the phase interface flux. Below we explain the flux contribution
towards any given cell $\mathscr K_{i,n}$. We introduce two sets of flow
variables at the $n$-th time level
\[
\bm U_{_{\mathscr K_{i,n}}}^\pm = \left[\begin{array}{c}
\rho_{_{\mathscr K_{i,n}}}^\pm \\ 
\rho_{_{\mathscr K_{i,n}}}^\pm\bm u_{_{\mathscr K_{i,n}}}^\pm \\
E_{_{\mathscr K_{i,n}}}^\pm 
\end{array}\right],
\]
which refer to the constant states in the cell $\mathscr K_{i,n}^\pm$. Note that
the flow variables vanish if there is no corresponding medium in a given cell. 

\begin{itemize}
\item {\bf Cell edge flux}

The cell edge flux is the exchange of the flux between the same medium across
the cell boundary. For any edge $\mathscr L_{ij}$ between the cell $\mathscr
K_{i,n}$ and one of its adjacent cells $\mathscr K_{j,n}$, let $\bm n_{ij,n}$ be
the unit normal pointing from $\mathscr K_{i,n}$ to $\mathscr K_{j,n}$. The
cell edge flux across $\mathscr L_{ij,n}^\pm$ is calculated as
\begin{equation}\label{eq:cell_bound_flux}
  \hat{\bm F}_{ij,n}^\pm = \Delta t_n
  \left|\mathscr L_{ij,n}^\pm \right| 
  \hat{\bm F} \left(
    \bm U_{_{\mathscr K_{i,n}}}^\pm, \bm U_{_{\mathscr K_{j,n}}}^\pm; 
    \bm n_{ij,n} \right),
\end{equation}
where $\Delta t_n$ denotes the current time step length, and $\hat{\bm F}(\bm
U_l, \bm U_r; \bm n)$ is a consistent monotonic numerical flux along $\bm n$.
Here we adopt the local Lax-Friedrich flux
\[
\hat{\bm F}(\bm U_l, \bm U_r; \bm n) = \dfrac{1}{2} \left(\bm F(\bm U_l) +
  \bm F(\bm U_r)\right) \cdot \bm n -
  \dfrac{1}{2}\lambda \left(\bm U_r - \bm U_l\right),
\]
where $\lambda$ is the maximal signal speed over $\bm U_l$ and $\bm U_r$.

\item {\bf Phase interface flux}

The phase interface flux is the exchange of the flux between two mediums due to
the interaction of mediums at the phase interface. If $\mathscr K_{i,n}$ is an
interface cell, then the flux across the interface $\Gamma_{_{\mathscr
K_{i,n}}}$ can be approximated by
\begin{equation}
\hat{\bm F}_{_{\mathscr K_{i,n}}}^{\pm}=
\Delta t_n\left|\Gamma_{_{\mathscr K_{i,n}}}\right|
\begin{bmatrix}
0 \\ q_{_{\mathscr K_{i,n}}}^*\bm n_{_{\mathscr K_{i,n}}} \\
q_{_{\mathscr K_{i,n}}}^*u_{_{\mathscr K_{i,n}}}^*
\end{bmatrix}.
\label{eq:phaseflux}
\end{equation}
Here $q_{_{\mathscr K_{i,n}}}^*$ and $u_{_{\mathscr K_{i,n}}}^*$ are the
interface stress and normal velocity, which are obtained by applying the
approximate solver we proposed in Section \ref{sec:aps} to a local
one-dimensional Riemann problem in the normal direction of the phase interface
with initial states
\begin{align*}
\left[\rho_l,u_l,p_l, S_l\right]^\top &=
\left[\rho_{_{\mathscr K_{i,n}}}^-,~
 \bm u_{_{\mathscr K_{i,n}}}^-\cdot \bm n_{_{\mathscr K_{i,n}}},~
  p_{_{\mathscr K_{i,n}}}^-,~
  \bm n_{_{\mathscr K_{i,n}}}^\top \cdot {\bf S}_{_{\mathscr K_{i,n}}}^-
  \cdot \bm n_{_{\mathscr K_{i,n}}} \right]^\top, \\
\left[\rho_r, u_r, p_r, S_r\right]^\top &=
\left[\rho_{_{\mathscr K_{i,n}}}^+,~
 \bm u_{_{\mathscr K_{i,n}}}^+\cdot \bm n_{_{\mathscr K_{i,n}}},~
  p_{_{\mathscr K_{i,n}}}^+,~
  \bm n_{_{\mathscr K_{i,n}}}^\top \cdot {\bf S}_{_{\mathscr K_{i,n}}}^+ 
  \cdot \bm n_{_{\mathscr K_{i,n}}} \right]^\top.
 \end{align*}
Here $p_{_{\mathscr K_{i,n}}}^\pm$ and ${\bf S}_{_{\mathscr
K_{i,n}}}^\pm$ in the initial states are given through the corresponding
equations of state and deviatoric constitutive laws, respectively.
\end{itemize}

\item[(3).] {\bf Update of conservative variables}

Once the edge flux \eqref{eq:cell_bound_flux} and phase interface flux
\eqref{eq:phaseflux} are computed, the conservative variables at the $(n+1)$-th 
time level are thus assigned as
\[
\bm U^{\pm}_{_{\mathscr K_{i,n+1}}} \!=\! \left\{\begin{array}{ll}
\bm 0, & \mathscr K^\pm_{i,n+1} \!=\! \varnothing, \\
\dfrac{1}{\left|\mathscr K^\pm_{i,n+1}\right|} \!
\left(|\mathscr K^\pm_{i,n}|
\bm U^\pm_{_{\mathscr K_{i,n}}} \!+\! 
\displaystyle\sum_{\mathscr L_{ij,n}^\pm\subseteq 
\partial \mathscr K_{i,n}^\pm}
\hat{\bm F}_{ij,n}^\pm \!+\!
\hat{\bm F}^\pm_{_{\mathscr K_{i,n}}}\right), 
& \mathscr K^\pm_{i,n+1} \!\neq\! \varnothing.
\end{array}\right.
\]
\end{itemize}

Basically, the steps we present above include the overall numerical scheme,
while there are more details in the practical implementation to guarantee the
stability of the scheme. Please see \cite{Guo2016} for those details.


\section{Numerical Examples}\label{sec:num}

In this section we present some numerical examples to validate our methods,
including one-dimensional Riemann problems and two-dimensional shock impact 
problems. One-dimensional simulations are carried out on uniform interval 
meshes, while two-dimensional simulations are carried out on unstructured 
triangular meshes. 

\subsection{One-dimensional Riemann problems}

In this part, we present some numerical examples of one-dimensional Riemann
problems. The computational domain is $[0,1]$ with $400$ cells, and both the
left and right boundaries are set as outflow conditions. The reference
solutions, if mentioned, are given from either published results or computed on
a very fine mesh with $10^4$ cells.

\subsubsection{Gas-gas Riemann problem}
In the first example, we study a single-phase problem from \cite{Toro2008},
where a standard Eulerian scheme also works well with no oscillation. We take it
as a two-phase problem by artifically embedding an interface at $x=0.5$
initially. The initial values are 
\[
[\rho, u, p]^\top = \left\{
\begin{array}{ll}
[1.0, ~0, ~10^3]^\top,     &x<0.5,\\ [2mm]
[1.0, ~0, ~10^{-2}]^\top, & x > 0.5.
\end{array}
\right.
\]

\begin{figure}[htbp]
\centering
\subfloat[Density]
{\includegraphics[width=0.3\textwidth]{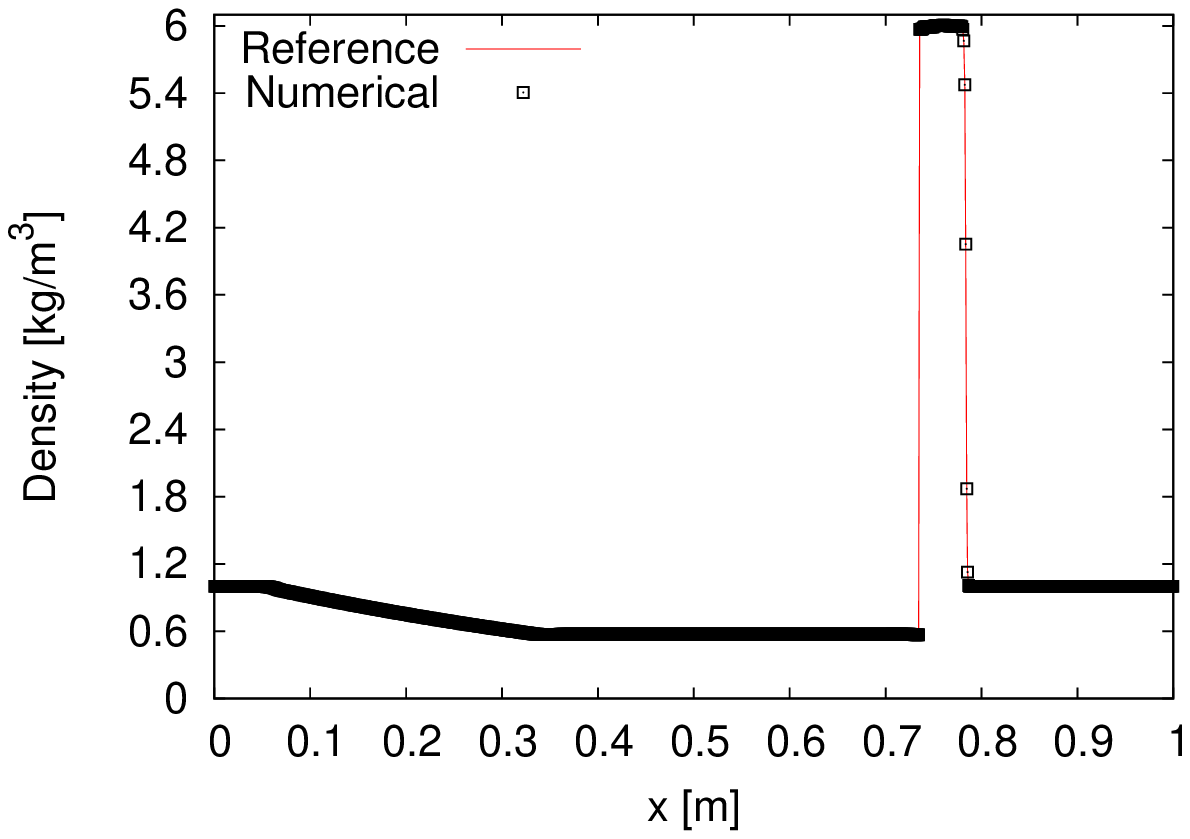}}
\subfloat[Pressure]
{\includegraphics[width=0.3\textwidth] {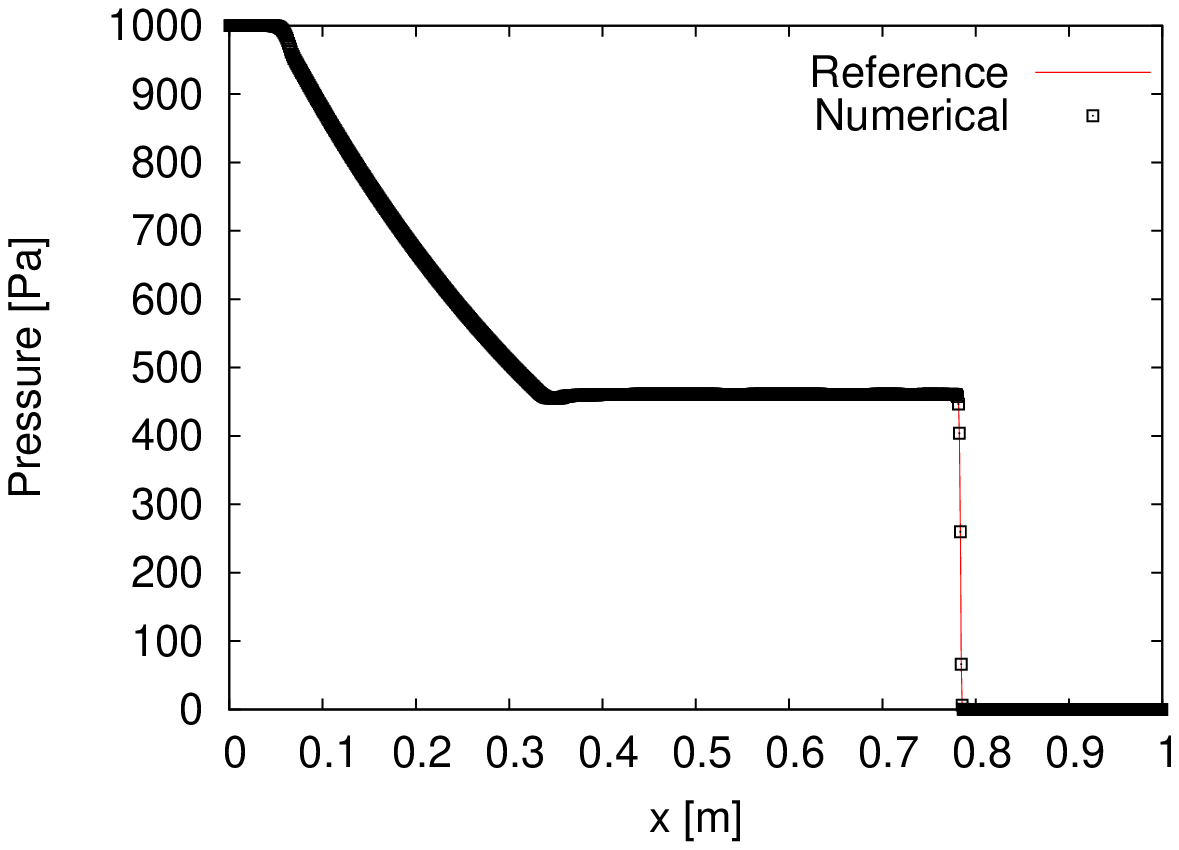}}
\subfloat[Velocity]
{\includegraphics[width=0.3\textwidth]{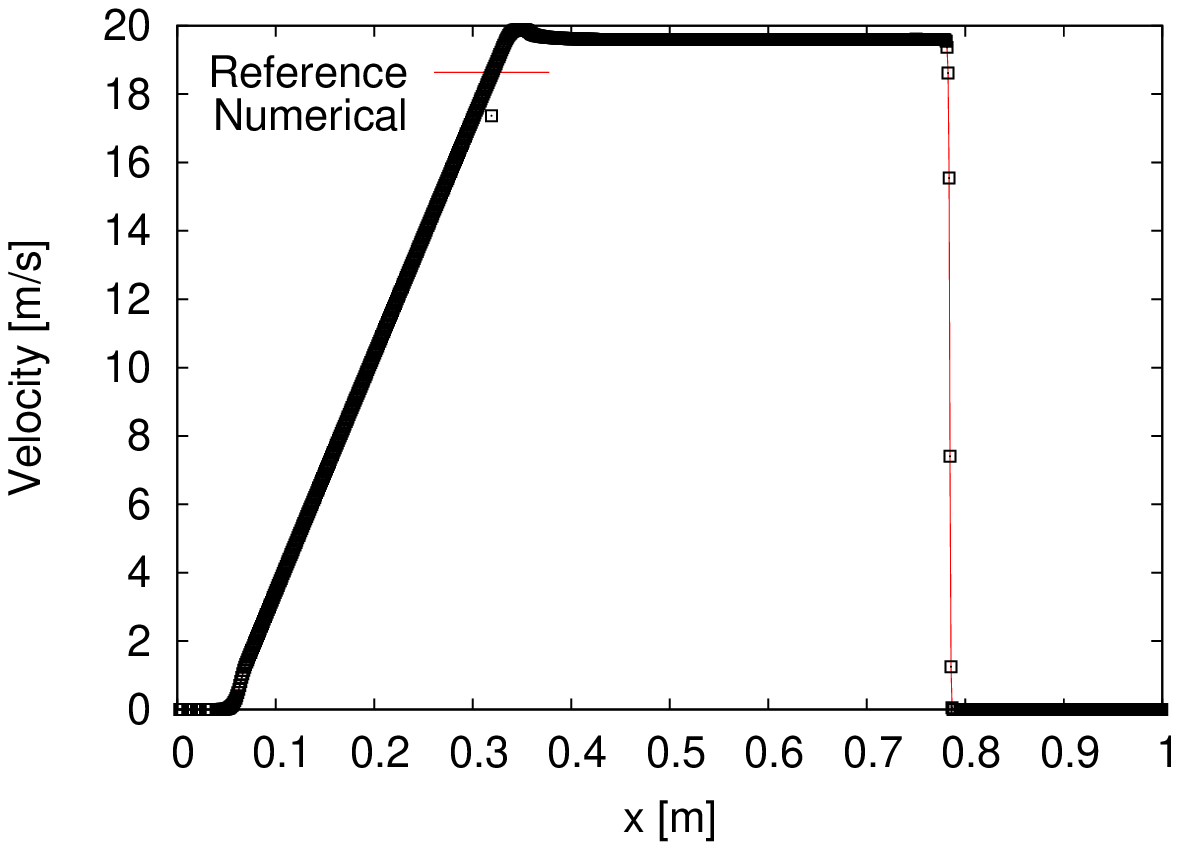}} 
\caption{Gas-gas Riemann problem.}
\label{res:gas-gas}
\end{figure}

We carry out the simulation to a final time of 0.012. Fig. \ref{res:gas-gas}
shows the comparison between numerical results and exact solutions. From the 
comparison we can see that the numerical results behave in perfect agreement 
with the exact solutions.

\subsubsection{JWL-polynomial Riemann problem}
\label{sec:jwl-water}
This example concerns the JWL-polynomial Riemann problem. The initial states
are 
\[
[\rho, u, p]^\top = \left\{
\begin{array}{ll}
[1630, ~0, ~8.3\times 10^9]^\top, &x<0.5,\\ [2mm]
[1000, ~0, ~1.0\times 10^5]^\top, & x > 0.5.
\end{array}
\right.
\]
We use the following values to describe the TNT \cite{Smith1999}:
$A_1=\unit[3.712\times10^{11}]{Pa}$,
$A_2=\unit[3.230\times10^9]{Pa}$,
$\omega=0.30$, 
$R_1=4.15$, 
$R_2=0.95$ and $\rho_0=\unit[1630]{kg/m^3}$. 
The parameters of the polynomial EOS are 
$\rho_0=\unit[1000]{kg/m^3}$,
$A_1=\unit[2.20\times10^9]{Pa}$, $A_2=\unit[9.54\times10^9]{Pa}$,
$A_3=\unit[1.45\times10^{10}]{Pa}$, $B_0=B_1=0.28$, 
$T_1=\unit[2.20\times10^9]{Pa}$ and $T_2=0$ \cite{Jha2014}. 

\begin{figure}[htbp]
\centering
\subfloat[Density]
{\includegraphics[width=0.3\textwidth]{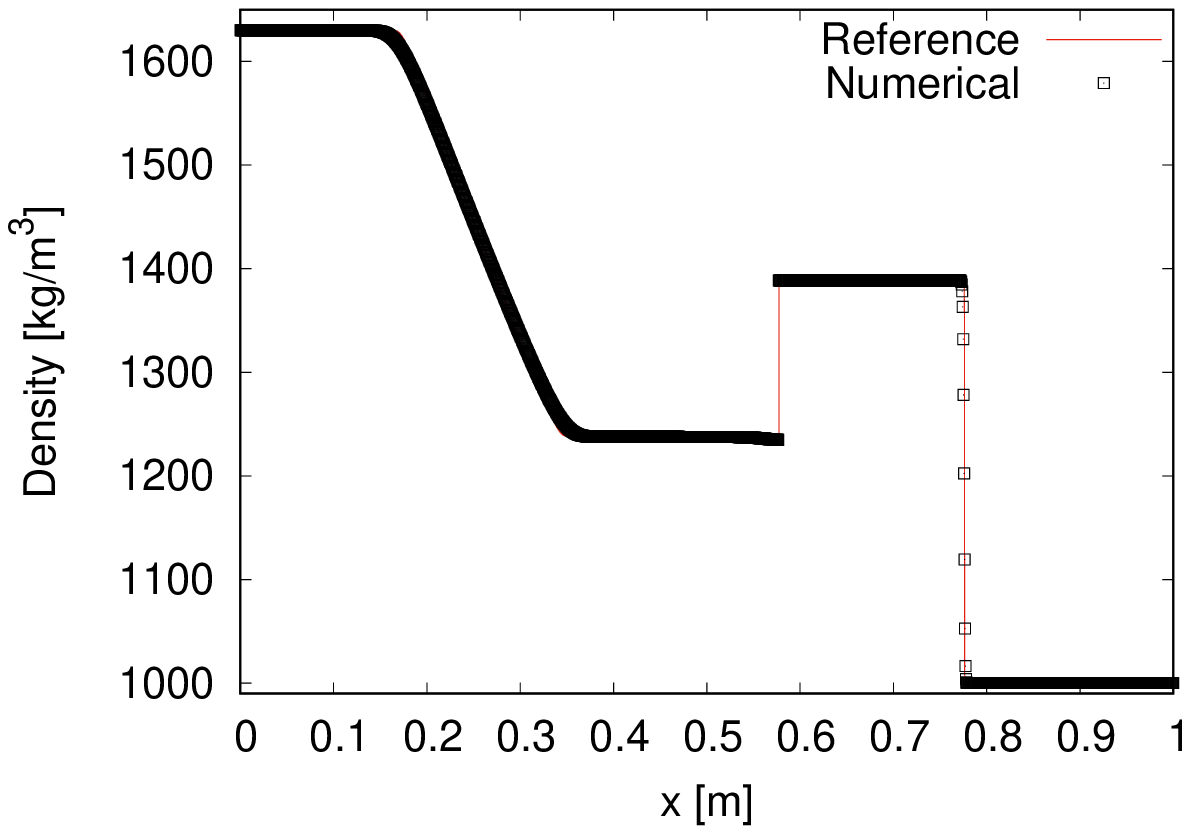}}
\subfloat[Pressure]
{\includegraphics[width=0.3\textwidth] {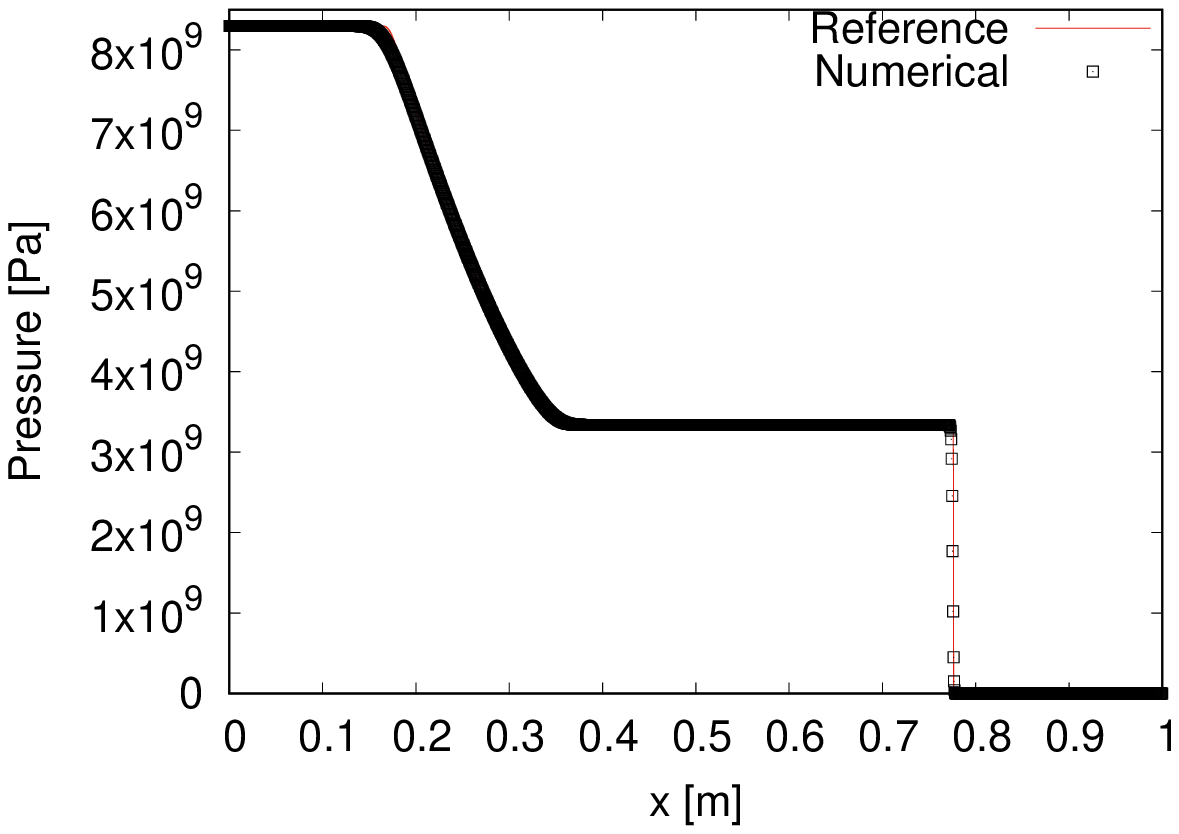}}
\subfloat[Velocity]
{\includegraphics[width=0.3\textwidth]{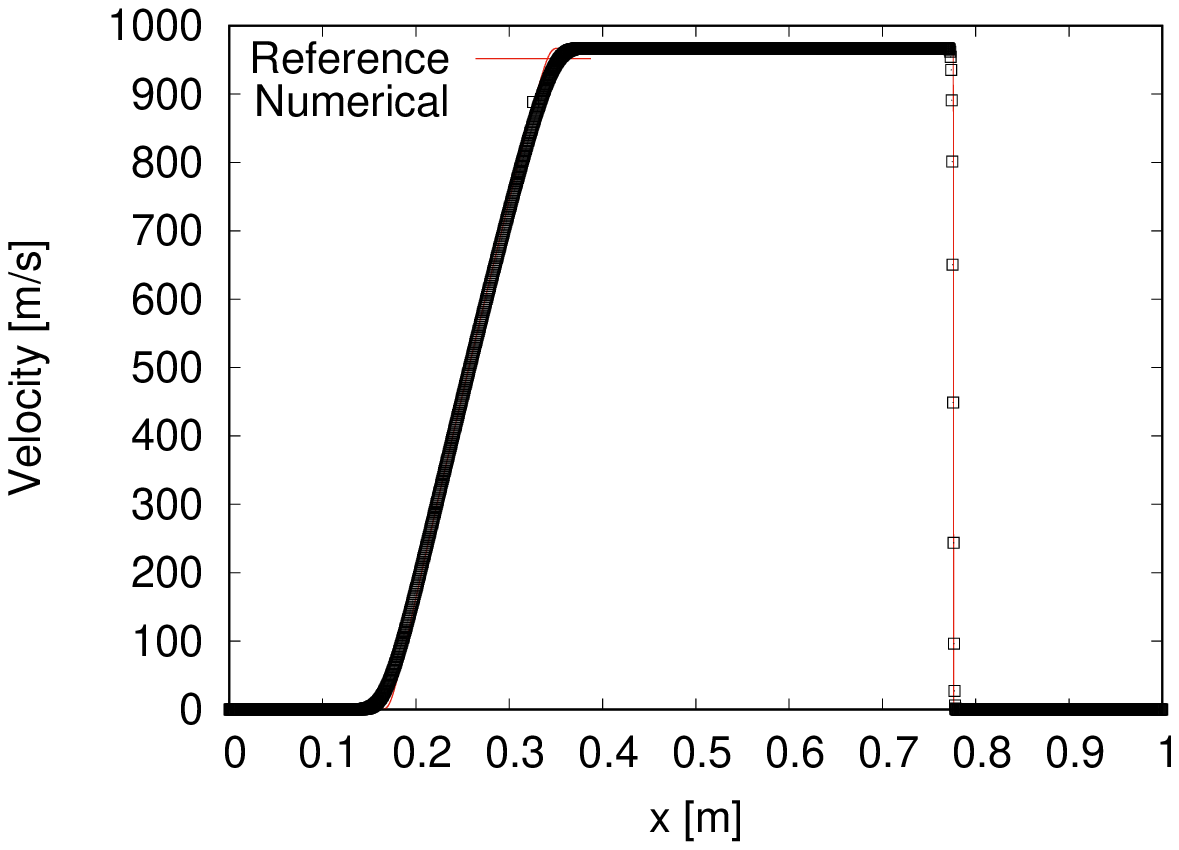}} 
\caption{JWL-polynomial Riemann problem.}
\label{res:jwl-ploy}
\end{figure}

The result at $t=8.0\times 10^{-5}$ is shown in Fig. \ref{res:jwl-ploy}, where
we can observe that both the interface and shock are captured well without
spurious oscillation.

\subsubsection{Gavrilyuk's elastic solid Riemann problem}
\label{sec:e-e}
In this problem, we simulate an elastic solid Riemann problem 
\cite{Gavrilyuk2008}. The hydrostatic pressure of the solid is described by the
stiffened gas EOS with parameters $\gamma=4.4$, $p_\infty=6\times 10^6$ Pa.
And the deviatoric component obeys the Hooke's law, whose elastic shear modulus 
is $\mu^{_\mathscr E}=10^{10}$ Pa. The initial values are given by
\[
 [\rho, u, p]^\top = \left\{
  \begin{array}{ll}
    [10^3, ~100, ~10^5]^\top, & x < 0.5,\\[2mm]
    [10^3, ~-100, ~10^5]^\top, & x > 0.5.
  \end{array}
 \right.
\]

The comparison between our numerical results and reference solutions at
$6.1\times 10^{-5}$ is shown in Fig. \ref{res:e-e}, from which we can see that 
our results agree well with the reference solutions, and there is no 
oscillation in the vicinity of phase interface and shock waves.

\begin{figure}[htbp]
\centering
\subfloat
{\includegraphics[width=0.3\textwidth]
  {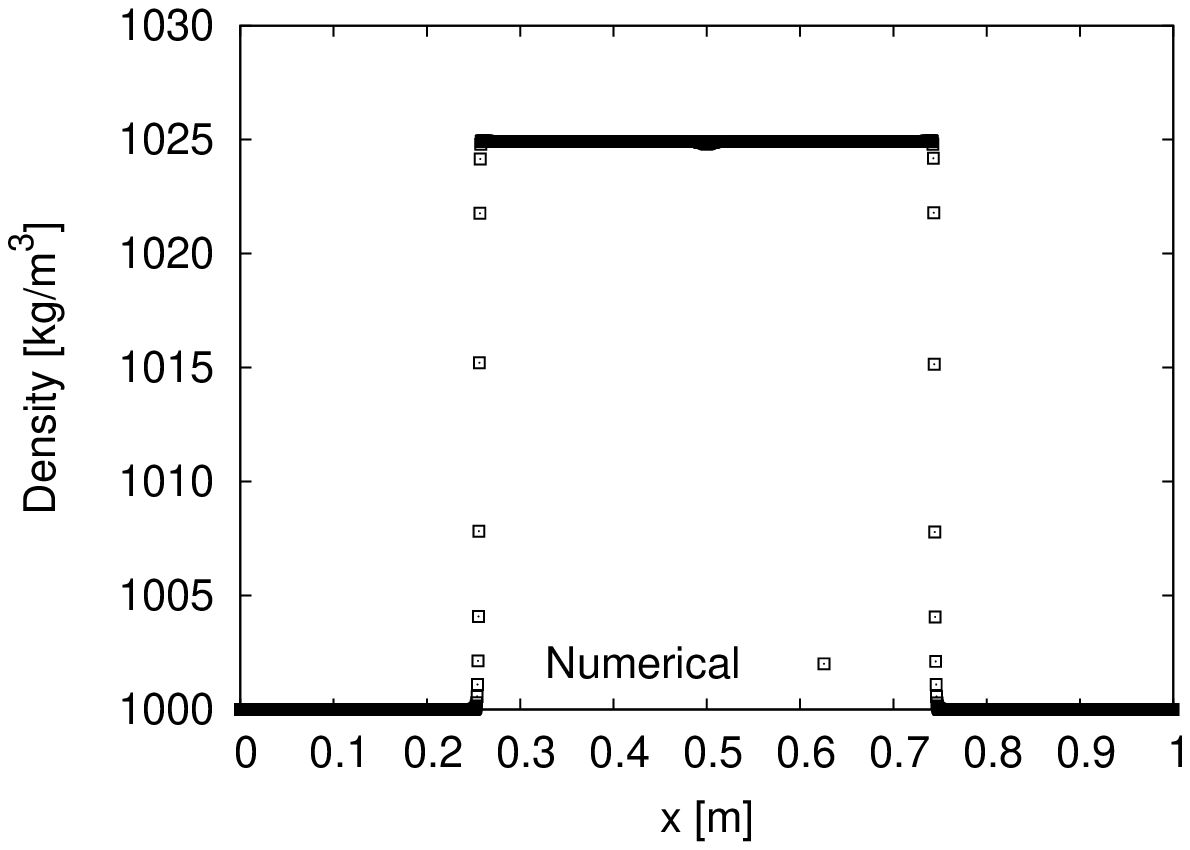}}
\subfloat
{\includegraphics[width=0.3\textwidth] 
  {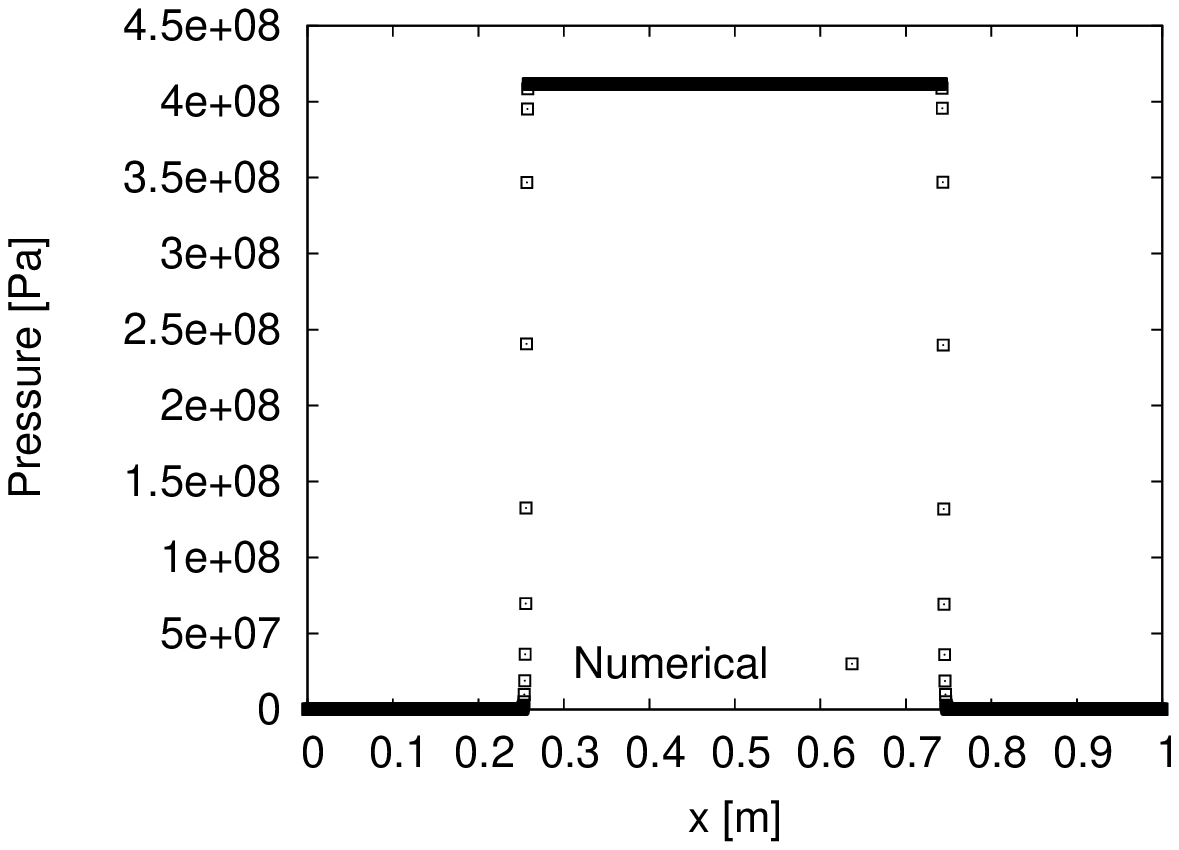}}
\subfloat
{\includegraphics[width=0.3\textwidth]
  {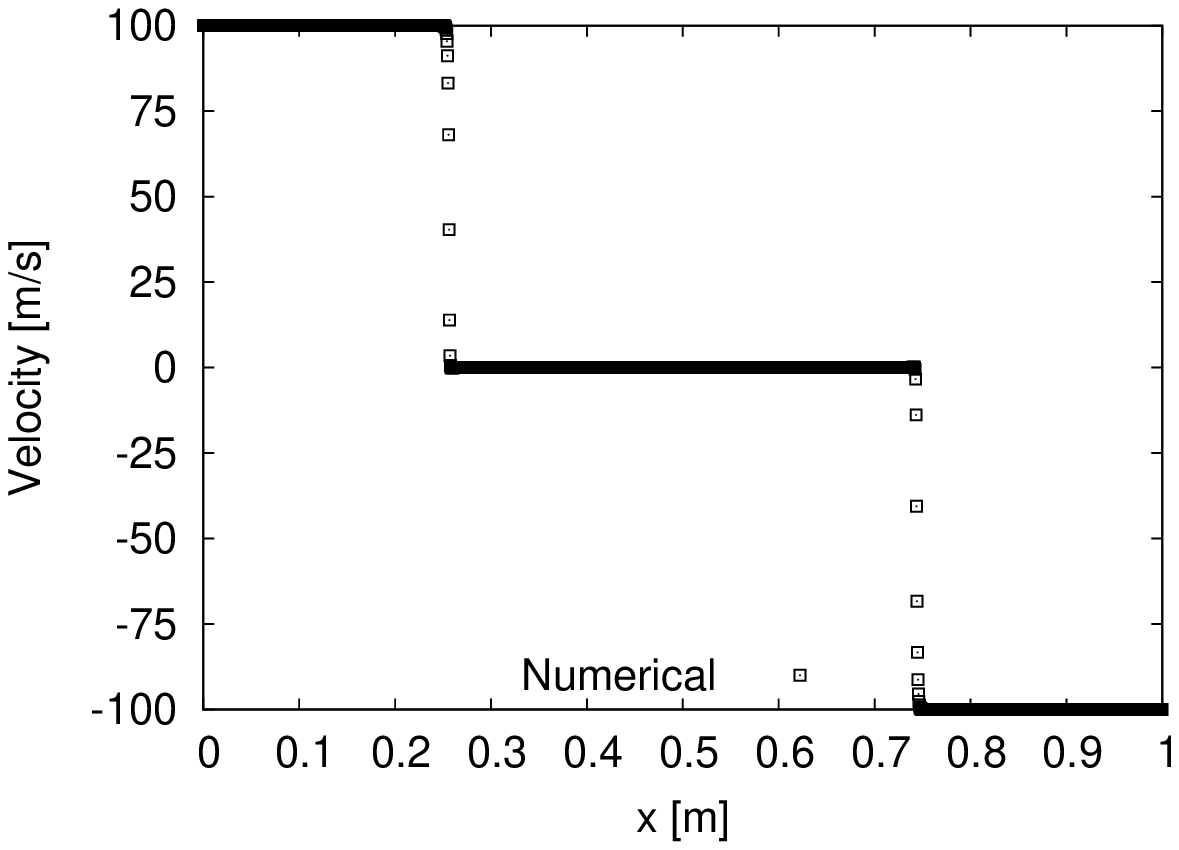}} \\
\addtocounter{subfigure}{-3}
\subfloat[Density]
{\includegraphics[width=0.3\textwidth]
  {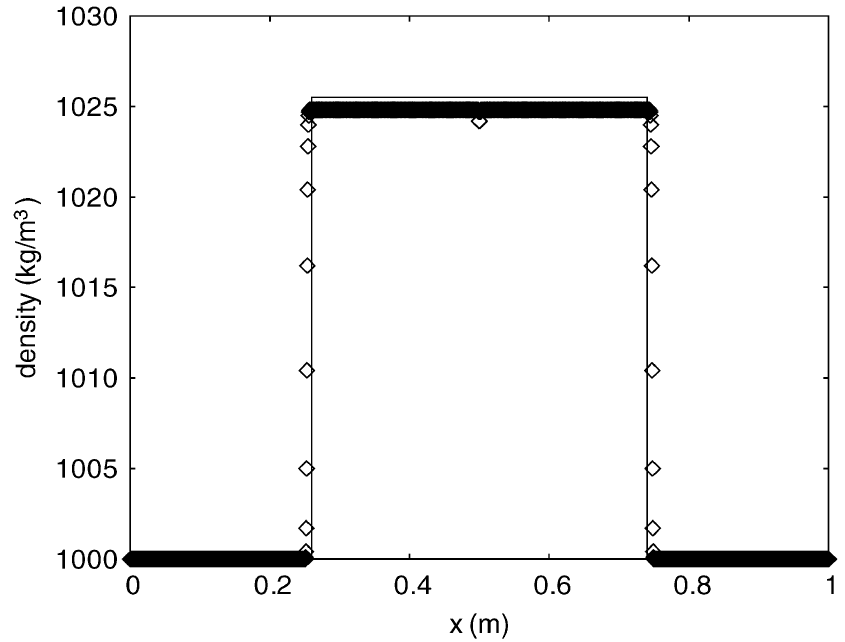}}
\quad
\subfloat[Pressure]
{\includegraphics[width=0.29\textwidth]
  {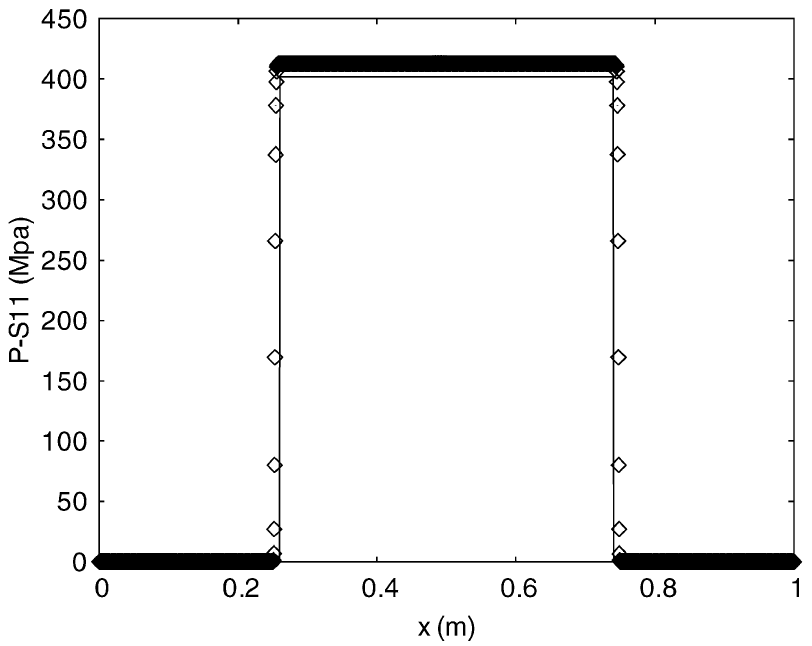}}
\quad
\subfloat[Velocity]
{\includegraphics[width=0.29\textwidth]
  {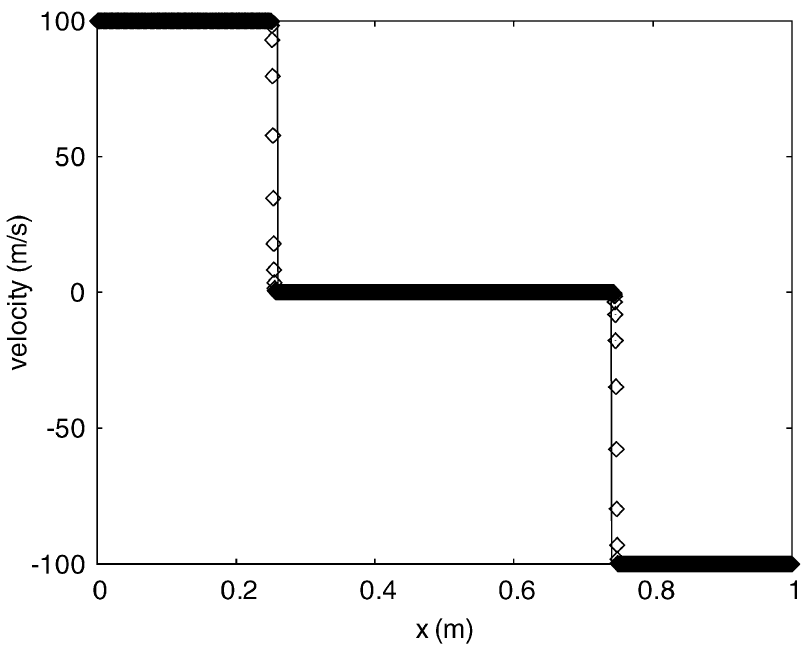}} \\
\caption{Gavrilyuk's elastic solid Riemann problem (top row: our results, bottom
row: results from Gavrilyuk \textit{et al.} \cite{Gavrilyuk2008}).} 
\label{res:e-e}
\end{figure}

\subsubsection{JWL-elastic solid Riemann problem}
\label{res:jwl-elastic}
In this problem, we simulate a JWL-elastic solid Riemann problem. The JWL
EOS has the following parameter: $A_1=\unit[8.545\times10^{11}]{Pa}$, 
$A_2=\unit[2.050\times10^{10}]{Pa}$, $\omega=0.25$, $R_1=4.6$, $R_2=1.35$, 
$\rho_0=\unit[1840]{kg/m^3}$. The elastic solid has the same constitutive law  
as Section \ref{sec:e-e} with $\beta^{_\mathscr E}=10^{14}~\mbox{Pa} 
\cdot\mbox{kg/m}^3$. The initial values are
\[
 [\rho, u, p]^\top = \left\{
  \begin{array}{ll}
    [1630, ~0, ~9.2\times 10^9]^\top, & x < 0.5, \\ [2mm]
    [7800, ~0, ~10^5]^\top, & x > 0.5.
  \end{array}
 \right.
\]

\begin{figure}[htbp]
\centering
\subfloat[Density]
{\includegraphics[width=0.3\textwidth]{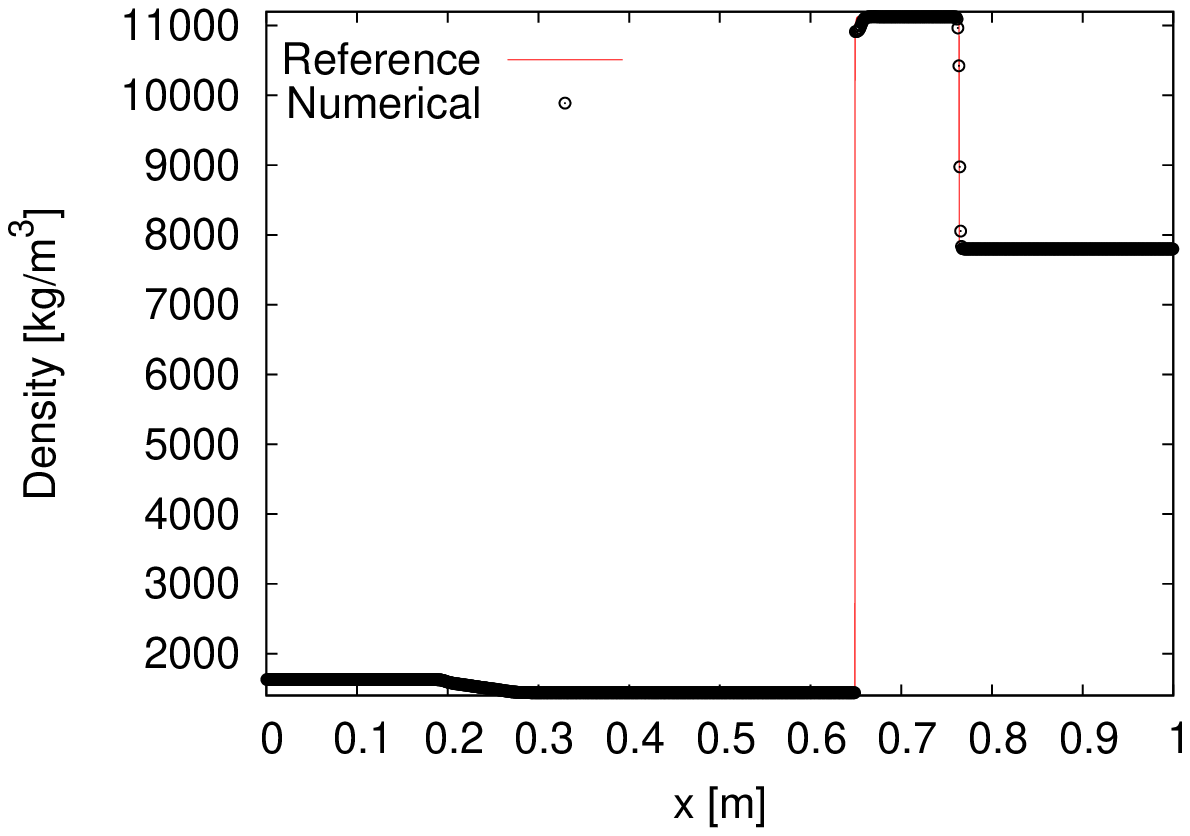}}
\subfloat[Pressure]
{\includegraphics[width=0.3\textwidth] {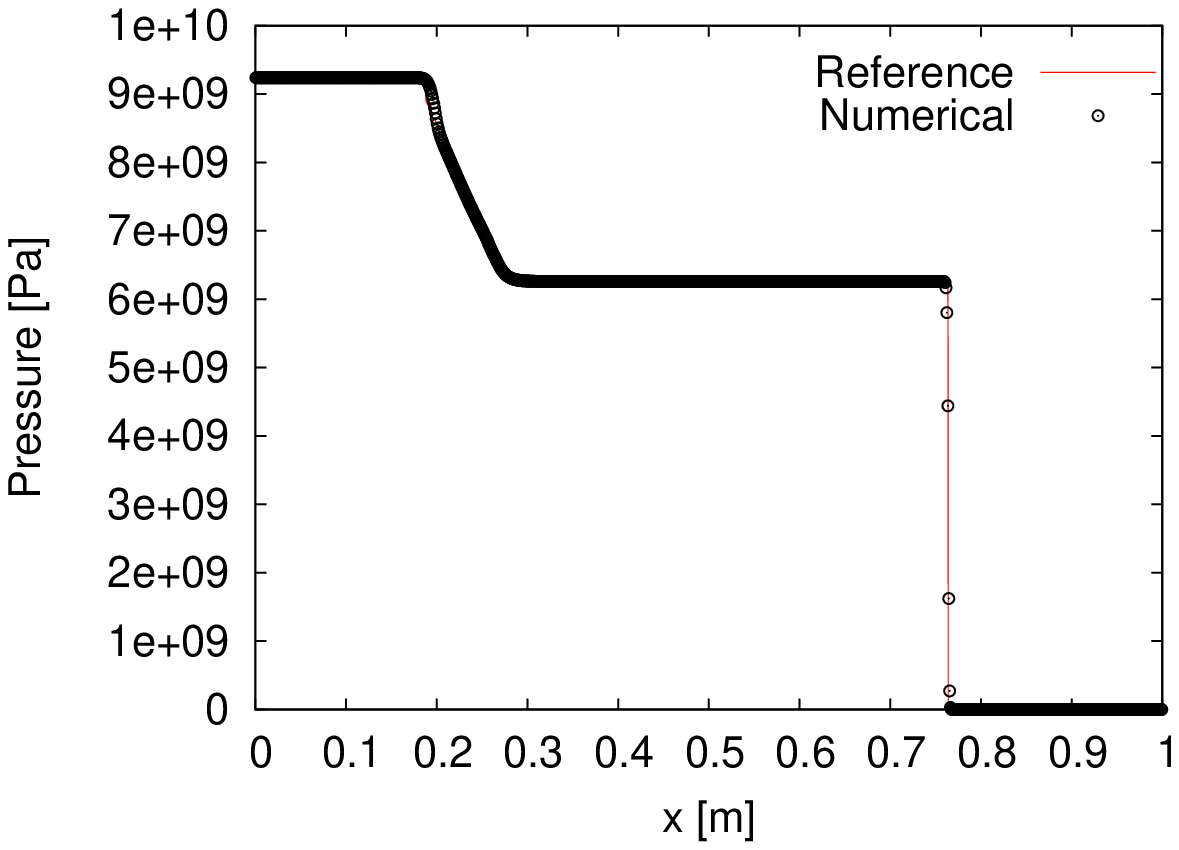}}
\subfloat[Velocity]
{\includegraphics[width=0.3\textwidth]{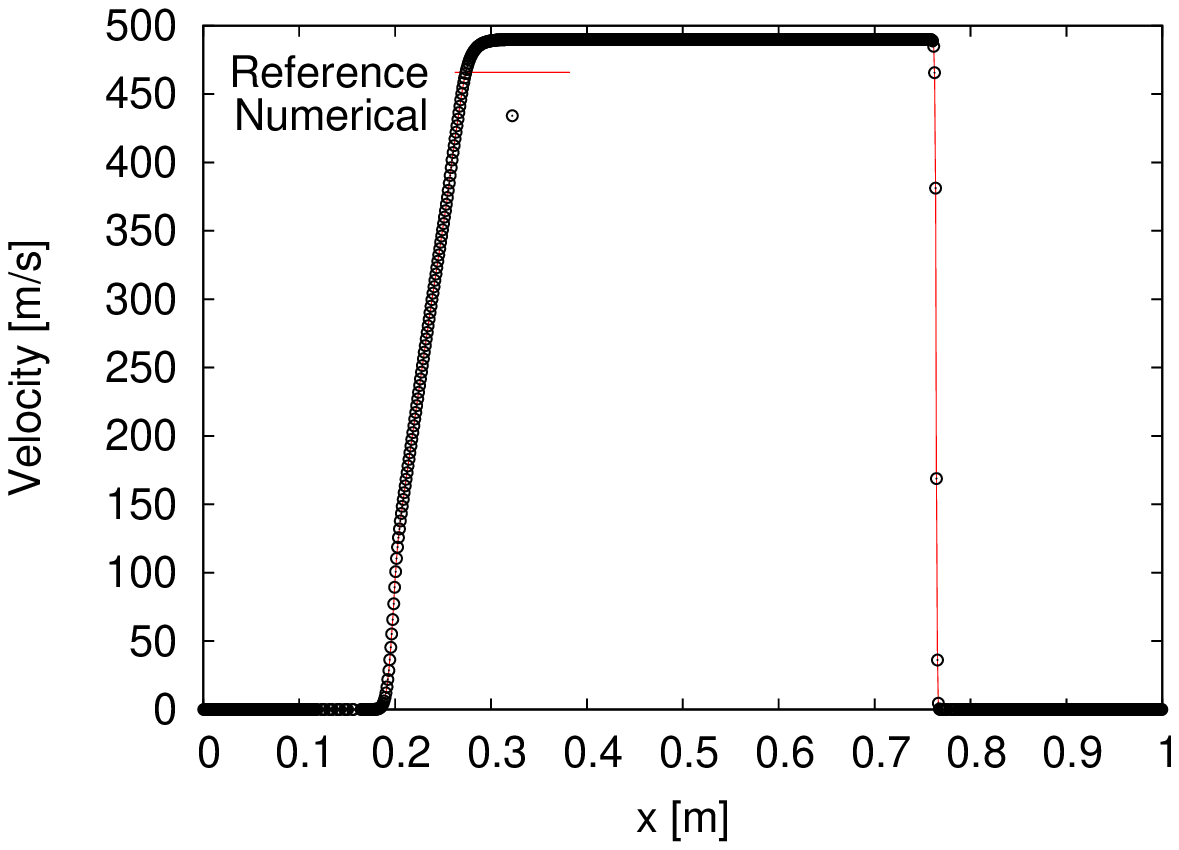}} 
\caption{JWL-elastic solid Riemann problem.}
\label{res:jwl-e}
\end{figure}

The computation terminates at $10^{-4}$. Fig. \ref{res:jwl-e} displays the
results of our numerical scheme and the exact solutions, where we can see that
there is no non-physical pressure and velocity across the contact discontinuity
in our numerical scheme.

\subsubsection{Perfectly elastoplastic solid Riemann problem}
\label{res:ep-ep}
In this problem, we extend our methods to simulate the perfectly elastoplastic 
solid-solid Riemann problem \cite{Liu2008}. We take the Murnagham EOS 
\eqref{eq:murnagham} to describe the hydrostatic pressure of the solid, and set
the following parameters for the solid: $\mu^{_\mathscr E}=\unit[8.53\times
10^5]{Pa},~ \mu^{_\mathscr P}=0,~ K=\unit[2.225\times 10^6]{Pa},~
Y^{_\mathscr E}=\unit[6.50\times 10^3]{Pa}$, $Y^{_\mathscr P}=0$. The initial 
values are given by
\[
 [\rho, u, p]^\top = \left\{
  \begin{array}{ll}
    [7.8, ~10, ~1.0]^\top, & x < 0.5,\\[2mm]
    [7.8, ~-5, ~1.0]^\top, & x > 0.5.
  \end{array}
 \right.
\]

The comparison between our numerical results and reference solutions at
$6.751\times 10^{-4}$ is shown in Fig. \ref{res:rm_ep_ep}. Each solid has two
nonlinear waves, the leading elastic shock wave and tailing plastic shock
wave, and there is no oscillation in the vicinity of phase interface and shock
waves. Both the elastic and plastic shock waves are captured correctly.
\begin{figure}[htbp]
\hspace{-2mm}
\centering
\subfloat
{\includegraphics[width=0.3\textwidth]{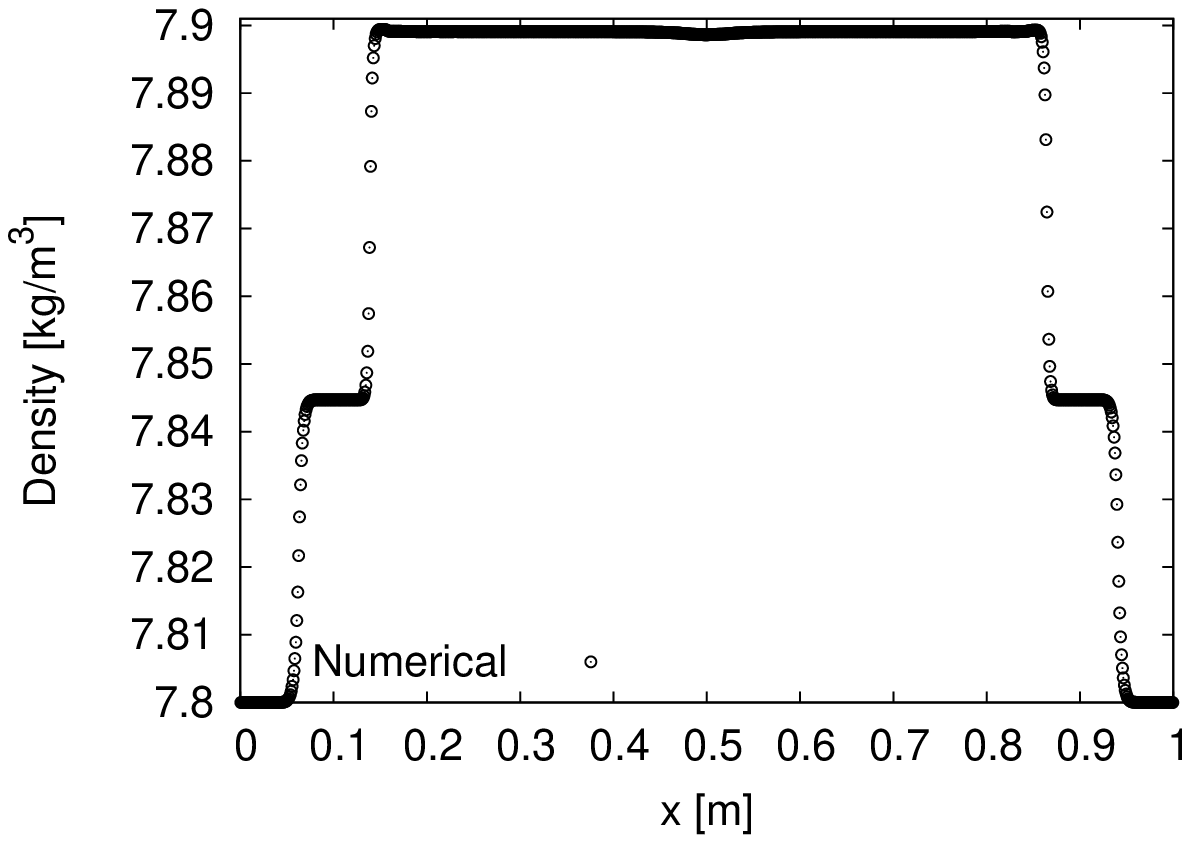}}
\subfloat
{\includegraphics[width=0.3\textwidth] {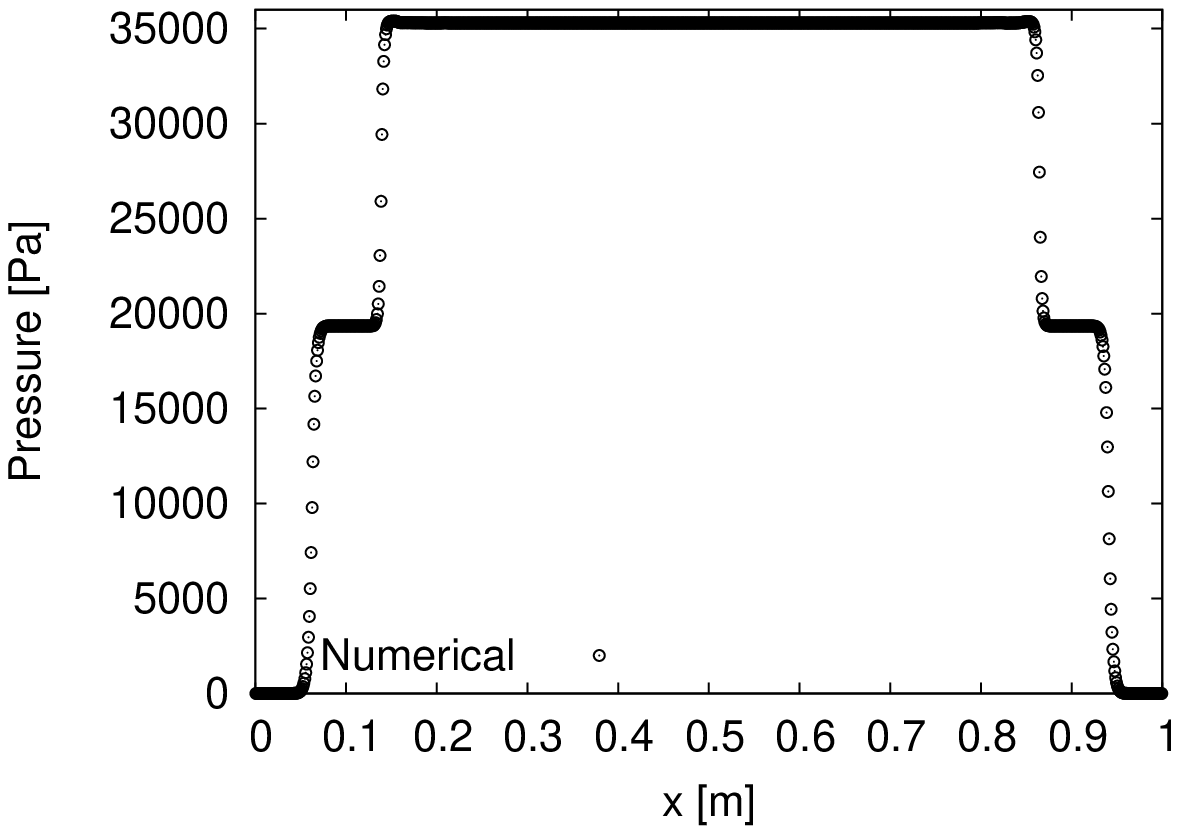}}
\subfloat
{\includegraphics[width=0.3\textwidth]{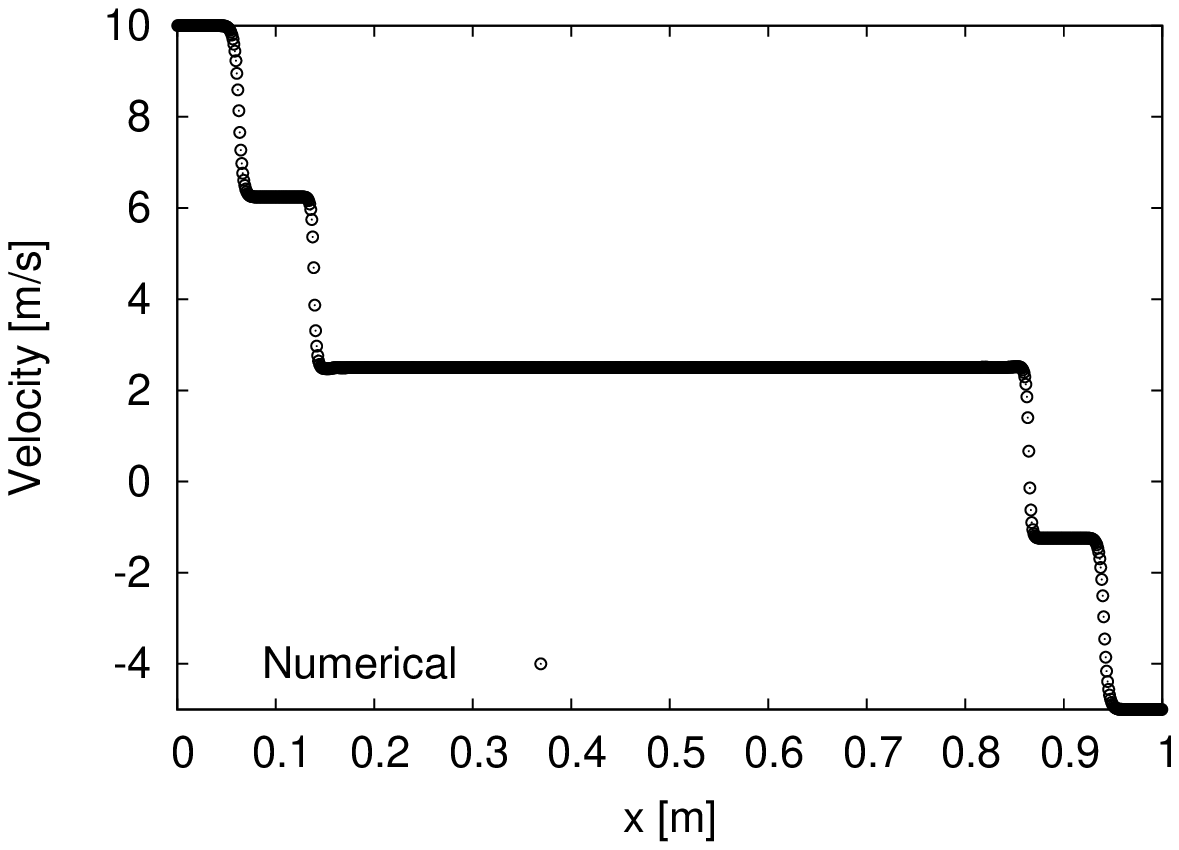}} \\
\addtocounter{subfigure}{-3}
\vspace{-1mm}
\subfloat[Density]
{\includegraphics[width=0.28\textwidth]{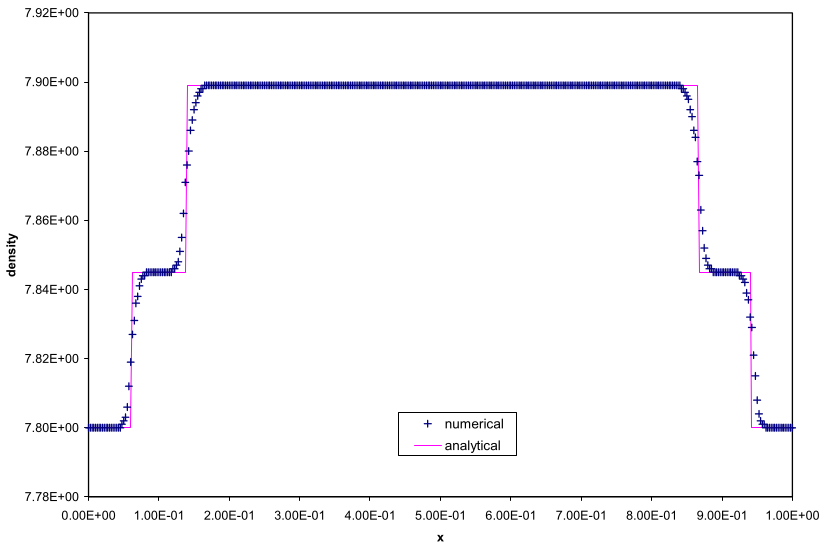}}
\quad
\subfloat[Pressure]
{\includegraphics[width=0.28\textwidth]{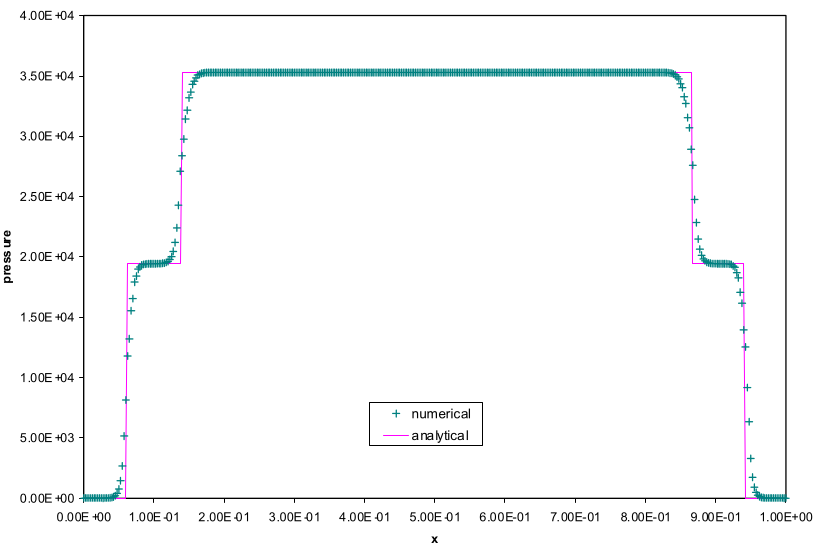}}
\quad
\subfloat[Velocity]
{\includegraphics[width=0.28\textwidth]{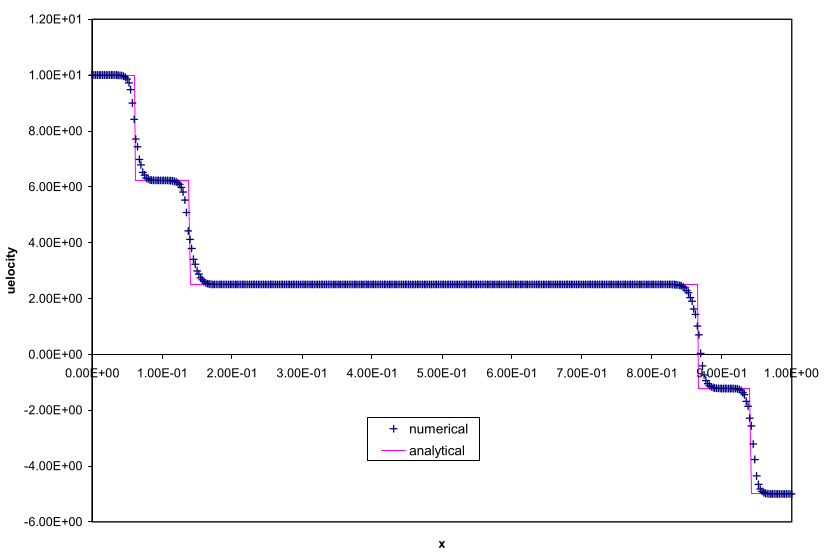}}
\caption{Perfectly elastoplastic solid Riemann problem (top panel: our numerical
results, bottom panel: reference solutions from T.G. Liu \cite{Liu2008}).}
\label{res:rm_ep_ep}  
\end{figure}

\subsubsection{Hydro-elastoplastic solid Riemann problem}
\label{res:fp-fp}
In this problem, we extend our methods to simulate the hydro-elastoplastic solid 
Riemann problem \cite{Liu2008}, which has the same initial conditions and 
parameters as Section \ref{res:ep-ep} except 
$Y^{_\mathscr P}=\unit[9.75\times 10^3]{Pa}$, 
$\mu^{_\mathscr P}=\mu^{_\mathscr E}/2$.

Our numerical results at $6.751\times 10^{-4}$ is shown in Fig. 
\ref{res:rm_fp_fp}. Due to the discrepency of the deviatoric constitutive law, 
each solid has three nonlinear waves, a leading elastic shock wave, an 
intermediate plastic shock wave and a tailing fluid shock wave. From the 
comparison, we can see that there is no oscillation in the vicinity of phase 
interface and shock waves.
\begin{figure}[htbp]
\subfloat[Density]
{\includegraphics[width=0.28\textwidth]{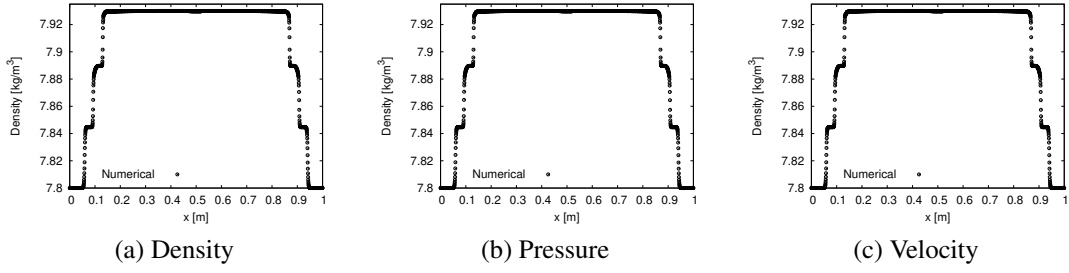}}
\quad
\subfloat[Pressure]
{\includegraphics[width=0.28\textwidth]{fp-fp/dens.eps}}
\quad
\subfloat[Velocity]
{\includegraphics[width=0.28\textwidth]{fp-fp/dens.eps}}
\caption{Hydro-elastoplastic solid Riemann problem.}
\label{res:rm_fp_fp}  
\end{figure}

\subsection{Two-dimensional applications}

In this part, we present a few two-dimensional problems in engineering
applications, which are carried out on triangular meshes, including gas-bubble 
interaction, blast wave reflection, implosion compression and high speed impact 
problems.

\subsubsection{Gas-bubble interaction problem}
In this problem, we simulate a gas-bubble interaction problem from Hass 
\cite{Haas1987interaction, ullah2013towards, quirk1996dynamics}, which has 
been widely used as a benchmark problem for validations of numerical schemes. 
The computational domain for our simulation is shown in Fig. 
\ref{fig:shock_bubble_model}. A cylindrical bubble with diameter 50mm is placed 
in the middle of the square shock tube filled with air. A planar weak shock of 
Mach number 1.22 vertical to the walls of the shock tube is produced on the 
right of the bubble, and it propagates towards and hits the bubble. The 
behaviors of the helium bubble and air are modeled by the ideal gas EOS, and 
the initial parameters are presented in Tab. \ref{tab:shock_bubble_intial}. The 
reflective wall boundary conditions are presented on the top and bottom, and 
outflow conditions are prescribed on the left and right ends of the domain. 
 
 \begin{figure}[htbp]
\centering
\includegraphics[width=0.6\textwidth]
{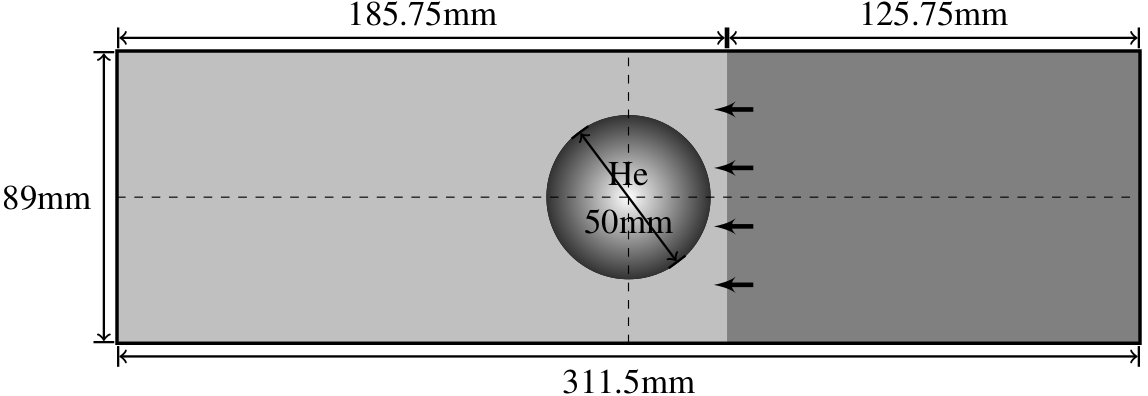}
\caption{Model of gas-bubble interaction problem.}
\label{fig:shock_bubble_model}
\end{figure}

\begin{table}[htbp]
\caption{Initial parameters for gas-bubble problem.}
\label{tab:shock_bubble_intial}
\centering
\begin{tabular}{ccccc}
\toprule 
Parameters & $\rho(\mbox{kg/m}^3)$ & $u(\mbox{m/s})$ &$p(\mbox{Pa})$ & $\gamma$\\
\midrule
Helium(bubble) &  0.2228 & 0 & 101325 &  1.648  \\
\midrule
Air(Before Shock) &  1.2250 & 0 & 101325 &  1.400  \\
\midrule
Air(After Shock) &  1.6861 & -113.534 & 159059 &  1.400  \\
\bottomrule
\end{tabular}
\end{table}

We present the contour images of the numerical density at the times 23 $\mu$s, 
43 $\mu$s, 53 $\mu$s, 66 $\mu$s, 75 $\mu$s, 102 $\mu$s, 260 $\mu$s, 445 $\mu$s, 
674 $\mu$s and 983 $\mu$s, and compare them with the experimental shadowgraphs 
picked from \cite{Haas1987interaction} at times 32 $\mu$s, 53 $\mu$s, 62 
$\mu$s, 72 $\mu$s, 82 $\mu$s, 102 $\mu$s, 245 $\mu$s, 427 $\mu$s, 674 $\mu$s 
and 983 $\mu$s. As is seen from the comparison, our numerical results are 
qualitatively in good agreement with the experiment. Our numerical simulation 
provides clear images for the severely deformed bubble, especially from 427 
$\mu$s to 983 $\mu$s, which show the ability of our methods in dealing with 
the large deformation of phase interface.

\begin{figure}[htbp]
\centering
\subfloat
{\includegraphics[width=0.24\textwidth]
{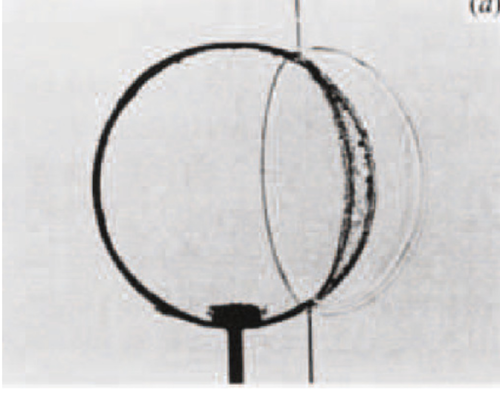}}\quad
\subfloat
{\includegraphics[width=0.7\textwidth]
{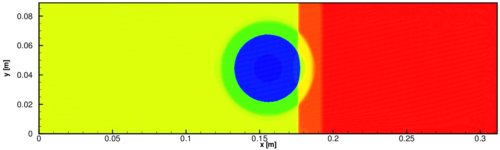}}\\
\vspace{5mm}
\subfloat
{\includegraphics[width=0.24\textwidth]
{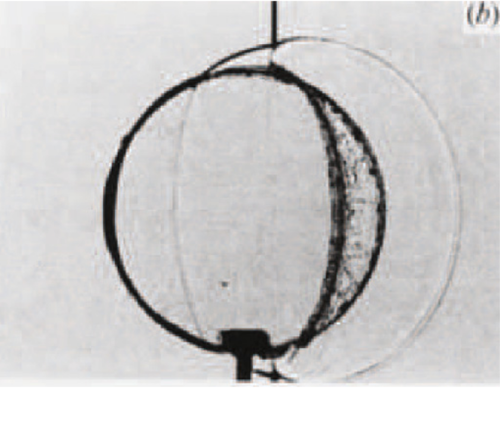}}\quad
\subfloat
{\includegraphics[width=0.7\textwidth]
{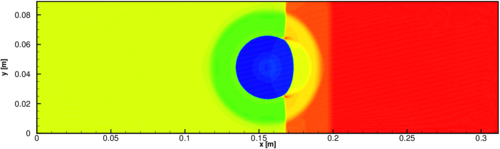}}\\
\vspace{5mm}
\subfloat
{\includegraphics[width=0.24\textwidth]
{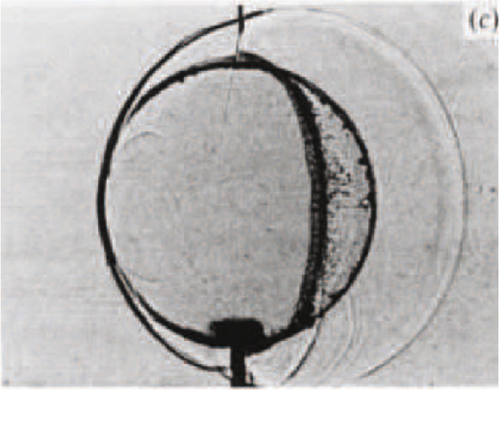}}\quad
\subfloat
{\includegraphics[width=0.7\textwidth]
{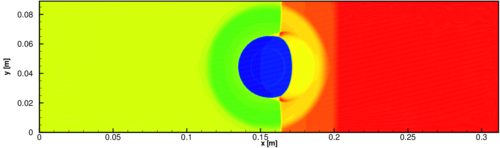}} \\
\vspace{5mm}
\subfloat
{\includegraphics[width=0.24\textwidth]
{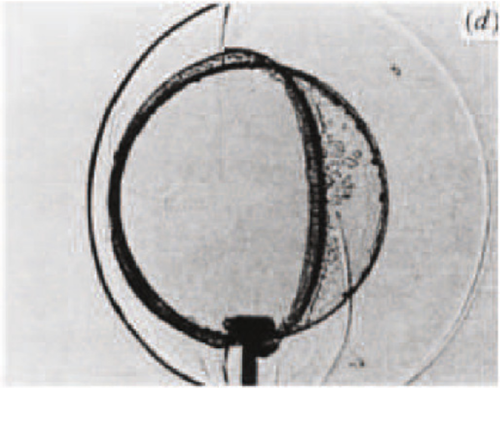}}\quad
\subfloat
{\includegraphics[width=0.7\textwidth]
{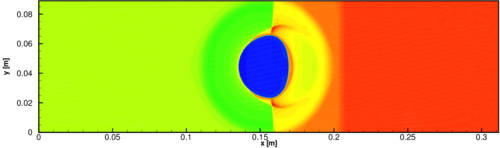}} \\
\vspace{5mm}
\subfloat
{\includegraphics[width=0.24\textwidth]
{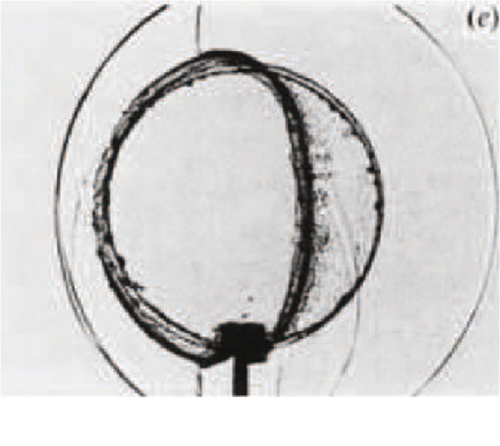}}\quad
\subfloat
{\includegraphics[width=0.7\textwidth]
{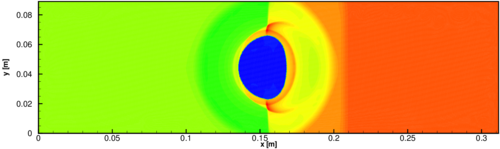}} 
\caption{Gas-bubble interaction problem. 
The first row: 32$\mu\mbox{s}$,~53$\mu\mbox{s}$,
62$\mu\mbox{s}$,~72$\mu\mbox{s}$, 82$\mu\mbox{s}$;~
The second row: 23$\mu\mbox{s}$,~43$\mu\mbox{s}$,
53$\mu\mbox{s}$,~66$\mu\mbox{s}$, 75$\mu\mbox{s}$.}
\label{res:shock-bubble1}
\end{figure}

\begin{figure}[htbp]
\centering
\subfloat
{\includegraphics[width=0.24\textwidth]
{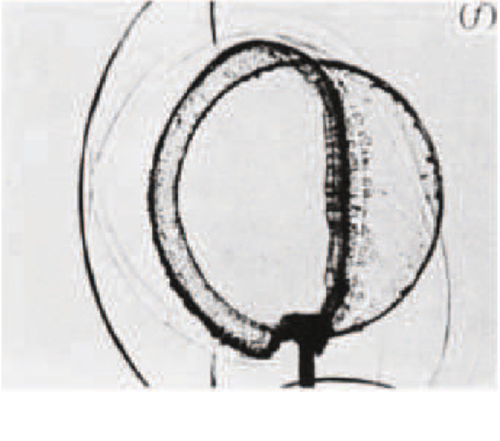}}\quad
\subfloat
{\includegraphics[width=0.7\textwidth]
{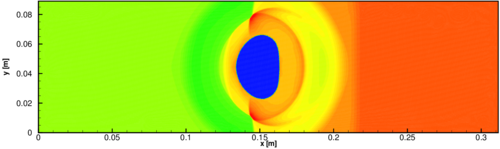}}\\
\vspace{5mm}
\subfloat
{\includegraphics[width=0.24\textwidth]
{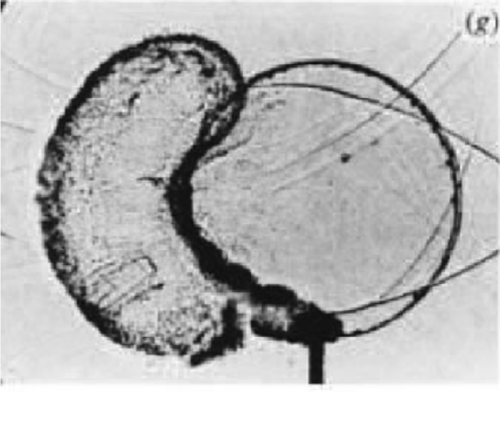}}\quad
\subfloat
{\includegraphics[width=0.7\textwidth]
{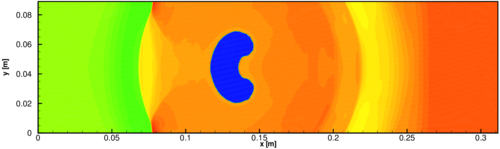}}\\
\vspace{5mm}
\subfloat
{\includegraphics[width=0.24\textwidth]
{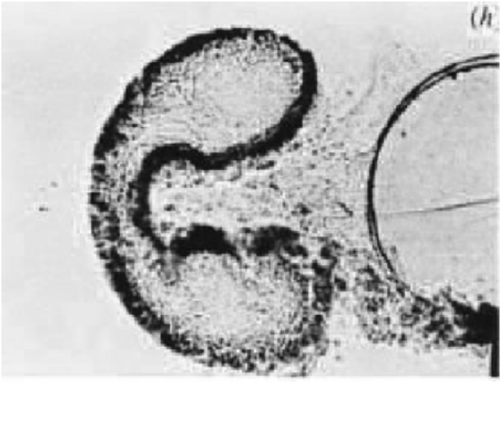}}\quad
\subfloat
{\includegraphics[width=0.7\textwidth]
{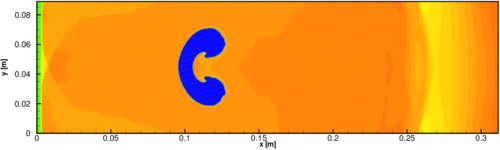}}\\
\vspace{5mm}
\subfloat
{\includegraphics[width=0.24\textwidth]
{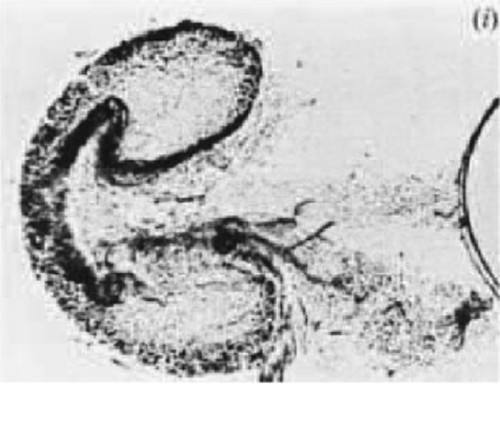}}\quad
\subfloat
{\includegraphics[width=0.7\textwidth]
{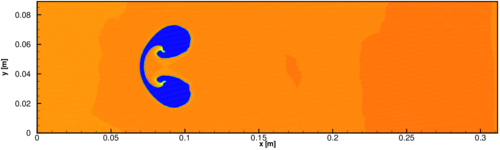}}\\
\vspace{5mm}
\subfloat
{\includegraphics[width=0.24\textwidth]
{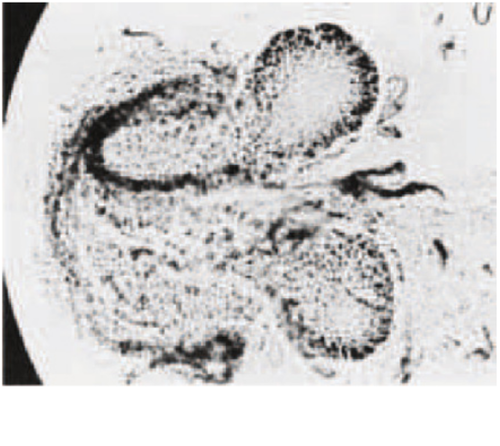}}\quad
\subfloat
{\includegraphics[width=0.7\textwidth]
{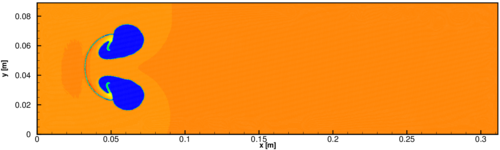}}
\caption{Gas-bubble interaction problem. 
The first row: 102$\mu\mbox{s}$,~245$\mu\mbox{s}$,
427$\mu\mbox{s}$,~674$\mu\mbox{s}$, 983$\mu\mbox{s}$;~
The second row: 102$\mu\mbox{s}$,~260$\mu\mbox{s}$,
445$\mu\mbox{s}$,~674$\mu\mbox{s}$, 983$\mu\mbox{s}$.}
\label{res:shock-bubble2}
\end{figure}

\subsubsection{Blast wave reflection of TNT explosion}
In this problem, we simulate a TNT explosion problem, where the blast wave is  
reflected by a rigid surface near the explosion center. We use this example to 
assess the isotropic behavior of TNT explosion in a computational domain $0\le 
r\le \unit[1]{m},~\unit[2]{m}\le z\le \unit[8]{m}$.  The air is modeled by the 
ideal gas EOS with adiabatic exponent $\gamma=1.4$, and the TNT is modeled by 
the JWL EOS with the same parameters as Section \ref{sec:jwl-water}. The initial 
conditions are: $\rho=\unit[1630]{kg/m^3}$, $u=\unit[0]{m/s}$, 
$p=\unit[9.5\times 10^9]{Pa}$ for the TNT, and $\rho=\unit[1.29]{kg/m^3}$, 
$u=\unit[0]{m/s}$, $p=\unit[1.013\times 10^5]{Pa}$ for the air. The initial 
interface is a sphere of radius $\unit[0.0527]{m}$ centered at the height 
$z=\unit[5]{m}$. All of the physical boundaries are set as rigid walls.

The results of shock produced by the high explosives are shown in Fig.  
\ref{res:tnt_ref_cont}. From here we can see that the shock wave propagates as 
an expansive spherical surface in the earlier period. When the spherical shock 
wave impinges on the rigid surface, it will be reflected firstly and propagate 
along the rigid wall simultaneously. When the incident angle exceeds the 
limit, the reflective wave switches from regular to irregular, and a Mach blast 
wave occurs. The shock parameters, shown in Fig. \ref{res:tnt_ref}, agree well 
with the experimental data in \cite{Baker1973, Huffington1985, Hokanson1978} 
and \cite{Zhangdz2009}.

\begin{figure}[htbp]
\centering
\subfloat[$t=1.0 \mu$s]
{\includegraphics[width=0.18\textwidth]
{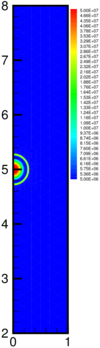}}
\subfloat[$t=4.3 \mu$s]
{\includegraphics[width=0.181\textwidth]
{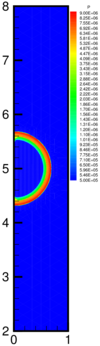}}
\subfloat[$t=3.0 \mu$s]
{\includegraphics[width=0.176\textwidth]
{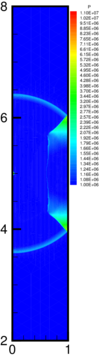}} 
\subfloat[$t=7.0 \mu$s]
{\includegraphics[width=0.181\textwidth]
{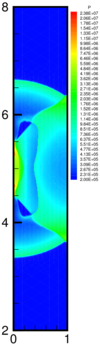}}
\subfloat[$t=11.0 \mu$s]
{\includegraphics[width=0.183\textwidth]
{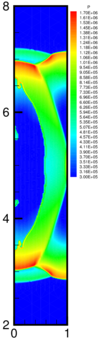}} \\
\caption{Pressure contours for TNT explosion problem.}
\label{res:tnt_ref_cont}
\vspace{15mm}
\subfloat[Peak overpressure]
{\includegraphics[width=0.48\textwidth]
{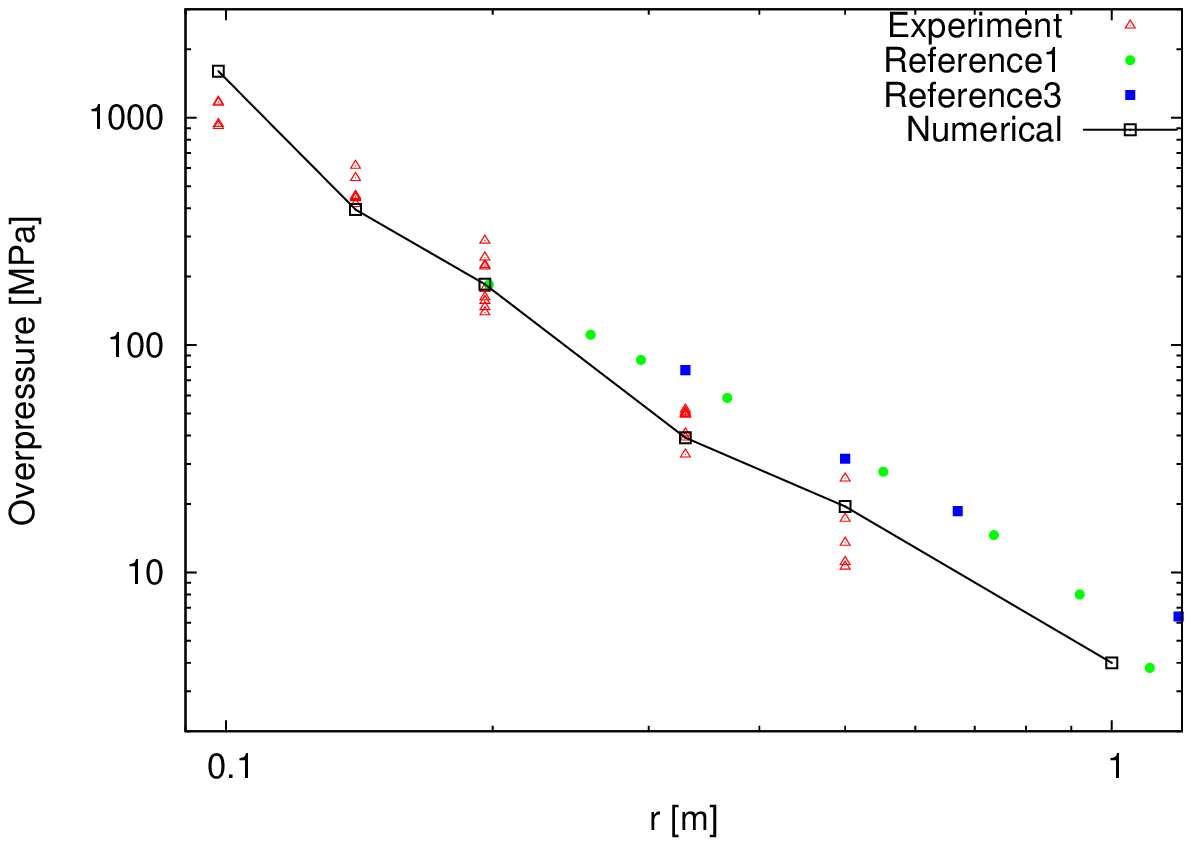}}
\subfloat[Impulse]
{\includegraphics[width=0.48\textwidth]
{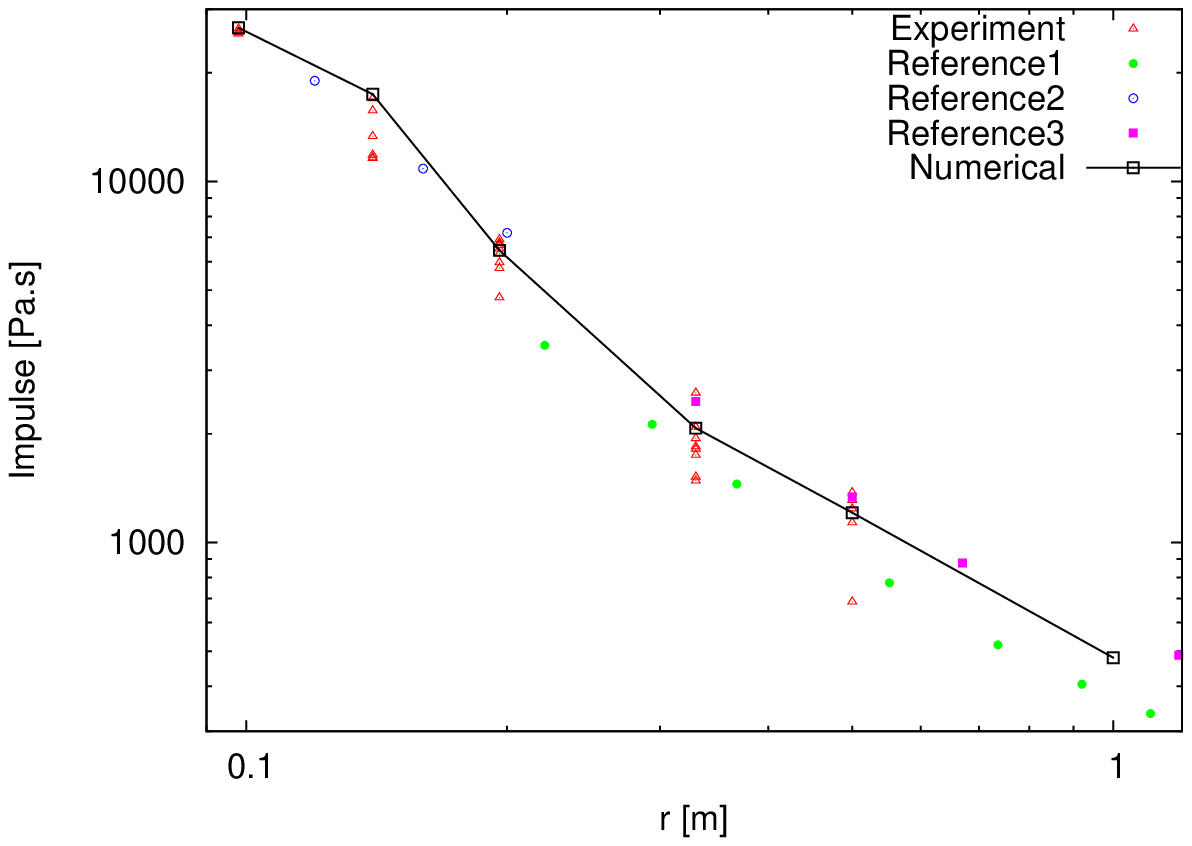}}
\caption{Shock wave parameters for TNT explosion problem.
The reference solution 1 is taken from \cite{Baker1973},
the reference solution 2 is taken from \cite{Huffington1985},
the reference solution 3 is taken from \cite{Hokanson1978}, and 
the experiment solution is taken from \cite{Zhangdz2009}. }
\label{res:tnt_ref}
\end{figure}

\subsubsection{Implosion compression problem}
We consider a two dimensional problem of implosion compression, which is
applied widely in inertial confinement fusion applications 
\cite{lindl1998inertial, jia2014}. The initial shape of the model is a sphere 
containing three distinct mediums, TNT, tungsten and air. The outmost layer 
is the high explosive products of TNT , which is described by the JWL EOS with 
parameters $A_1=\unit[8.524\times10^{11}]{Pa}$, 
$A_2=\unit[1.802\times10^{10}]{Pa}$, $\omega=0.38$, $R_1=4.6$, $R_2=1.3$, 
$\rho_0=\unit[1842]{kg/m^3}$. The intermediate layer is the tungsten, which is 
described by the stiffened gas EOS with parameters $\gamma=4.075$, 
$\rho_0=7.85~\mbox{g/cm}^3$. The innermost layer is the air, which is described 
by the ideal gas with $\gamma=1.4$. All the boundaries are set as outflow 
conditions. The initial values are 

\[
 [\rho, u, p]^\top = \left\{
  \begin{array}{ll}
    [1.29, ~0, ~10^5]^\top, & r < 0.1, \\ [2mm]
    [1.9237\times 10^4, ~0, ~10^5]^\top, & 0.1 \le r \le 0.105, \\ [2mm]
    [1.63\times 10^3, ~0, ~10^5]^\top, & 0.105 < r < 0.12. \\ [2mm]
  \end{array}
 \right.
\]

Fig. \ref{res:icf1} and \ref{res:icf2} show the pressure contours of the whole 
computational domain at different time. Due to the high pressure of the 
explosives at the outmost layer, it produces a strong shock wave inward and 
drives the tungsten and air moving into the center. The shock wave reaches a 
smallest radius at $37.5 ~\mu\mbox{s}$, whose pressure will increase to 
about $1.29 \times 10^{12}~\mbox{Pa}$. Then the shock wave will expand and 
propagate outward with a decreasing shock front. The symmetry of shock waves
and interfaces are kept well during the whole computation, which shows good 
efficiency of our schemes dealing with the highly nonlinear equations of state.

\begin{figure}[htbp]
\centering
\subfloat[$10\times 10^{-6}~\mbox{s}$]
{\includegraphics[width=0.45\textwidth]
{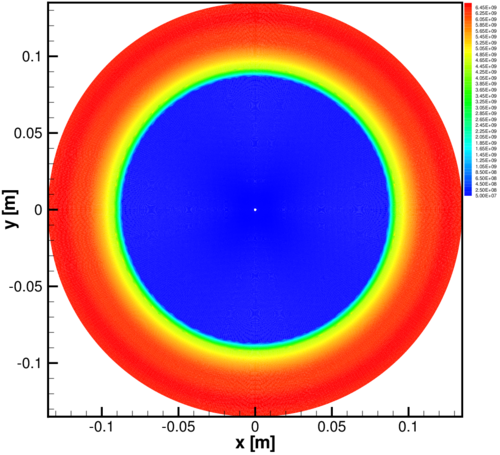}} 
\subfloat[$20\times 10^{-6}~\mbox{s}$]
{\includegraphics[width=0.45\textwidth]
{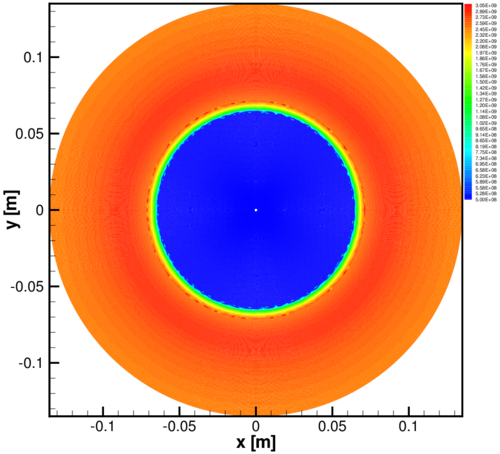}} \\
\vspace{5mm}
\subfloat[$30\times 10^{-6}~\mbox{s}$]
{\includegraphics[width=0.45\textwidth]
{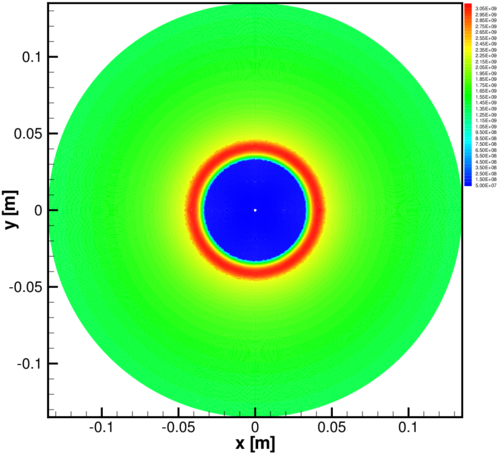}}
\subfloat[$32\times 10^{-6}~\mbox{s}$]
{\includegraphics[width=0.45\textwidth]
{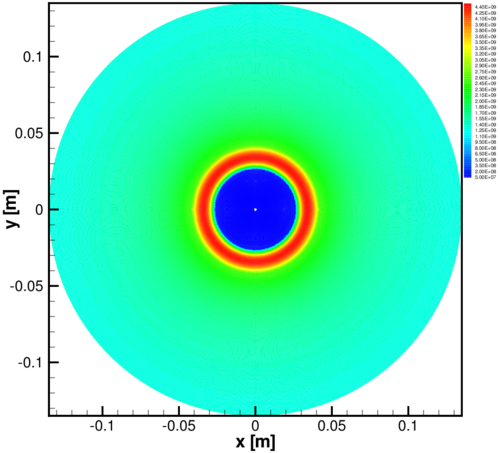}} \\
\vspace{5mm}
\subfloat[$34\times 10^{-6}~\mbox{s}$]
{\includegraphics[width=0.45\textwidth]
{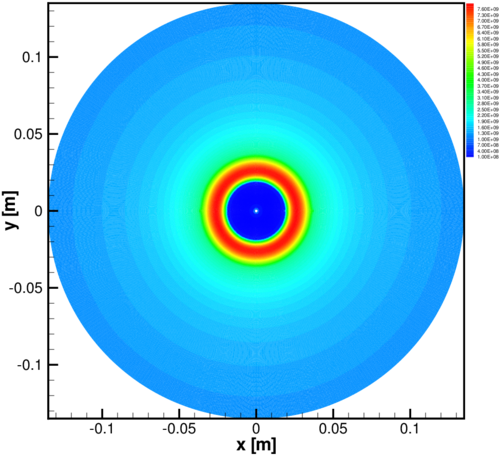}} 
\subfloat[$36\times 10^{-6}~\mbox{s}$]
{\includegraphics[width=0.45\textwidth]
{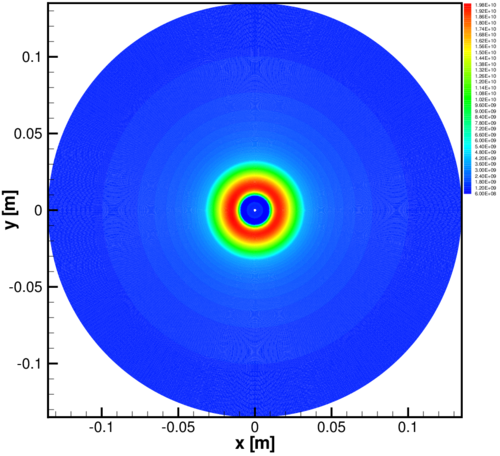}}
\caption{Pressure contours for implosion compression problems. }
\label{res:icf1}
\end{figure}

\begin{figure}[htbp]
\centering
\subfloat[$37\times 10^{-6}~\mbox{s}$]
{\includegraphics[width=0.45\textwidth]
{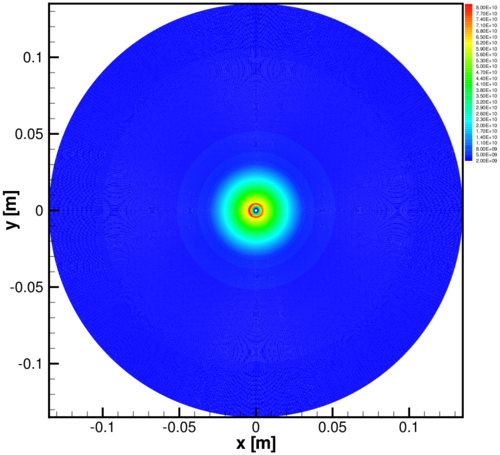}} 
\subfloat[$37.2\times 10^{-6}~\mbox{s}$]
{\includegraphics[width=0.45\textwidth]
{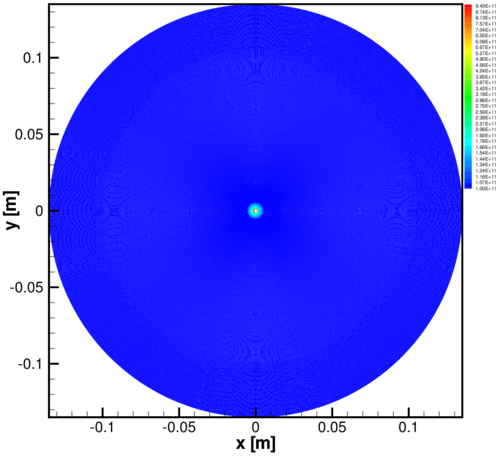}} \\
\vspace{5mm}
\subfloat[$37.5\times 10^{-6}~\mbox{s}$]
{\includegraphics[width=0.45\textwidth]
{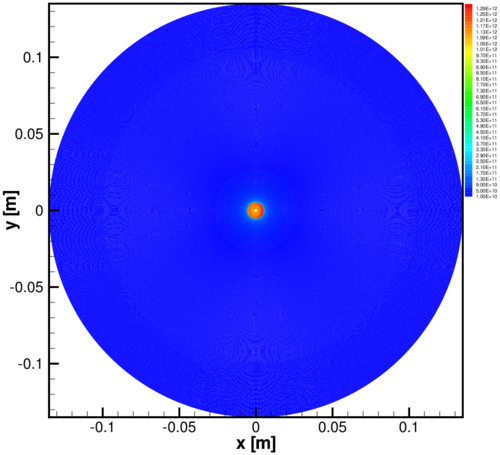}}
\subfloat[$39\times 10^{-6}~\mbox{s}$]
{\includegraphics[width=0.45\textwidth]
{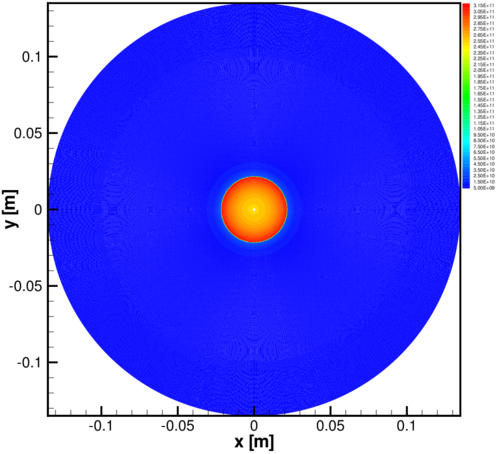}} \\
\vspace{5mm}
\subfloat[$40\times 10^{-6}~\mbox{s}$]
{\includegraphics[width=0.45\textwidth]
{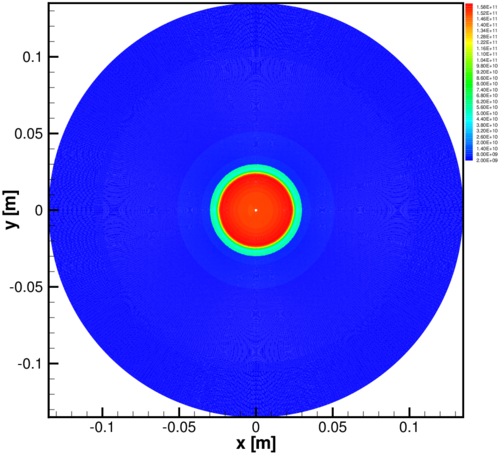}} 
\subfloat[$50\times 10^{-6}~\mbox{s}$]
{\includegraphics[width=0.45\textwidth]
{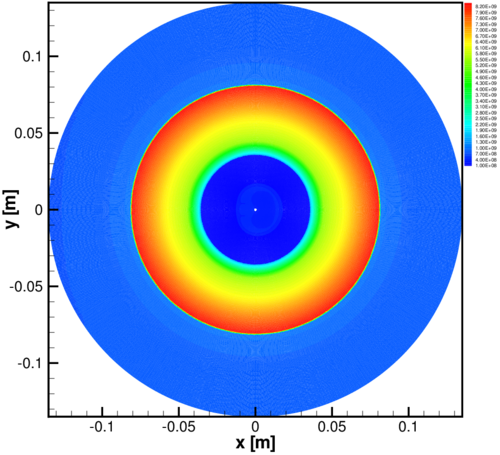}}
\caption{Pressure contours for implosion compression problems. }
\label{res:icf2}
\end{figure}

\subsubsection{High speed impact applications}
In this problem, we simulate a two dimensional high speed impact problem
between three elastoplastic solids. A cylindrical rod made of steel with an
initial radius of $\unit[0.2]{m}$ is given a velocity of $\unit[7000]{m/s}$
and impacts against two layers of static aluminum, shown as Fig. 
\ref{res:impact} (a). Each aluminum has an initial radius of $\unit[0.5]{m}$ 
and a length of $\unit[0.2]{m}$. In fact, the whole problem involves three 
mediums, the steel, the aluminum and the air. The equations of state for the 
hydrostatic pressure component of the steel and aluminum are both taken as the 
stiffened gas EOS, and the deviatoric component are both taken as the perfect 
elastoplasticity. The initial parameters are 
$\rho_0=\unit[7840]{kg/m^3}$, $\gamma=4.075$,
$\mu^{_\mathscr E}=\unit[78.5\times10^9]{Pa}$, $\mu^{_\mathscr P}=0$, 
$Y^{_\mathscr E}=\unit[160\times10^6]{Pa}$ for the steel, 
and $\rho_0=\unit[2790]{kg/m^3}$, $\gamma=2.75$, 
$\mu^{_\mathscr E}=\unit[27.4\times10^9]{Pa}$, $\mu^{_\mathscr P}=0$, 
$Y^{_\mathscr E}=\unit[34\times10^6]{Pa}$ for the aluminum. The air is modeled
by the ideal gas EOS with the following initial parameters:
$\rho_0=\unit[1.29]{kg/m^3}$,
$\gamma=1.4$ and $p_0=\unit[1.013\times10^5]{Pa}$.

\begin{figure}[htbp]
\centering
\hspace{2mm}
\subfloat[Illustration of the impact model]
{\includegraphics[width=0.41\textwidth]
{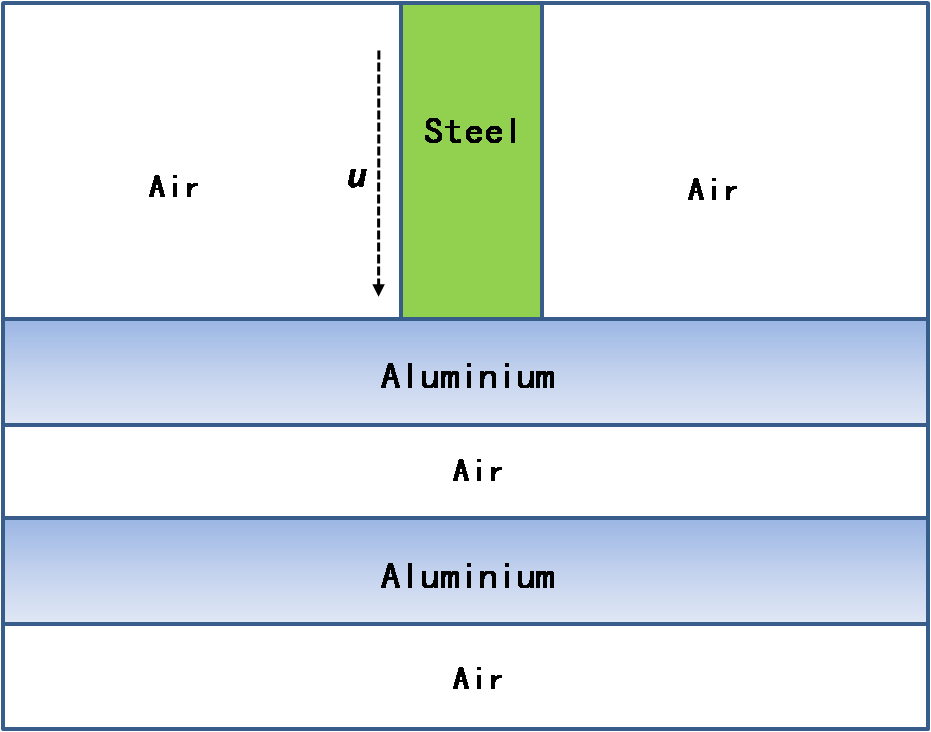}}
\hspace{5mm}
\subfloat[$2.0\times 10^{-7}~\mbox{s}$]
{\includegraphics[width=0.48\textwidth]
{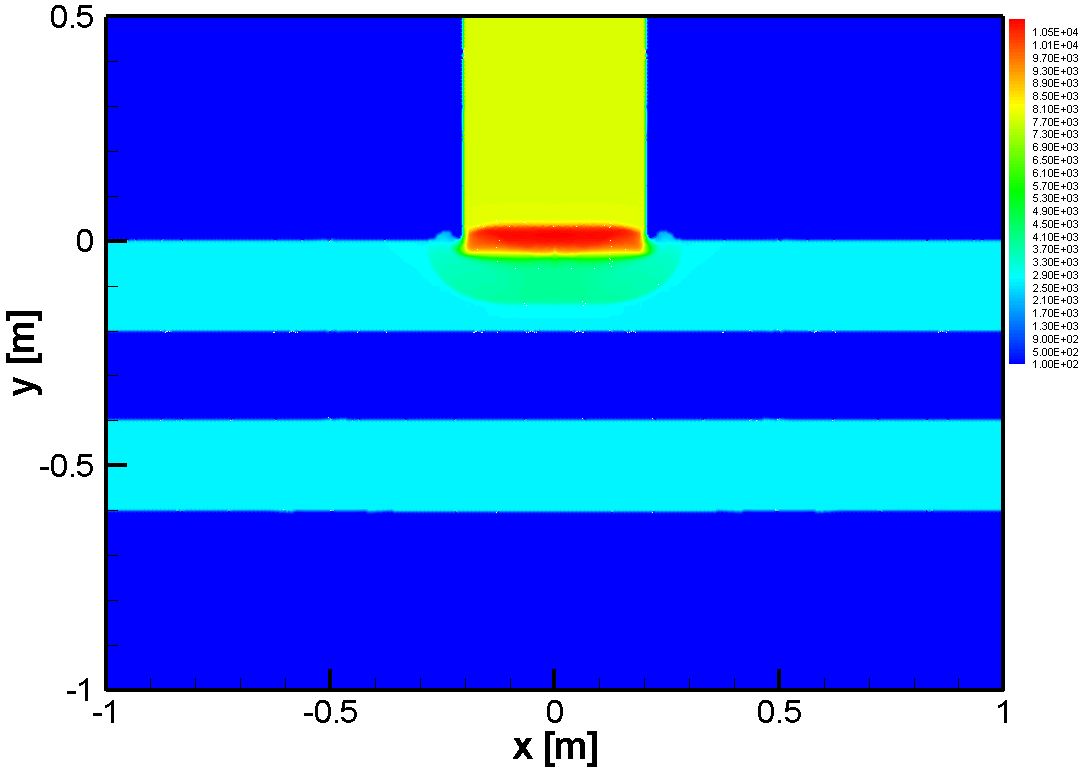}} \\
\vspace{10mm}
\subfloat[$1.0\times 10^{-6}~\mbox{s}$]
{\includegraphics[width=0.48\textwidth]
{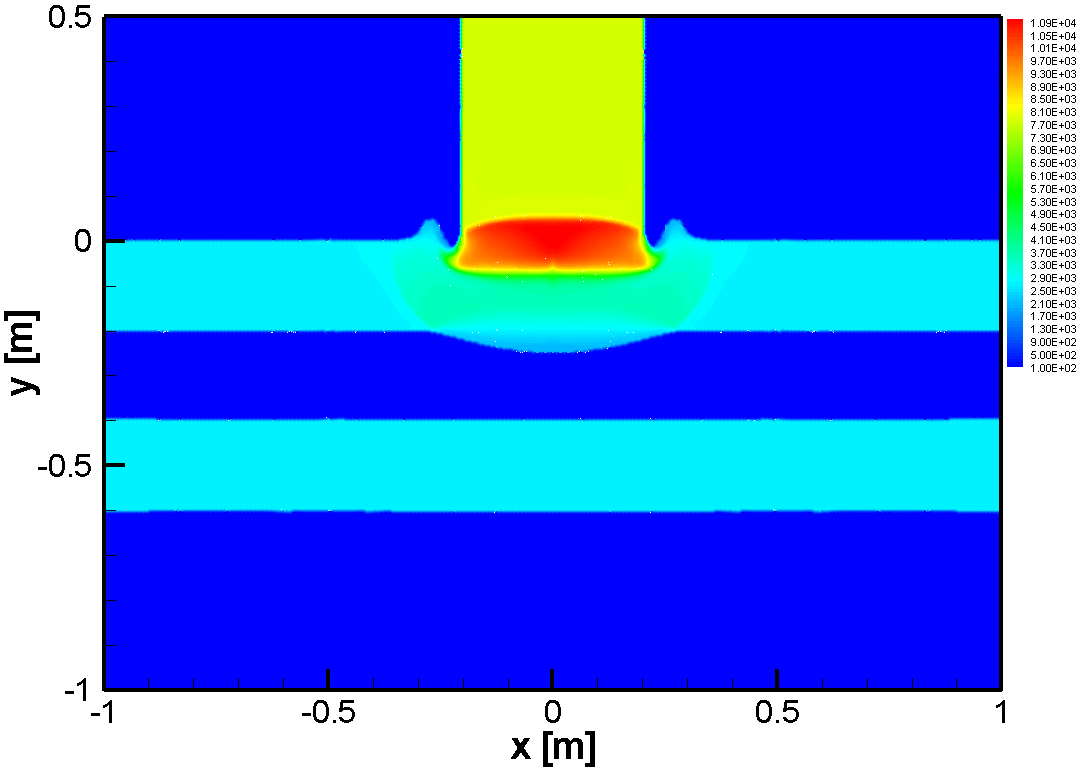}}
\subfloat[$1.8\times 10^{-6}~\mbox{s}$]
{\includegraphics[width=0.48\textwidth]
{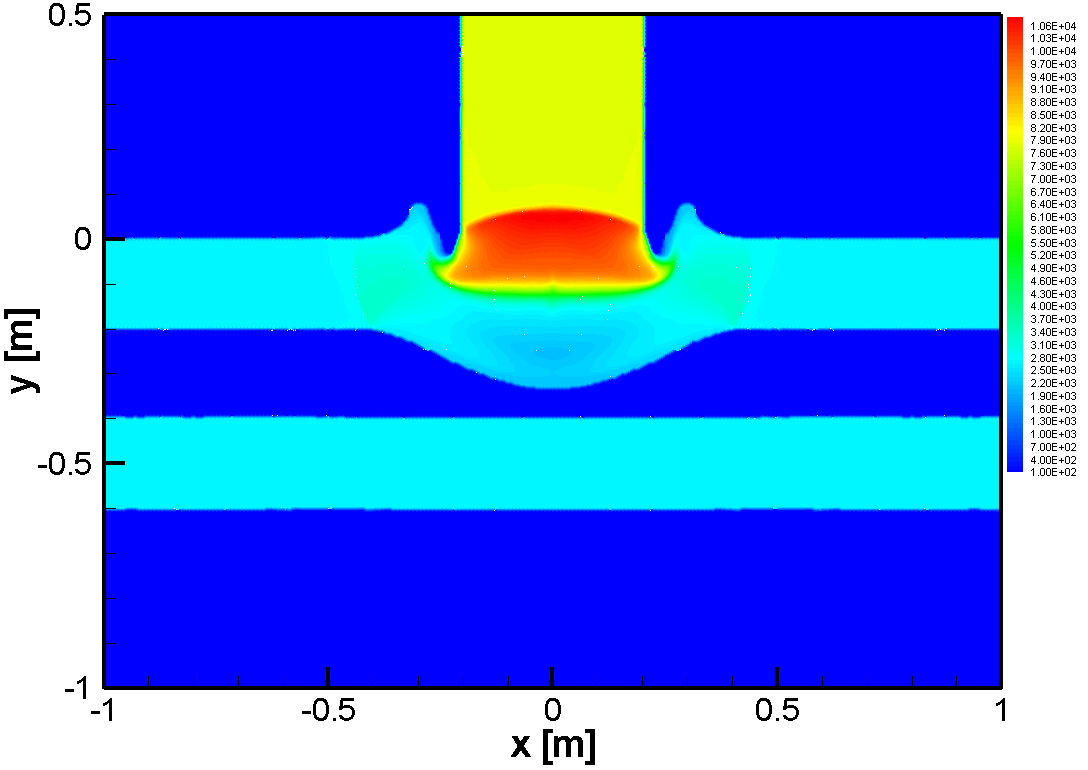}} \\
\vspace{10mm}
\subfloat[$2.5\times 10^{-6}~\mbox{s}$]
{\includegraphics[width=0.48\textwidth]
{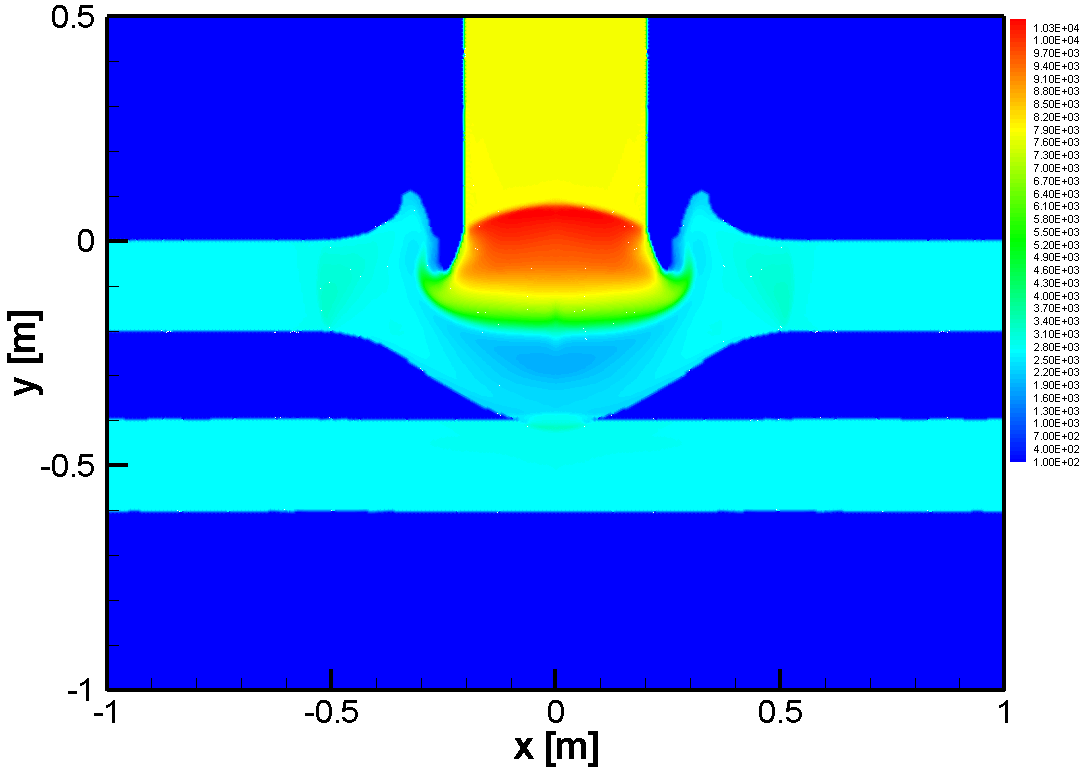}} 
\subfloat[$3.6\times 10^{-6}~\mbox{s}$]
{\includegraphics[width=0.48\textwidth]
{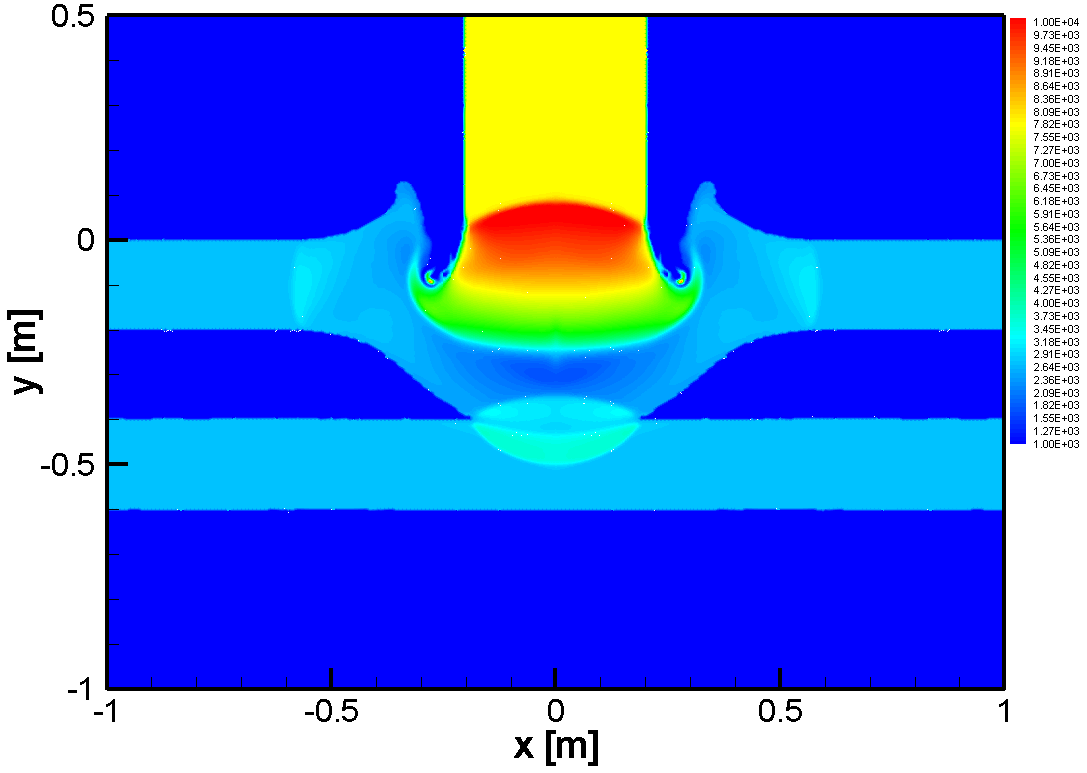}}
\caption{Density contours for high speed impact problems. }
\label{res:impact}
\end{figure}

Fig. \ref{res:impact} (b)--(f) show the density contours of the whole steel and
aluminum at different time. When the steel rod reaches the aluminum, strong
interaction occurs between them. Since the steel has a much higher density and
stiffness, it will lead to the severe deformation and penetration of the
aluminum finally. In the whole calculation, we can see that the interfaces
between each pair of the aluminum, steel and air can be captured sharply,
which show that our numerical scheme can handle the large deformation of
compressible materials and phase interfaces naturally.


\section{Conclusions}\label{sec:conclusion}
We extend the numerical scheme in Guo \textit{et al.} \cite{Guo2016} to the
multi-medium interaction problems that obey a general Mie-Gr{\"u}neisen 
equations of state for the volumetric deformation and hydro-elastoplastic 
constitutive law for the deviatoric deformation. The numerical procedures to 
solve the multi-medium Riemann problem are elaborated. A variety of preliminary 
numerical examples and engineering applications validate our methods. In our 
future work, we will generalize the framework to more complex multiphase 
problems, such as multiphase flow with chemical reaction, heat radiation, and so 
on, which have great initial density and pressure discrepencies and more complex 
physical phenomena.

\section*{Acknowledgments}

The authors appreciate the financial supports provided by the National
Natural Science Foundation of China (Grant No. 91330205, 11421110001,
11421101 and 11325102).

\newpage


\appendix
\renewcommand{\appendixname}{Appendix~\Alph{section}}

\subsubsection*{\bf Ideal gas EOS}

Most of gases can be modeled by the ideal gas law
\begin{equation}
p = (\gamma-1) \rho e,
\label{eq:ideal}
\end{equation}
where $\gamma>1$ is the adiabatic exponent.

\subsubsection*{\bf Stiffened gas EOS}

When considering water under high pressures, the following stiffened 
gas EOS is often used \cite{Rallu2009,Wang2008}:
\begin{equation}
p = (\gamma-1) \rho e-\gamma p_\infty,
\label{eq:stiffened}
\end{equation}
where $\gamma>1$ is the adiabatic exponent, and $p_\infty$ is a
constant. 

\subsubsection*{\bf Murnagham EOS}

Murnagham EOS is widely used in models of solid materials
\begin{equation}
p = \dfrac{K}{\gamma} \left[\left(  
 \dfrac{\rho}{\rho_0} \right)^{\gamma} -1 \right ] + p_0.
\label{eq:murnagham}
\end{equation}
For the steel we adopt the following values
$\rho_0=\unit[7800]{kg/m^3}$, $p_0=\unit[1.0]{\times 10^5
Pa}$, $K=\unit[2.225\times 10^{11}]{Pa}$ and $\gamma=3.7$
\cite{Tang1999, Liu2008}.

\subsubsection*{\bf Polynomial EOS}

The polynomial EOS \cite{Jha2014} can be used to model various
materials
\begin{equation}
p = 
\begin{cases}
  A_1\mu + A_2\mu^2 + A_3\mu^3 + (B_0+B_1\mu)\rho_0 e, & 
  \mu>0, \\
  T_1 \mu  + T_2\mu^2 + B_0 \rho_0e,  &\mu \le 0,
\end{cases}
\label{eq:poly}
\end{equation}
where $\mu = {\rho}/{\rho_0} -1$ and $A_1,A_2,A_3,B_0,B_1,
T_1,T_2,\rho_0$ are positive constants. In this paper, we take an
alternative formulation in the tension branch \cite{Autodyn2003},
where $p=T_1 \mu  + T_2\mu^2 + (B_0+B_1\mu)\rho_0e$ for $\mu\le0$, to
ensure the continuity of the speed of sound at $\mu=0$. Such a formulation  
avoids the occurance of anomalous waves in the Riemann problem, which does not 
exist in real physics. When $B_1\le B_0\le B_1+2$ and $T_1\ge 2T_2$, the 
polynomial EOS satisfies the conditions \textbf{(C1)} and \textbf{(C3)}. In 
addition, if the density $\rho\ge {B_0\rho_0}/{(B_1+2)}$, then the polynomial 
EOS also satisfies the condition \textbf{(C2)}.

\subsubsection*{\bf JWL EOS}

Various detonation products of high explosives can be characterized
by the JWL EOS \cite{Smith1999}
\begin{equation}
p = A_1\left(1-\frac{\omega\rho}{R_1\rho_0}\right)
 \exp\left(-\frac{R_1\rho_0}{\rho}\right) +
 A_2\left(1-\frac{\omega\rho}{R_2\rho_0}\right) \exp\left(-\frac{
 R_2\rho_0}{\rho}\right) + \omega\rho e,
\label{eq:jwl}
\end{equation}
where $A_1,A_2,\omega,R_1,R_2$ and $\rho_0$ are positive constants.
Obviously the JWL EOS \eqref{eq:jwl} satisfies the conditions
\textbf{(C1)}
and \textbf{(C2)}. To enforce the condition \textbf{(C3)} we first
notice that 
\[
\lim_{\rho\rightarrow 0^+} h'(\rho) = 0.
\]
Then it suffices to ensure that $h''(\rho)\ge 0$, which is equivalent
to
the following inequality in terms of $\nu=\rho_0/\rho$:
\[
R_1\nu-2-\omega
\ge G(\nu):=\dfrac{A_2R_2}{A_1R_1}(2+\omega -R_2\nu)
\exp((R_1-R_2)\nu).
\]
A simple calculus shows that the maximum value of the function 
$G(\nu)$ above is given by
\[
\alpha = \dfrac{A_2R_2^2}{A_1R_1(R_1-R_2)}
\exp\left(\dfrac{(2+\omega)(R_1-R_2)-R_2}{R_2}\right).
\]
Therefore a sufficient condition for
\textbf{(C3)} is that the density satisfies
\[
\rho\le \dfrac{R_1}{2+\omega+\alpha}\rho_0,
\]
which is valid for most cases.

\section*{\refname}
\bibliographystyle{unsrt}
\bibliography{reference}

\begin{thebibliography}{10}

\bibitem{Benson1992}
D.J. Benson.
\newblock Computational methods in {L}agrangian and {E}ulerian hydrocodes.
\newblock {\em Computer Methods in Applied Mechanics and Engineering},
  99(2-3):235--394, 1992.

\bibitem{Camacho1997adaptive}
G.T. Camacho and M.~Ortiz.
\newblock Adaptive {L}agrangian modelling of ballistic penetration of metallic
  targets.
\newblock {\em Computer Methods in Applied Mechanics and Engineering},
  142(3-4):269--301, 1997.

\bibitem{Bessette2003modeling}
G.C. Bessette, E.B. Becker, L.M. Taylor, and D.L. Littlefield.
\newblock Modeling of impact problems using an h-adaptive, explicit lagrangian
  finite element method in three dimensions.
\newblock {\em Computer Methods in Applied Mechanics and Engineering},
  192(13-14):1649--1679, 2003.

\bibitem{Udaykumar1997}
H.S. Udaykumar, H.C. Kan, W.~Shyy, and R.~Tran-Son-Tay.
\newblock Multiphase dynamics in arbitrary geometries on fixed {C}artesian
  grids.
\newblock {\em Journal of Computational Physics}, 137(2):366--405, 1997.

\bibitem{Tran2004}
L.~Tran and H.S. Udaykumar.
\newblock A particle-level set-based sharp interface {C}artesian grid method
  for impact, penetration, and void collapse.
\newblock {\em Journal of Computational Physics}, 193(2):469--510, 2004.

\bibitem{Udaykumar2003}
H.S. Udaykumar, L.~Tran, D.M. Belk, and K.J. Vanden.
\newblock An {E}ulerian method for computation of multimaterial impact with
  {ENO} shock-capturing and sharp interfaces.
\newblock {\em Journal of Computational Physics}, 186(1):136--177, 2003.

\bibitem{Sambasivan2011}
S.K. Sambasivan and H.S. Udaykumar.
\newblock A sharp interface method for high-speed multi-material flows: strong
  shocks and arbitrary materialpairs.
\newblock {\em International Journal of Computational Fluid Dynamics},
  25(3):139--162, 2011.

\bibitem{Sambasivan2013}
S.~Sambasivan, A.~Kapahi, and H.S. Udaykumar.
\newblock Simulation of high speed impact, penetration and fragmentation
  problems on locally refined {C}artesian grids.
\newblock {\em Journal of Computational Physics}, 235:334--370, 2013.

\bibitem{Kapahi2013}
A.~Kapahi, S.~Sambasivan, and H.S. Udaykumar.
\newblock A three-dimensional sharp interface cartesian grid method for solving
  high speed multi-material impact, penetration and fragmentation problems.
\newblock {\em Journal of Computational Physics}, 241:308--332, 2013.

\bibitem{Wang2009}
J.T. Wang, K.X. Liu, and D.L. Zhang.
\newblock An improved {CE/SE} scheme for multi-material elastic--plastic flows
  and its applications.
\newblock {\em Computers \& Fluids}, 38(3):544--551, 2009.

\bibitem{Wang2007}
G.~Wang, D.L. Zhang, and K.X. Liu.
\newblock An improved {CE/SE} scheme and its application to detonation
  propagation.
\newblock {\em Chinese Physics Letters}, 24(12):3563, 2007.

\bibitem{Wang2010}
G.~Wang, D.L. Zhang, K.X. Liu, and J.T. Wang.
\newblock An improved {CE/SE} scheme for numerical simulation of gaseous and
  two-phase detonations.
\newblock {\em Computers \& Fluids}, 39(1):168--177, 2010.

\bibitem{Chen2012}
Q.Y. Chen and K.X. Liu.
\newblock A high-resolution {E}ulerian method for numerical simulation of
  shaped charge jet including solid--fluid coexistence and interaction.
\newblock {\em Computers \& Fluids}, 56:92--101, 2012.

\bibitem{Chen2010}
Q.Y. Chen, J.T. Wang, and K.X. Liu.
\newblock Improved {CE/SE} scheme with particle level set method for numerical
  simulation of spall fracture due to high-velocity impact.
\newblock {\em Journal of Computational Physics}, 229(19):7503--7519, 2010.

\bibitem{Hua2011}
S.~Hua, K.X. Liu, and D.L. Zhang.
\newblock Three-dimensional simulation of detonation propagation in a
  rectangular duct by an improved {CE/SE} scheme.
\newblock {\em Chinese Physics Letters}, 28(12):124705, 2011.

\bibitem{Mehmandoust2009}
B.~Mehmandoust and A.R. Pishevar.
\newblock An {E}ulerian particle level set method for compressible deforming
  solids with arbitrary {EOS}.
\newblock {\em International Journal for Numerical Methods in Engineering},
  79(10):1175--1202, 2009.

\bibitem{Sijoy2015}
C.D. Sijoy and S.~Chaturved.
\newblock An {E}ulerian multi-material scheme for elastic--plastic impact and
  penetration problems involving large material deformations.
\newblock {\em European Journal of Mechanics-B/Fluids}, 53:85--100, 2015.

\bibitem{Abgrall1996}
R.~Abgrall.
\newblock How to prevent pressure oscillations in multicomponent flow
  calculations: a quasi conservative approach.
\newblock {\em Journal of Computational Physics}, 125(1):150--160, 1996.

\bibitem{Abgrall2001}
R.~Abgrall and S.~Karni.
\newblock Computations of compressible multifluids.
\newblock {\em Journal of Computational Physics}, 169(2):594--623, 2001.

\bibitem{Saurel1999}
R.~Saurel and R.~Abgrall.
\newblock A simple method for compressible multifluid flows.
\newblock {\em SIAM Journal on Scientific Computing}, 21(3):1115--1145, 1999.

\bibitem{Saurel2009}
R.~Saurel, F.~Petitpas, and R.A. Berry.
\newblock Simple and efficient relaxation methods for interfaces separating
  compressible fluids, cavitating flows and shocks in multiphase mixtures.
\newblock {\em Journal of Computational Physics}, 228(5):1678--1712, 2009.

\bibitem{Petitpas2009}
F.~Petitpas, J.~Massoni, and R.~Saurel.
\newblock Diffuse interface model for high speed cavitating underwater systems.
\newblock {\em International Journal of Multiphase Flow}, 35(8):747--759, 2009.

\bibitem{Ansari2013}
M.R. Ansari and A.~Daramizadeh.
\newblock Numerical simulation of compressible two-phase flow using a diffuse
  interface method.
\newblock {\em International Journal of Heat and Fluid Flow}, 42(8):209--223,
  2013.

\bibitem{Scardovelli1999}
R.~Scardovelli and S.~Zaleski.
\newblock Direct numerical simulation of free-surface and interfacial flow.
\newblock {\em Annual Review of Fluid Mechanics}, 31(1):567--603, 1999.

\bibitem{Noh1976}
W.F. Noh and P.~Woodward.
\newblock {SLIC} (simple line interface calculation).
\newblock In {\em Proceedings of the Fifth International Conference on
  Numerical Methods in Fluid Dynamics}, pages 330--340. Springer, 1976.

\bibitem{Sethian2001}
J.A. Sethian.
\newblock {E}volution, implementation, and application of level set and fast
  marching methods for advancing fronts.
\newblock {\em Journal of Computational Physics}, 169(2):503--555, 2001.

\bibitem{Sussman1994}
M.~Sussman, P.~Smereka, and S.~Osher.
\newblock A level set approach for computing solutions to incompressible
  two-phase flow.
\newblock {\em Journal of Computational Physics}, 114(1):146--159, 1994.

\bibitem{Ahn2007}
H.T. Ahn and M.~Shashkov.
\newblock Multi-material interface reconstruction on generalized polyhedral
  meshes.
\newblock {\em Journal of Computational Physics}, 226(2):2096--2132, 2007.

\bibitem{Dyadechko2008}
V.~Dyadechko and M.~Shashkov.
\newblock Reconstruction of multi-material interfaces from moment data.
\newblock {\em Journal of Computational Physics}, 227(11):5361--5384, 2008.

\bibitem{Anbarlooei2009}
H.R. Anbarlooei and K.~Mazaheri.
\newblock Moment of fluid interface reconstruction method in multi-material
  arbitrary {L}agrangian {E}ulerian ({MMALE}) algorithms.
\newblock {\em Computer Methods In Applied Mechanics And Engineering},
  198(47):3782--3794, 2009.

\bibitem{Glimm1998}
J.~Glimm, J.W. Grove, and X.L. Li.
\newblock Three-dimensional front tracking.
\newblock {\em SIAM Journal on Scientific Computing}, 19(3):1703--727, 1998.

\bibitem{Tryggvason2001}
G.~Tryggvason, B.~Bunner, and A.~Esmaeeli.
\newblock A front-tracking method for the computations of multiphase flow.
\newblock {\em Journal of Computational Physics}, 169(2):708--759, 2001.

\bibitem{Yadav1982converging}
H.S. Yadav and V.P. Singh.
\newblock Converging shock waves in metals.
\newblock {\em Pramana}, 18(4):331--338, 1982.

\bibitem{Shyue2001}
K.M. Shyue.
\newblock A fluid-mixture type algorithm for compressible multicomponent flow
  with {M}ie--{G}r{\"u}neisen equation of state.
\newblock {\em Journal of Computational Physics}, 171(2):678--707, 2001.

\bibitem{Arienti2004}
M.~Arienti, E.~Morano, and J.E. Shepherd.
\newblock Shock and detonation modeling with the {M}ie-{G}r{\"u}neisen equation
  of state.
\newblock Technical report, California Institute of Technology, 2004.

\bibitem{Lee2013}
B.J. Lee, E.F. Toro, C.E. Castro, and N.Nikiforakis.
\newblock Adaptive {O}sher-type scheme for the {E}uler equations with highly
  nonlinear equations of state.
\newblock {\em Journal of Computational Physics}, 246:165--183, 2013.

\bibitem{Banks2010}
J.W. Banks.
\newblock On exact conservation for the {E}uler equations with complex
  equations of state.
\newblock {\em Communications in Computational Physics}, 8(5):995, 2010.

\bibitem{Kamm2015}
J.R. Kamm.
\newblock Solution of the 1{D} {R}iemann problem with a general {EOS} in
  {E}xact{P}ack.
\newblock In {\em 4th ASME Conference on Verification and Validation of
  Simulations, Las Vegas, NV}, 2015.

\bibitem{Kaboudian2014}
A.~Kaboudian and B.C. Khoo.
\newblock The ghost solid method for the elastic solid--solid interface.
\newblock {\em Journal of Computational Physics}, 257:102--125, 2014.

\bibitem{Xiao1996}
L.~Xiao.
\newblock Numerical computation of stress waves in solids.
\newblock {\em Akademie Verlag Gmbh, Berlin}, 1996.

\bibitem{Tang1999}
H.S. Tang and F.~Sotiropoulos.
\newblock A second-order {G}odunov method for wave problems in coupled
  solid--water--gas systems.
\newblock {\em Journal of Computational Physics}, 151(2):790--815, 1999.

\bibitem{Abouziarov2000}
M.~Abouziarov, V.G. Bazhenov, V.~Kotov, A.V. Kochetkov, S.V. Krylov, and V.R.
  Fel'dgun.
\newblock A {G}odunov-type method in dynamics of elastoplastic media.
\newblock {\em Zhurnal Vychislitel'noi Matematiki i Matematicheskoi Fiziki},
  40(6):940--953, 2000.

\bibitem{Bazhenov2002}
V.G. Bazhenov and V.L. Kotov.
\newblock Modification of {G}odunov's numerical scheme for solving problems of
  pulsed loading of soft soils.
\newblock {\em Journal of Applied Mechanics and Technical Physics},
  43(4):603--611, 2002.

\bibitem{Menshov2014}
I.S. Menshov, A.V. Mischenko, and A.A. Serejkin.
\newblock Numerical modeling of elastoplastic flows by the {G}odunov method on
  moving {E}ulerian grids.
\newblock {\em Mathematical Models and Computer Simulations}, 6(2):127--141,
  2014.

\bibitem{Liu2008}
T.G. Liu, W.F. Xie, and B.C. Khoo.
\newblock The modified ghost fluid method for coupling of fluid and structure
  constituted with hydro-elasto-plastic equation of state.
\newblock {\em SIAM Journal on Scientific Computing}, 30(3):1105--1130, 2008.

\bibitem{Liu2011}
T.G. Liu, A.W. Chowdhury, and B.C. Khoo.
\newblock The modified ghost fluid method applied to fluid-elastic structure
  interaction.
\newblock {\em Advances in Applied Mathematics and Mechanics}, 3(05):611--632,
  2011.

\bibitem{Feng2017}
Z.W. Feng, A.~Kaboudian, J.L. Rong, and B.C. Khoo.
\newblock The simulation of compressible multi-fluid multi-solid interactions
  using the modified ghost method.
\newblock {\em Computers \& Fluids}, 2017.

\bibitem{Gao2017}
S.~Gao and T.G. Liu.
\newblock 1{D} exact elastic-perfectly plastic solid {R}iemann solver and its
  multi-material application.
\newblock {\em Advances in Applied Mathematics and Mechanics}, 9(3):621--650,
  2017.

\bibitem{Gao2018}
S.~Gao, T.G. Liu, and C.B. Yao.
\newblock A complete list of exact solutions for one-dimensional
  elastic-perfectly plastic solid {R}iemann problem without vacuum.
\newblock {\em Communications in Nonlinear Science and Numerical Simulation},
  63(2):205--227, 2018.

\bibitem{Gavrilyuk2008}
S.L. Gavrilyuk, N.~Favrie, and R.~Saurel.
\newblock Modelling wave dynamics of compressible elastic materials.
\newblock {\em Journal of Computational Physics}, 227(5):2941--2969, 2008.

\bibitem{Despres2007}
B.~Despres.
\newblock A geometrical approach to nonconservative shocks and elastoplastic
  shocks.
\newblock {\em Archive for Rational Mechanics and Analysis}, 186(2):275--308,
  2007.

\bibitem{Guo2016}
Y.H. Guo, R.~Li, and C.B. Yao.
\newblock A numerical method on {E}ulerian grids for two-phase compressible
  flow.
\newblock {\em Advances in Applied Mathematics and Mechanics}, 8(2):187--212,
  2016.

\bibitem{Lichen2018}
L.~Chen, R.~Li, and C.B. Yao.
\newblock An approximate solver for multi-medium {R}iemann problem with
  {M}ie-{G}r{\"u}neisen equations of state.
\newblock {\em Research in the Mathematical Sciences}, 5(3):31--59, 2018.

\bibitem{Dembo1982}
R.S. Dembo, S.C. Eisenstat, and T.~Steihaug.
\newblock Inexact {N}ewton methods.
\newblock {\em SIAM Journal on Numerical Analysis}, 19(2):400--408, 1982.

\bibitem{Heuze2012}
O.~Heuz{\'e}.
\newblock General form of the {M}ie--{G}r{\"u}neisen equation of state.
\newblock {\em Comptes Rendus Mecanique}, 340(10):679--687, 2012.

\bibitem{Trangenstein1991}
J.A. Trangenstein and P.~Colella.
\newblock A higher-order {G}odunov method for modeling finite deformation in
  elastic-plastic solids.
\newblock {\em Communications on Pure and Applied Mathematics}, 44(1):41--100,
  1991.

\bibitem{Godunov1976}
S.K. Godunov, A.V. Zabrodin, M.I. Ivanov, A.N. Kraiko, and G.P. Prokopov.
\newblock Numerical solution of multidimensional problems of gas dynamics.
\newblock {\em Moscow Izdatel Nauka}, 1, 1976.

\bibitem{Fehlberg1970}
E.~Fehlberg.
\newblock Klassische {R}unge-{K}utta-{F}ormeln vierter und niedrigerer
  {O}rdnung mit {S}chrittweiten-{K}ontrolle und ihre {A}nwendung auf
  {W}{\"a}ermeleitungsprobleme.
\newblock {\em Computing}, 6(1-2):61--71, 1970.

\bibitem{Di2007}
Y.~Di, R.~Li, T.~Tang, and P.~Zhang.
\newblock Level set calculations for incompressible two-phase flows on a
  dynamically adaptive grid.
\newblock {\em Journal of Scientific Computing}, 31(1):75--98, 2007.

\bibitem{Toro2008}
E.F. Toro.
\newblock {\em {R}iemann Solver and Numerical Methods for Fluid Dynamics}.
\newblock Springer, 2008.

\bibitem{Smith1999}
R.W. Smith.
\newblock {AUSM} ({ALE}): a geometrically conservative arbitrary
  {L}agrangian--{E}ulerian flux splitting scheme.
\newblock {\em Journal of Computational Physics}, 150(1):268--286, 1999.

\bibitem{Jha2014}
N.~Jha and B.S.K. Kumar.
\newblock Under water explosion pressure prediction and validationa using
  {ANSYS/AUTODYN}.
\newblock {\em International Journal of Science and Research}, 3:1162--1165,
  2014.

\bibitem{Haas1987interaction}
J.F. Haas and B.~Sturtevant.
\newblock Interaction of weak shock waves with cylindrical and spherical gas
  inhomogeneities.
\newblock {\em Journal of Fluid Mechanics}, 181:41--76, 1987.

\bibitem{ullah2013towards}
M.A. Ullah, W.~Gao, and D.~Mao.
\newblock Towards front-tracking based on conservation in two space dimensions
  {III}, tracking interfaces.
\newblock {\em Journal of Computational Physics}, 242:268--303, 2013.

\bibitem{quirk1996dynamics}
J.J. Quirk and S.~Karni.
\newblock On the dynamics of a shock--bubble interaction.
\newblock {\em Journal of Fluid Mechanics}, 318:129--163, 1996.

\bibitem{Baker1973}
W.E. Baker.
\newblock {\em Explosions in Air}.
\newblock University of Texas Press, 1973.

\bibitem{Huffington1985}
N.J.~Huffington Jr and W.O. Ewing.
\newblock {\em Reflected impulse near spherical charges}.
\newblock 1985.

\bibitem{Hokanson1978}
J.C. Hokanson, E.D. Esparza, and A.B. Wenzel.
\newblock {\em Blast Effects of Simultaneous Multiple-Charge Detonations}.
\newblock 1978.

\bibitem{Zhangdz2009}
D.Z. Zhang, Y.~Li, and D.W. Wang.
\newblock Experiment investigations on normal reflected blast wave near the
  sperical explosive (in chinese).
\newblock {\em Acta Armamentarii}, 12:1663--1667, 2009.

\bibitem{lindl1998inertial}
J.D. Lindl.
\newblock {\em Inertial confinement fusion: the quest for ignition and energy
  gain using indirect drive}.
\newblock American Institute of Physics, 1998.

\bibitem{jia2014}
Z.P. Jia, H.D. Zhang, and X.J. Yu.
\newblock {\em Numerical Methods of Multi-material Simulations}.
\newblock Peking: Science Press, 2014.

\bibitem{Rallu2009}
A.S.D. Rallu.
\newblock {\em A Multiphase Fluid-Structure Computational Framework for
  Underwater Implosion Problems}.
\newblock PhD thesis, Stanford University, 2009.

\bibitem{Wang2008}
C.~Wang, H.~Tang, and T.~Liu.
\newblock An adaptive ghost fluid finite volume method for compressible
  gas--water simulations.
\newblock {\em Journal of Computational Physics}, 227(12):6385--6409, 2008.

\bibitem{Autodyn2003}
N.N. Autodyn.
\newblock {\em Autodyn Theory Manual}, 2003.

\end{thebibliography}
\end{document}